\documentclass{amsart}
\usepackage{amsmath}
\usepackage{amsfonts}
\usepackage{amssymb}
\usepackage{epsfig}
\newtheorem{theorem}{Theorem}[section]

\newtheorem{lm}[theorem]{Lemma}

\newtheorem{cor}[theorem]{Corollary}

\newtheorem{rem}[theorem]{Remark}
\newtheorem{pr}[theorem]{Proposition}
\newtheorem{example}[theorem]{Example}

\usepackage{color}

\begin{document}

\title{Group superschemes}
\author{A.Masuoka}
\address{Institute of Mathematics, University of Tsukuba, Ibaraki 305-8571, Japan}
\email{akira@math.tsukuba.ac.jp}
\author{A. N. Zubkov}
\address{Department of Mathematical Science, UAEU, Al Ain, United Arabic Emirates; Sobolev Institute of Mathematics, Omsk Branch, Pevtzova 13, 644043 Omsk, Russian Federation}
\email{a.zubkov@yahoo.com}

\begin{abstract}
We develop a general theory of algebraic group superschemes, which are
not necessarily affine.
Our key result is a category equivalence between those group superschemes
and Harish-Chandra pairs, which generalizes the result known for affine
algebraic group superschemes. Then we present
the applications,
including the Barsotti-Chevalley Theorem in the super context, and
an explicit construction of the quotient superscheme
$\mathbb{G}/\mathbb{H}$ of an algebraic group superscheme $\mathbb{G}$ by
a group super-subscheme $\mathbb{H}$.
\end{abstract}
\maketitle

\section*{Introduction}

The purpose of this paper is to generalize a description of algebraic supergroups, which uses  Harish-Chandra pairs, to the category of locally algebraic group superschemes. The first results of this kind were developed in  \cite{mas}, where it has been shown, using Hopf superalgebra technique, that any algebraic supergroup is the product of its largest even super-subgroup and a purely-odd super-subscheme. An explicit construction of such a decomposition was first proposed in the article \cite{gav}, and then significantly revised and generalized in the papers \cite{masshib, masshib2}.
The main result of these papers can be formulated as a fundamental equivalence between the category of (affine) algebraic supergroups and specific Harish-Chandra pairs. 

In the broadest sense, a Harish-Chandra pair is the couple $(\mathsf{G}, \mathsf{V})$, where $\mathsf{G}$ is a group scheme, and $\mathsf{V}$ is a finite-dimensional $\mathsf{G}$-module, equipped with a bilinear symmetric $\mathsf{G}$-equivariant map from $\mathsf{V}\times\mathsf{V}$ to the Lie algebra $\mathsf{g}$ of $\mathsf{G}$ (see Section 12.1 for more details). We call a Harish-Chandra pair affine, algebraic, or locally algebraic, provided $\mathsf{G}$ is an affine, algebraic, or locally algebraic group scheme, respectively. 

Throughout this article we freely use the equivalence between the category of geometric superschemes $\mathcal{SV}$ and the category of superschemes $\mathcal{SF}$, regarded as $\Bbbk$-functors. The objects and morphisms of $\mathcal{SF}$ are denoted by $\mathbb{X}, \mathbb{Y}, \ldots$, and
${\bf f}, {\bf g}, \ldots$ respectively. The corresponding, by virtue of this equivalence, objects and morphisms of $\mathcal{SV}$ are denoted by $X, Y, \ldots$, and $f, g, \ldots$, and vice versa.
 
To every affine algebraic group superscheme $\mathbb{G}$, we can associate the affine algebraic Harish-Chandra pair
$(\mathbb{G}_{ev}, \mathfrak{g}_1)$, where $\mathbb{G}_{ev}$ is the largest purely-even group super-subscheme of $\mathbb{G}$ and $\mathfrak{g}$ is the Lie superalgebra of $\mathbb{G}$. The action of $\mathbb{G}_{ev}$ on $\mathfrak{g}_1$ is induced by the adjoint action of $\mathbb{G}$ on $\mathfrak{g}$. The corresponding bilinear map is the restriction of Lie super-bracket to $\mathfrak{g}_1$.
The functor $\mathbb{G}\mapsto (\mathbb{G}_{ev}, \mathfrak{g}_1)$ is called the \emph{Harish-Chandra functor}.

There is a functor that, given an affine algebraic Harish-Chandra pair $(\mathsf{G}, \mathsf{V})$, constructs an algebraic supergroup as a superscheme product $\mathsf{G}\times {\bf E}$, where ${\bf E}$ is a purely-odd superscheme isomorphic to $\mathrm{SSp}(\Lambda(\mathsf{V}^*))$. Moreover, this functor is a quasi-inverse of the Harish-Chandra functor.

The Harish-Chandra functor can be extended to the larger category of locally algebraic group superschemes.
We prove that this functor has a quasi-inverse as well. It means that every locally algebraic group superscheme $\mathbb{G}$ is isomorphic to a superscheme product  $\mathbb{G}_{ev}\times {\bf E}$ as above, where $\mathbb{G}_{ev}$ is regarded as a group scheme. Although the proof mostly follows the papers \cite{masshib, masshib2}, we introduce a new object that plays a crucial role for non-affine group superschemes. We show that every group superscheme $\mathbb{G}$ contains a normal group subfunctor $\mathbb{N}(\mathbb{G})$, which can be regarded as a formal supergroup (see \cite[Remark 3.6]{homastak}). We call $\mathbb{N}(\mathbb{G})$ a \emph{formal neighborhood of the identity}. The superscheme 
$\bf E$ is contained in $\mathbb{N}(\mathbb{G})$. In particular, $\mathbb{G}=\mathbb{G}_{ev}\mathbb{N}(\mathbb{G})$, i.e.,
$\mathbb{N}(\mathbb{G})$ is sufficiently large as a group subfunctor.

Moreover, 
$\mathbb{N}(\mathbb{G})$ is "quasi-affine" in the sense that a complete Hopf superalgebra
$\widehat{\mathcal{O}_e}$ can represent it (see Section 9). Note that $\mathbb{N}(\mathbb{G})$ is "negligible" from the topological point of view. That is, for every open super-subscheme $\mathbb{U}$ of $\mathbb{G}$, we have $\mathbb{U}\mathbb{N}(\mathbb{G})\subseteq\mathbb{U}$. The reason for introducing $\mathbb{N}(\mathbb{G})$ is that $\bf E$ is not a group subfunctor, but every product $xy$ of elements of $\bf E$ can be uniquely expressed as $f z$, where $f\in\mathbb{N}(\mathbb{G})_{ev}$ and $z\in{\bf E}$.  

In our proof, we also exploit a Hopf superalgebra pairing between the
Hopf superalgebra $\widehat{\mathcal{O}_e}$ and the hyperalgebra $\mathrm{hyp}(\mathbb{G})$ of $\mathbb{G}$.
For example, each $\mathbb{N}(\mathbb{G})(R)$ can be identified with the subgroup of $(\mathrm{hyp}(\mathbb{G})\otimes R)^{\times}$ consisting of all group-like elements. 

Finally, to show that $\mathbb{G}$ coincides with $\mathbb{G}_{ev}{\bf E}$, we use the endofunctor $X\mapsto\mathsf{gr}(X)=\mathsf{gr}_{\mathcal{I}_X}(X)$ of the category of geometric superschemes, more precisely, its functor-of-points counterpart. 

The functor $\mathsf{gr}$ is interesting on its own. Indeed, it preserves
immersions and induces an endofunctor of the category of group superschemes.  Moreover, a morphism $f : X\to Y$ of superschemes (locally of finite type) is an isomorphism if and only if $\mathsf{gr}(f)$ is. 
The group superscheme $\mathsf{gr}(\mathbb{G})$ is a semi-direct product of $\mathsf{gr}(\mathbb{G})_{ev}\simeq\mathbb{G}_{ev}$ and a normal purely-odd group super-subscheme $\mathbb{G}_{odd}$. In this case,  $\mathbb{G}_{odd}=\bf E$. Then Proposition \ref{structure of odd complement} implies that the graded companion of the natural embedding  $\mathbb{G}_{ev}{\bf E}\to\mathbb{G}$ is an isomorphism, hence
$\mathbb{G}_{ev}{\bf E}=\mathbb{G}$.

The remainder of the paper is devoted to applications of the fundamental equivalence. First, we prove that a slightly modified (super)version of the well-known Barsotti-Chevalley theorem occurs in the category of algebraic group superschemes. More precisely, a connected algebraic group superscheme $\mathbb{G}$ has
normal group super-subschemes $\mathbb{G}_1\leq\mathbb{G}_2$ such that $\mathbb{G}_1$ is affine, $\mathbb{G}_2/\mathbb{G}_1$ is an abelian group variety, and $\mathbb{G}/\mathbb{G}_2$ is again affine.  

Further, we describe abelian supervarieties and anti-affine algebraic group superschemes. Contrary to the purely-even case, we show that the class of pseudoabelian group superschemes is extensive. Regardless of whether the ground field is perfect or not, we construct pseudoabelian group superschemes, which are neither abelian supervarieties nor even solvable group functors (compare with \cite{milne}, chapter 8).  

In the last section, we prove that for any algebraic group superscheme $\mathbb{G}$ and its group super-subscheme $\mathbb{H}$, the sheaf quotient $\mathbb{G}/\mathbb{H}$ is a superscheme of finite type.
Using \cite[Remark 9.11]{maszub1}, and Theorem 1.1 from \cite{brion}, we reduce the general case to the case when $\mathbb{G}_{ev}=\mathsf{M}\times\mathsf{A}(\mathsf{G})$, where $\mathsf{M}$ is an affine group subscheme, $\mathsf{A}(\mathsf{G})$ is an abelian variety, and $\mathbb{H}_{ev}\leq\mathsf{M}$. This immediately implies that $\mathbb{H}$ is affine.
Note also that $\mathsf{A}(\mathsf{G})$ is a central group subscheme of $\mathbb{G}_{ev}$.

Under these conditions, $\mathbb{G}$ has an open covering by (finitely many) affine $\mathbb{H}$-saturated super-subschemes $\mathsf{U}{\bf E}$, where $\mathsf{U}=\mathsf{U}_{aff}\times\mathsf{U}_{ab}$,
$\mathsf{U}_{aff}$ and $\mathsf{U}_{ab}$ form open affine coverings of $\mathsf{M}$ and $\mathsf{A}(\mathsf{G})$, respectively, and $\mathsf{U}/\mathsf{H}\simeq \mathsf{U}_{aff}/\mathsf{H}\times\mathsf{U}_{ab}$ is an affine superscheme. Modifying the proof of the main theorem from \cite{mastak}, we show that $\mathbb{U}$
is isomorphic to a \emph{homogeneous fiber quotient} $\mathbb{X}\times^{\mathsf{H}}\mathbb{H}$, where $\mathbb{X}$ is an affine superscheme on which $\mathsf{H}$ acts on the right, and such that $(\mathbb{X}\times^{\mathsf{H}}\mathbb{H})/\mathbb{H}$ is affine. The morphism $\mathbb{X}\times^{\mathsf{H}}\mathbb{H}\to\mathbb{U}$ is constructed in a more straightforward way than in \cite{mastak}. To prove that it is an isomorphism, we use the above functor $\mathsf{gr}$. Each $\mathbb{U}/\mathbb{H}$ is affine, and the standard arguments imply that they form an open covering of $\mathbb{G}/\mathbb{H}$.  
 
The article is organized as follows. In the first section, we have collected all necessary notations, definitions, and results on superschemes. Throughout the article, we distinguish between superschemes as functors and geometric superschemes, although the categories of superschemes and geometric superschemes are equivalent to each other by the Comparison Theorem. For example, to introduce the notions of separated and proper morphisms in the second section, the language of geometric superschemes seems to be more natural and convenient. However, the group superschemes are introduced in two incarnations, as group functors and as group objects in the category of geometric superschemes (see Section 4). In the third section we recall the notions of flat and faithfully flat morphisms of superschemes. Lemma \ref{finite type and flatness} plays crucial role in the proof of Theorem \ref{sheaf quotient}, the main result of the final section.

In the fifth section, we prove that every group superscheme has the largest affine quotient, again using the language of geometric superschemes. In the sixth section, for every $\Bbbk$-functor $\mathbb{X}$, we define  its tangent functor at a $\Bbbk$-point and investigate its properties in the case when $\mathbb{X}$ is a superscheme of locally finite type. In the seventh section, the Lie superalgebra of a group superscheme is introduced as functor and a superalgebra. In the eighth section, we develop a fragment of the differential calculus on locally algebraic group superschemes. In the ninth section we introduce the normal group subfunctor $\mathbb{N}(\mathbb{G})$ and discuss how it relates to the Hopf superalgebra $\mathrm{hyp}(\mathbb{G})$. 

The tenth section is devoted to studying the functor $\mathsf{gr}$. We prove that it commutes with direct products, thus it induces an endofunctor of the category of geometric group superschemes. Moreover, a superscheme
$\mathsf{gr}(X)$ is always "split" in the sense that there are two superscheme morphisms $i_X : X_{ev}\to  
\mathsf{gr}(X)$ and $q_X : \mathsf{gr}(X)\to X_{ev}$ such that $q_X i_X=\mathrm{id}_{X_{ev}}$, and $i_X$ induces an isomorphism onto $\mathsf{gr}(X)_{ev}$. 
In particular, if $\mathbb{G}$ is a locally algebraic group superscheme, then $\mathsf{gr}(\mathbb{G})\simeq \mathbb{G}_{ev}\ltimes\mathbb{G}_{odd}$, where $\mathbb{G}_{odd}=\ker {\bf q}_G$ is a purely-odd group superscheme. In the eleventh section, we show that
$\mathbb{G}_{odd}\simeq\mathrm{SSp}(\Lambda(\mathfrak{g}_1^*))$ and for any morphism $\mathbb{G}\to\mathbb{H}$ of group superschemes, the induced morphism $\mathbb{G}_{odd}\to\mathbb{H}_{odd}$ is uniquely defined by the corresponding linear map  $\mathfrak{g}_1\to\mathfrak{h}_1$. 

In the twelfth section, using the techniques developed previously, we prove that the category of locally algebraic group superschemes is equivalent to locally algebraic Harish-Chandra pairs.
The content of the last two sections has already been discussed in detail above.

\section{Superschemes}

For the content of this section, we refer to \cite{ccf, jan, manin, maszub1, zub2}. Throughout this article, $\Bbbk$ is the ground field of odd or zero characteristic. 

\subsection{Superschemes as $\Bbbk$-functors}

Let $\mathsf{SAlg}_{\Bbbk}$ denote the category of super-commutative $\Bbbk$-superalgebras. If $A\in\mathsf{SAlg}_{\Bbbk}$, then let $I_A$ and $\overline{A}$ denote the super-ideal $AA_1$ and the factor-algebra 
$A/I_A$, respectively.

A \emph{$\Bbbk$-functor} $\mathbb{X}$ is a functor from the category $\mathsf{SAlg}_{\Bbbk}$ to the category of sets $\mathsf{Sets}$. The category of $\Bbbk$-functors is denoted by $\mathcal{F}$. The morphisms in $\mathcal{F}$ are denoted by bold letters ${\bf f}, {\bf g}, ..$.

A $\Bbbk$-functor $\mathbb{X}$ is called an \emph{affine superscheme} if it is representable by a superalgebra $A$, that is,
\[\mathbb{X}(B)=\mathrm{Hom}_{\mathsf{SAlg}_{\Bbbk}}(A, B) \text{ for } B\in \mathsf{SAlg}_{\Bbbk}.\]
In this case, $\mathbb{X}$ is denoted by $\mathrm{SSp}(A)$. A \emph{closed affine super-subscheme} of $\mathbb{X}=\mathrm{SSp}(A)$ is defined as
\[\mathbb{V}(I)(B)=\{\phi\in\mathbb{X}(B)\mid \phi(I)=0\} \text{ for } B\in\mathsf{SAlg}_{\Bbbk},\]
where $I$ is a super-ideal of $A$. It is easy to see that $\mathbb{V}(I)\simeq\mathrm{SSp}(A/I)$.

Similarly, a super-ideal $I$ of $A$ defines an \emph{open $\Bbbk$-subfunctor} of $\mathbb{X}=\mathrm{SSp}(A)$ as
\[\mathbb{D}(I)(B)=\{\phi\in\mathbb{X}(B)\mid \phi(I)B=B\} \text{ for } B\in\mathsf{SAlg}_{\Bbbk} .\]

In general, a $\Bbbk$-subfunctor $\mathbb{Y}$ of a $\Bbbk$-functor $\mathbb{X}$ is called \emph{closed} (respectively, \emph{open}) if for any morphism of functors ${\bf f} : \mathrm{SSp}(A)\to\mathbb{X}$ the inverse image ${\bf f}^{-1}(\mathbb{Y})$ is closed (respectively, open) in  $\mathrm{SSp}(A)$.

A collection of open $\Bbbk$-subfunctors $\{\mathbb{X}_i\}_{i\in I}$ is said to be an \emph{open covering} of a $\Bbbk$-functor $\mathbb{X}$, provided $\mathbb{X}(F)=\cup_{i\in I}\mathbb{X}_i(F)$ for any field extension $\Bbbk\subseteq F$. A $\Bbbk$-functor $\mathbb{X}$ is called \emph{local} if for any $\Bbbk$-functor $ \mathbb{Y}$ and any open covering $\{\mathbb{Y}_i\}_{i\in I}$ of $\mathbb{Y}$ the diagram 
\[\mathrm{Mor}_{\mathcal{F}}(\mathbb{Y} , \mathbb{X})\to\prod_{i\in I}\mathrm{Mor}_{\mathcal{F}}(\mathbb{Y}_i , \mathbb{X})\rightrightarrows
\prod_{i, j\in I}\mathrm{Mor}_{\mathcal{F}}(\mathbb{Y}_i\cap\mathbb{Y}_j , \mathbb{X}),\]
\[{\bf f}\mapsto \prod_{i\in I} {\bf f}|_{\mathbb{Y}_i}, \ \prod_{i\in I} {\bf f}_i\mapsto \prod_{i, j\in I} {\bf f}_i|_{\mathbb{Y}_i\cap\mathbb{Y}_j}, \  \prod_{i\in I} {\bf f}_i\mapsto \prod_{i, j\in I} {\bf f}_i|_{\mathbb{Y}_j\cap\mathbb{Y}_i},\]
is exact. 

A local $\Bbbk$-functor $\mathbb{X}$ is called a \emph{superscheme} whenever $\mathbb{X}$ has an open covering by affine super-subschemes. Superschemes form a full subcategory of $\mathcal{F}$, denoted by $\mathcal{SF}$. 

Let $\mathbb{X}$ be a $\Bbbk$-functor. If $\mathbb{Y}$ is a subfunctor of $\mathbb{X}$, such that for any $A\in\mathsf{SAlg}_{\Bbbk}$, the set
$\mathbb{Y}(A)$ contains exactly one element, then $\mathbb{Y}$ is called a \emph{one-point subfunctor} of $\mathbb{X}$. 
\begin{lm}\label{one point subfunctor}
Assume that $\mathbb{X}$ is a superscheme. Then every one-point subfunctor of $\mathbb{X}$ is closed.
\end{lm}
\begin{proof}
Let $\mathbb{Y}$ be a one-point subfunctor of $\mathbb{X}$. For any superalgebra $A$, we have $\mathbb{Y}(A)=\{y_A\}$. Choose a covering of $\mathbb{X}$ by open affine super-subschemes $\mathbb{X}_i\simeq\mathrm{SSp}(A_i)$ for $i\in I$. A verbatim superization of \cite[I.1.7(6)]{jan}, implies that $\mathbb{X}_i$ meets $\mathbb{Y}$ if and only if $y_{\Bbbk}$ belongs to $\mathbb{X}_i(\Bbbk)$. In the latter case,
$y_{\Bbbk}$ corresponds to a superalgebra morphism $\phi : A_i\to \Bbbk$ and $\mathbb{Y}\cap\mathbb{X}_i=\mathbb{V}(\ker\phi)$. Then \cite[Lemma 9.1]{maszub1} concludes the proof. 
\end{proof}

\subsection{Geometric superschemes}

Recall that a \emph{geometric superspace} $X$ consists of a topological space $X^e$ and a sheaf of super-commutative superalgebras $\mathcal{O}_X$ such that all stalks $\mathcal{O}_{X, x}$ for $x\in X^e$ are local superalgebras. A morphism of
superspaces $f : X\to  Y$ is a pair $( f^e, f^*
)$, where $f^e : X^e\to Y^e$ is a morphism of topological spaces
and $f^* : \mathcal{O}_Y\to f^e_{
*}\mathcal{O}_X$ is a morphism of sheaves such that $f^*_
x : \mathcal{O}_{Y , f^e (x)}\to\mathcal{O}_{X,x}$ is a local morphism for
any $x\in X^e$. Let $\mathcal{V}$ denote the category of geometric superspaces.

To simplify the notations, we use $\mathcal{O}(X)$ instead of $\mathcal{O}_X(X^e)$. For any open subsets $V\subseteq U\subseteq X^e$ the image of $f\in\mathcal{O}_X(U)$ in $\mathcal{O}_X(V)$ is denoted by $f|_V$.

Let $X$ be a geometric superspace. If $U$ is an open subset of $X^e$, then $(U, \mathcal{O}_X|_U)$ is again a geometric superspace, called an \emph{open super-subspace} of $X$. In what follows 
$(U, \mathcal{O}_X|_U)$ is denoted by $U$.

Let $R$ be a superalgebra. An \emph{affine superscheme} $\mathrm{SSpec}(R)$ can be defined as follows.
The underlying topological space of $\mathrm{SSpec}(R)$  coincides
with the prime spectrum of $R$, endowed with the Zariski topology. For any open subset $U\subseteq (\mathrm{SSpec}(R))^e$, the super-ring $\mathcal{O}_{\mathrm{SSpec}(R)}(U)$ consists of all locally constant functions $h : U\to\sqcup_{\mathfrak{p}\in U} R_{\mathfrak{p}}$ such that $h(\mathfrak{p})\in R_{\mathfrak{p}} , \mathfrak{p}\in U$. 

A superspace $X$ is called a \emph{(geometric) superscheme} if there is an open covering $X^e =
\cup_{i\in I} U_i$, such that each open super-subspace $U_i$ is isomorphic to an affine superscheme $\mathrm{SSpec}(R_i)$. Superschemes form a full subcategory of $\mathcal{V}$, denoted by $\mathcal{SV}$. 

If $X$ is a superscheme, then every open super-subspaces of $X$ is a superscheme, called an \emph{open super-subscheme}. A superscheme $Z$ is a \emph{closed super-subscheme} of $X$ if there is a
closed embedding $\iota : Z^e\to X^e$ such that the sheaf $\iota_*\mathcal{O}_Z$ is an epimorphic image of the sheaf $\mathcal{O}_X$. 
\begin{example}\label{the largest purely-even geometric super-subscheme}
The sheafification of the presheaf $U\mapsto I_{\mathcal{O}_X(U)}=\mathcal{O}_X(U) \mathcal{O}_X (U)_1$ is a sheaf of $\mathcal{O}_X$-super-ideals, which is denoted by $\mathcal{I}_X$. Then $(X^e, \mathcal{O}_X/\mathcal{I}_X)$ is the largest purely-even closed super-subscheme of $X$, denoted by $X_{ev}$.  Regarded as a geometric scheme, $X_{ev}$ is denoted by $X_{res}$. Moreover, every morphism $f : X\to Y$ induces the morphism $f_{ev} : X_{ev}\to Y_{ev}$ so that $X\mapsto X_{ev}$ is an endofunctor of $\mathcal{SV}$.
\end{example}

\subsection{Comparison Theorem}(see \cite[Theorem 10.3.7]{ccf}, or \cite[Theorem 5.14]{maszub1}, and \cite[I, \S 1, 4.4]{DG})
\begin{theorem}\label{comparison}
The functor $X\mapsto\mathbb{X}$, where $\mathbb{X}(A)=\mathrm{Mor}_{\mathcal{SV}}(\mathrm{SSpec}(A), X)$ for  $A\in\mathsf{SAlg}_{\Bbbk}$, defines an equivalence $\mathcal{SV}\simeq\mathcal{SF}$.
\end{theorem}
If $f : X\to Y$ is a morphism in the category $\mathcal{SV}$, then the corresponding morphism $\mathbb{X}\to \mathbb{Y}$ in the category $\mathcal{SF}$ is denoted by 
$\bf f$ and vice versa. 

\subsection{Immersions}

A morphism $f : X\to Y$ in the category $\mathcal{SV}$ is called an \emph{immersion} if there is an open subset $U$ in $Y^e$ such that $f^e$ induces a homeomorphism from $X^e$ onto the closed subset $f^e(X^e)\subseteq U$ and the induced morphism of sheaves $f^* : \mathcal{O}|_U\to f_* \mathcal{O}_X$ is surjective. The latter is equivalent to the condition that for every $x\in X^e$, the induced superalgebra morphism $\mathcal{O}_{Y, f^e(x)}\to\mathcal{O}_{X, x}$ is surjective.
If $f^e(X^e)=U$, then $f$ is called {\it an open immersion}. Finally, if $U=Y^e$, then $f$ is called {\it a closed immersion}.

The proofs of the following lemmas are standard, and we leave them for the reader.
\begin{lm}\label{locality of immersion}
A morphism $f : X\to Y$ in the category $\mathcal{SV}$ is closed, open or immersion if and only if for any open covering $\{V_i\}_{i\in I}$ of $Y$, each morphism $f|_{f^{-1}(V_i)} : f^{-1}(V_i)\to V_i$ is closed, open or immersion, correspondingly.
\end{lm}
\begin{lm}\label{locality of immersion 2}
Let $f : X\to Y$ be a morphism in the category $\mathcal{SV}$ such that $f^e$ is injective. Then $f$ is closed, or an immersion, respectively if and only if for any finite open covering $\{U_j\}_{j\in J}$ of $X$, each morphism $f|_{U_j} : U_j\to Y$ is closed, or an immersion, respectively. Finally, $f$ is an open immersion if and only if for any open covering $\{U_j\}_{j\in J}$ of $X$, each morphism $f|_{U_j} : U_j\to Y$ is an open immersion.
\end{lm}
\begin{pr}\label{closed (super)subschemes} (cf. \cite[I.l]{milne})
Let $f : Z\to X$ be a morphism of geometric superschemes. Then $f$ is a closed immersion if and only if $\bf f$ is an isomorphism of $\mathbb{Z}$ onto a closed super-subscheme of $\mathbb{X}$.
\end{pr}
\begin{proof}
Let $\{U_j\}_{j\in J}$ be an open covering of $X$. Then $\{Z_j=i^{-1}(U_j)\}_{j\in J}$ is an open covering of $Z$. By \cite[Lemma 5.2(2,3)]{maszub1}, $\{\mathbb{U}_j\}_{j\in J}$ is an open covering of $\mathbb{X}$. Moreover, each $\mathbb{U}_j$ is affine whenever $U_j$ is.

Observe that, for every superalgebra $A$, and every open subset $U\subseteq X^e$, the subset $\mathbb{U}(A)\subseteq \mathbb{X}(A)$ consists of all superscheme morphism $h : \mathrm{SSpec}(A)\to X$ such that $h^e((\mathrm{SSpec}(A))^e)\subseteq U$. Thus ${\bf f}^{-1}(\mathbb{U}_j)=\mathbb{Z}_j$ for any $j\in J$.

Assume now that $f$ is a closed immersion. Then $\bf f$ is an embedding of $\Bbbk$-functors.
 The above remarks, combined with \cite[Lemma 9.1 (1)]{maszub1}, reduce the general case to $X=\mathrm{SSpec}(A)$ for $A\in\mathsf{SAlg}_{\Bbbk}$. There is a super-ideal $I$ of $A$ such that $f$ can be identified with the closed immersion $\mathrm{SSpec}(A/I)\to\mathrm{SSpec}(A)$. Then \cite[Lemma 4.1]{maszub1} implies that $\bf f$ is the canonical isomorphism of $\mathrm{SSp}(A/I)$ onto the closed super-subscheme $\mathbb{V}(I)$ of $\mathrm{SSp}(A)$.

For the reverse statement, we use  \cite[Lemma 9.1 (1)]{maszub1} again, and the general case can be reduced to $X=\mathrm{SSpec}(A)$. There are a super-ideal $I$ of $A$ and an isomorphism ${\bf h} : \mathbb{Z}\to\mathrm{SSp}(A/I)$ such that ${\bf f}=\mathrm{SSp}(\pi){\bf h}$, where $\pi$ is the canonical epimorphism $A\to A/I$. By Comparison Theorem, there is an isomorphism $h : Z\to\mathrm{SSpec}(A/I)$ such that $f=\mathrm{SSpec}(\pi)h$. In other words, for every open affine covering $\{U_j\}_{j\in J}$ of $X$, each morphism $f|_{Z_j} : Z_j\to U_j$ is a closed immersion. Therefore, $f$ is a closed immersion, proving the proposition.
\end{proof}
\begin{example}
The closed immersion $X_{ev}\to X$ corresponds to the embedding $\mathbb{X}_{ev}\to\mathbb{X}$, where
the subfunctor $\mathbb{X}_{ev}$ is defined as 
\[\mathbb{X}_{ev}(A)=\mathbb{X}(\iota)(\mathbb{X}(A_0))\simeq\mathbb{X}(A_0) \text{ for } A\in\mathsf{SAlg}_K ,\]
and $\iota : A_0\to A$ is the natural embedding. In particular, $\mathbb{X}_{ev}$ is a closed super-subscheme of $\mathbb{X}$ (see also \cite[Proposition 9.2]{maszub1}). As above, $\mathbb{X}\to\mathbb{X}_{ev}$ is an endofunctor of the category $\mathcal{SF}$. Note that if $\mathbb{X}$ is not a superscheme, then the above map $\mathbb{X}(\iota)$ is no longer injective. For example, consider the functor $A\mapsto\overline{A}$ for $A\in\mathsf{SAlg}_{\Bbbk}$.   
\end{example}
Let $f : X\to Y$ be an immersion. Assume that $f$ factors through an open super-subscheme $U$ of $Y$.
Let $\mathcal{J}$ denote $\ker (f_*\mathcal{O}_U\to\mathcal{O}_X)$. For any non-negative integer $n$, one can define the \emph{$n$-th neighborhood} of $f$ as a closed super-subscheme of $U$, given by the super-ideal sheaf $\mathcal{J}^{n+1}$, denoted by $Y^n_f$.
\begin{lm}\label{compatibility in general} The definition of $Y^n_f$ does not depend on $U$. 
\end{lm}
\begin{proof}
If $f$ factors through another open super-subscheme $V$ of $Y$, then $f$ factors through $U\cap V$.
Without loss a generality, one can assume that $V\subseteq U$. Then for every $x\in X^e\subseteq V^e$
there is a natural isomorphism $\mathcal{O}_{Y^n_f, x}\simeq\mathcal{O}_{U, x}/\mathcal{J}^{n+1}_x\simeq
\mathcal{O}_{V, x}/(\mathcal{J}|_V)_x^{n+1}$. 
\end{proof}
For example, let $B$ be a local superalgebra with nilpotent maximal super-ideal $\mathfrak{m}$ such that $B/\mathfrak{m}=\Bbbk$. Every superscheme morphism $\mathrm{SSpec}(B)\to X$ is uniquely defined by a $\Bbbk$-point $x\in X^e$ and by a local morphism of superalgebras $\mathcal{O}_x\to B$. In particular, such morphism factors through any open neighborhood of $x$, and it is a closed immersion if and only if $\mathcal{O}_x\to B$ is surjective. In particular, we have a canonical closed immersion 
$\mathrm{SSpec}(\mathcal{O}_x/\mathfrak{m}_x^{n+1})\to X$, which is just the $n$-th neighborhood of the closed immersion $i_x : \mathrm{SSpec}(\mathcal{O}_x/\mathfrak{m}_x)\to X$.

\subsection{Morphisms (locally) of finite type}

Recall that a morphism $f : X\to Y$ of geometric superschemes is said to be \emph{locally of finite type} if there is an open covering of $Y$ by affine super-subschemes $V_i\simeq \mathrm{SSpec}(B_i)$
such that for every $i$, the open super-subscheme $f^{-1}(V_i)$ is covered by open super-subschemes
$U_{ij}\simeq\mathrm{SSpec}(A_{ij})$, where each $A_{ij}$ is a finitely generated $B_i$-superalgebra.

If each $f^{-1}(V_i)$ can be covered by a finite number of $U_{ij}$, then $f$ is said to be a morphism of \emph{finite type} (cf. \cite[II.3]{hart}).
\begin{lm}\label{finitely generated}
Let $\phi : B\to A$ be a superalgebra morphism. Then $A$, regarded as a $B$-superalgebra via $\phi$, is finitely generated if and only if there are $b_1, \ldots , b_s\in B_0$ such that $\sum_{1\leq i\leq s}B_0 b_i=B_0$ and $A_{\phi(b_i)}$ is a finitely generated $B_{b_i}$-superalgebra for each $1\leq i\leq s$. Symmetrically, if there are $a_1, \ldots, a_s\in A_0$ such that $\sum_{1\leq i\leq s}A_0a_i=A_0$ and $A_{a_i}$ is a finitely generated $B$-superalgebra for each $1\leq i\leq s$, then $A$ is a finitely generated $B$-superalgebra.  
\end{lm}
\begin{proof}
There is a super-subalgebra $C$ of $A$, finitely generated over $B[\phi(b_1), \ldots, \phi(b_s)]$ (respectively, finitely generated over $B[a_1, \ldots, a_k]$), such that for every $i$, there is $A_{\phi(b_i)}=C_{\phi(b_i)}$ (respectively, $A_{a_i}=C_{a_i}$). Then \cite[Lemma 1.2]{zub2} concludes the proof.
\end{proof}
Combining Lemma \ref{finitely generated} with \cite[Lemma 3.5]{maszub1}, one can easily derive the following characterization of morphisms locally of finite type, and morphisms of finite type as well (cf. \cite[Exercise II.3.3(b)]{hart}).
\begin{lm}\label{finite type}
A morphism $f : X\to Y$ in $\mathcal{SV}$ is locally of finite (respectively, of finite) type if and only if for every open super-subscheme 
$V\simeq\mathrm{SSpec}(B)$ of $Y$, the open super-subscheme $U=f^{-1}(V)$ has a (finite) open covering by
super-subschemes $U_i\simeq\mathrm{SSpec}(A_i)$ such that each $A_i$ is a finitely generated $B$-superalgebra. In particular, if both $X$ and $Y$ are affine, then $f$ is of finite type if and only if $\mathcal{O}(X)$ is a finitely generated $\mathcal{O}(Y)$-superalgebra.
\end{lm}
A superscheme $X$ is said to be \emph{(locally) of finite type} if the canonical morphism $X\to\mathrm{SSpec}(\Bbbk)$ is  (locally) of finite type.

\section{Separated and proper morphisms of superschemes}

Let $\mathcal{P}$ be a property of a class of geometric superschemes or a class of morphisms in $\mathcal{SV}$. We say that this property is \emph{even reducible}, provided $X$ satisfies $\mathcal{P}$ if and only if $X_{ev}$ does (respectively, $f$ satisfies $\mathcal{P}$ if and only if $f_{ev}$ does). For example, the property of a Noetherian (geometric) superscheme to be affine is even reducible (see \cite[Theorem 3.1]{zub1}). The equivalence $\mathcal{SV}\simeq\mathcal{SF}$ naturally translates this definition to the category $\mathcal{SF}$.  That is, $\mathcal{P}$ is even reducible in $\mathcal{SV}$ if and only if it is even reducible in $\mathcal{SF}$.

Since the categories $\mathcal{SV}$ and $\mathcal{SF}$ are equivalent and $\mathcal{SF}$ has fibered products,  $\mathcal{SV}$ has as well. Let $p_X$ and $p_Y$ denote the canonical projection morphisms $X\times_S Y\to X$ and $X\times_S Y\to Y$, respectively.
The fibered product $X\times_{\mathrm{SSpec}(\Bbbk)} Y$, denoted by $X\times Y$, is a direct product of superschemes $X$ and $Y$ in $\mathcal{SV}$.

The following lemma will be used later.
\begin{lm}\label{coveingsofproducts}
Let $X$ and $Y$ be superschemes over a superscheme $S$. Let $\{U_i\}_{i\in I}$ be an open covering of $X$. Then
the superschemes $U_i\times_S Y$ form an open covering of $X\times_S Y$.
\end{lm}
\begin{proof}
By Theorem 5.14 and remark after \cite[Proposition 5.12]{maszub1}, one can work in the category $\mathcal{SF}$. The easy superization of both \cite[I.1.7(3)]{jan} and \cite[I.1.7(4)]{jan} implies the statement.
\end{proof}
A morphism $f : X\to Y$ in $\mathcal{SV}$ is called {\it separated} if the {\it diagonal} morphism $\delta_f : X\to X\times_Y X$ is a closed immersion. We also say that $X$ is {\it separated} over $Y$. In particular, a superscheme $X$ is called {\it separated}, provided $X$ is separated over $\mathrm{SSpec}(\Bbbk)$ (cf. \cite[II, \S 4]{hart}). For example, any morphism of affine superschemes is separated (see \cite[Proposition II.4.1]{hart}). In particular, any affine superscheme is separated.
The following lemma superizes \cite[Corollary II.4.2]{hart}. 
\begin{lm}\label{when given morphism is separated?}
A morphism $f : X\to Y$ is separated if and only $\delta_f^e(X^e)$ 
is a closed subset of $(X\times_Y X)^e$.
\end{lm}
\begin{proof}
The proof of \cite[Corollary II.4.2]{hart} can be copied verbatim, provided we prove the following. If $V$ is an affine super-subscheme of $Y$ and $U$ is an affine super-subscheme of $X$ such that $f(U)\subseteq V$, then the natural morphism $U\times _V U\to X\times_Y X$ is an open immersion. By the remark after \cite[Proposition 5.12]{maszub1}, all we need is to check the analogous statement in the category $\mathcal{SF}$ which is evident (see \cite[I.1.7(3)]{jan}).
\end{proof}
\begin{lm}\label{even descent}
For any morphisms $f : X\to S$ and $g : Y\to S$ in $\mathcal{SV}$, the fibred product $X_{ev}\times_{S_{ev}} Y_{ev}$ is isomorphic to $(X\times_S Y)_{ev}$.
\end{lm}
\begin{proof}
For every $L\in\mathcal{SV}$, let $L_0$ denote a purely-even superscheme $(L^e, (\mathcal{O}_L)_0)$.
The morphisms $L\to Z_{ev}$ are in one-to-one correspondence with the morphisms $L_0\to Z$, which factor through $L_0\to Z_{ev}$. Therefore, by the universality of a fiber product, the morphisms $L\to X_{ev}\times_{S_{ev}} Y_{ev}$ are in one-to-one correspondence with the morphisms $L_0\to X\times_S Y$ hence with the morphisms $L\to (X\times_S Y)_{ev}$, proving the lemma.
\end{proof}
\begin{pr}\label{reduction to even}
A morphism $f : X\to Y$ is separated if and only if $f_{ev}$ is separated, if and only if $f_{res}$ is separated (the latter is regarded as a morphism of geometric schemes).
\end{pr}
\begin{proof}
Lemma \ref{even descent} implies that $\delta_f|_{X_{ev}}$ can be identified with $\delta_{f|_{X_{ev}}}$. The last equivalence is now apparent.
\end{proof}
\begin{cor}\label{separated is even reducible}
The property of a superscheme morphism to be separated is even reducible. In particular, all standard properties of separated morphisms of schemes, formulated in \cite[Corollary II.4.6]{hart}, are valid for superschemes.
\end{cor}
A morphism of geometric superschemes $f : X\to Y$ is called {\it closed} if $f^e$ takes closed subsets of $X^e$ to closed subsets of $Y^e$. It is called {\it universally closed} if for any morphism $Y'\to Y$ the projection 
$X\times_Y Y'\to Y'$ is closed. Obviously, the first property is even reducible. Hence, by Lemma \ref{even descent}, the second one is, too.

A morphism of geometric superschemes $f : X\to Y$ is called {\it proper} if it is separated, universally closed, and of finite type.
\begin{lm}
If $X$ is Noetherian, then $f : X\to Y$ is of finite type if and only if $f_{ev}$ is of finite type, if and only if $f_{res}$ is of finite type.
\end{lm}
\begin{proof}
By Lemma \ref{finite type}, it is enough to consider the case $X=\mathrm{SSpec}(A)$ and $Y=\mathrm{SSpec}(B)$, where $A$ is a Noetherian superalgebra, and $\overline{A}$ is finitely generated over $\overline{B}$. Since $A_1^n/A_1^{n+1}$ is a finitely generated $\overline{A}$-module for every $n\geq 1$ and $A_1^N =0$ for sufficiently large $N$, $A$ is finitely generated over $B_0$, hence over $B$.
\end{proof}
\begin{cor}\label{II.4.8}
In the full subcategory consisting of Noetherian superschemes, \cite[Corollary II.4.8]{hart}, can be superized verbatim. 
\end{cor}
A Noetherian superscheme $X$ is called {\it complete} if it is separated and proper over $\mathrm{SSpec}(\Bbbk)$, i.e., the morphism $X\to \mathrm{SSpec}(\Bbbk)$ is separated and proper simultaneously. By the above, this property is even reducible.

Recall that a superalgebra $A$ is called \emph{reduced}, provided the algebra $\overline{A}$ is.
A geometric superscheme $X$ is called reduced if $X_{res}$ is a reduced scheme (cf. \cite[2.3]{maszub3}). This property is local, i.e., $X$ is reduced if and only if the superalgebra $\mathcal{O}_x$ is reduced for every $x\in X^e$. In particular, an affine superscheme $\mathrm{SSpec}(A)$ is reduced if and only if the superalgebra $A$ is reduced. Thus a superscheme $X$ is reduced if and only if every open affine super-subscheme of $X$ is reduced. 

Following \cite{milne}, we call a geometric superscheme $X$ {\it (algebraic) supervariety} if $X$ is of finite type, separated and {\it geometrically reduced}. The latter means that the superscheme $X_{\overline{\Bbbk}}=X\times \mathrm{SSpec}(\overline{\Bbbk})$ is reduced for an algebraically closed field $\overline{\Bbbk}$ of $\Bbbk$. By the above, a supervariety is even reducible.

A superalgebra $A$ is called \emph{Grassman-like} if $\overline{A}=\Bbbk$ or, equivalently, if odd elements generate $A$. 
\begin{lm}\label{when an affine superscheme is complete?}
A connected supervariety $\mathrm{SSpec}(A)$ is complete if and only if $A$ is a Grassman-like superalgebra.
\end{lm}
\begin{proof}
Use even reducibility of this property and \cite[A.114(f)]{milne}.
\end{proof}
Let $f : X\to Y$ be a morphism in $\mathcal{SV}$ and let $X$ be a Noetherian superscheme. Then there is a closed super-subscheme $Z$ of $Y$ such that 
$f^e(X^e)\subseteq Z^e$, the morphism $f$ factors through the closed immersion $Z\to Y$, and for any closed super-subscheme $Z'$ of $Y$ such that $f$ factors through the closed immersion $Z'\to Y$, $Z\to Y$ factors through $Z'\to Y$. Since $X$ is a Noetherian superscheme, the third remark after \cite[Proposition 3.1]{zub1} implies that $f_*\mathcal{O}_X$ is a coherent sheaf of $\mathcal{O}_Y$-supermodules. Then, by \cite[Corollary 3.2]{zub1} and \cite[Proposition 2.5]{maszub3}, the superideal sheaf $\ker (\mathcal{O}_Y\to f_*\mathcal{O}_X)$ is a coherent sheaf of $\mathcal{O}_Y$-supermodules, which defines the superscheme $Z$,  called a {\it superscheme-theoretic image} of $f$ and denoted by $f(X)$. Observe also that if $X$ is (geometrically) reduced, then $f(X)$ is also (geometrically) reduced.
\begin{pr}\label{complete superschemes}
The following statements hold:
\begin{enumerate}
\item If $Z$ is a closed super-subscheme of a complete Noetherian superscheme $X$, then $Z$ is complete.
\item Let $f : X\to Y$ be a morphism of algebraic supervarieties. If $X$ is complete, then $f^e(X^e)$ is closed, and $f(X)$ is complete.
\item If $X$ is a complete connected supervariety, then $\mathcal{O}_X(X^e)$ is a Grassman-like superalgebra.
\end{enumerate}
\end{pr}
\begin{proof}
The first statement follows by superized \cite[Corollary II.4.8(a, b)]{hart}.

It is easy to see that $f_{res}(X_{res})$ coincides with $f(X)_{res}$. Then the second statement follows by \cite[A.114(d)]{milne}.

Let $A$ denote $\mathcal{O}(X)$.
The induced morphism $f : X\to \mathrm{SSpec}(A)$ factors through its superscheme-theoretic image $f(X)$. Since $f(X)$ is closed, $f(X)\simeq \mathrm{SSpec}(A/I)$ for a superideal $I$ of $A$. On the other hand, the composition of superalgebra morphisms $A\to A/I\to\mathcal{O}(X)$ is an identity map. Thus $I=0$, hence $f(X)=\mathrm{SSpec}(A)$ is a complete connected supervariety. Lemma \ref{when an affine superscheme is complete?} concludes the proof.
\end{proof}

\section{Flat and faithfully flat morphisms}

Let $A$ be a superalgebra and $M$ be an $A$-supermodule. Recall that $M$ is said to be \emph{flat}, if the functor $N\mapsto M \otimes_A N$ takes any exact sequence of $A$-supermodules to the exact sequence of superspaces. If this functor is faithfully exact, then $M$ is called
\emph{faithfully flat}. By \cite[Lemma 5.1(1)]{mas}, $A$-supermodule $M$ is (faithfully) flat if and only if $M$ is (faithfully) flat as
the (left and right) $A$-module. Besides, $M$ is faithfully flat as an $A$-supermodule if and only if $M$ is flat and for any maximal super-ideal $\mathfrak{m}$ of $A$ there is $M\neq \mathfrak{m}M$ (combine \cite[Lemma 1.1(iii)]{zub2} and \cite[Proposition 1, I, \S 3.1]{bur}).

A morphism of geometric superschemes $f : X\to Y$ is said to be \emph{flat} if for any $x\in X^e$ the induced
superalgebra morphism ${O}_{Y, f^e(x)}\to \mathcal{O}_{X, x}$ is flat. If, additionally, $f^e$ is surjective, then $f$ is called \emph{faithfully flat}. 
\begin{lm}\label{another definition of flatness}
	A morphism $f : X\to Y$ is flat if and only if for any open affine super-subschemes $U\subseteq X$ and $V\subseteq Y$, such that $f^e(U^e)\subseteq V^e$, the restriction of $f$ to $U$ is flat if and only if 
	$\mathcal{O}(U)$ is a flat $\mathcal{O}(V)$-supermodule via superalgebra morphism $f^*$. 
\end{lm}
\begin{proof}
	The first equivalence is obvious. The part "only if" follows by \cite[Proposition 1.1 (iii)]{zub2}. Conversely, let $\mathcal{O}(U)$ be a flat $\mathcal{O}(V)$-supermodule. For any prime superideal $\mathfrak{p}$ of $\mathcal{O}(U)$ set $\mathfrak{q}=\mathfrak{p}\cap\mathcal{O}(V)$. 
	Then $\mathcal{O}(U)_{\mathfrak{q}}\simeq \mathcal{O}(V)_{\mathfrak{q}}\otimes_{\mathcal{O}(V)} \mathcal{O}(U)$ is flat over $\mathcal{O}(V)_{\mathfrak{q}}$. By \cite[Lemma 1.2 (i)]{zub2}, $\mathcal{O}(U)_{\mathfrak{p}}$ is flat over $\mathcal{O}(U)_{\mathfrak{q}}$. Lemma is proven.  
\end{proof}	
\begin{cor}\label{faithfully flat}
	If $X$ and $Y$ are affine superschemes, then $f$ is faithfully flat if and only if $\mathcal{O}(X)$ is faithfully flat over $\mathcal{O}(Y)$. 	
\end{cor}
\begin{proof}
The map $f^e$ is surjective if and only if $\mathcal{O}(X)\mathfrak{n}\neq\mathcal{O}(X)$ for any maximal super-ideal $\mathfrak{n}$ of $\mathcal{O}(Y)$.	
\end{proof}
Recall that a morphism of geometric superschemes $f : X\to Y$ is called \emph{affine}, if for any affine superscheme $Z$ and any morphism $Z \to Y$ the fibered product
$X\times_Y Z$ is an affine superscheme (cf. \cite{jan, maszub1}). In particular, if $f$ is affine and $Z$ is an affine super-subscheme of $Y$, then $f^{-1}(Z)$ is an affine super-subscheme of $X$.
\begin{lm}\label{finite type and flatness}
	Let 
	\[\begin{array}{ccc} 
	Z & \to & W \\
	\downarrow & & \downarrow \\
	U & \to & V
	\end{array}  \]	
	be a cartesian square of geometric superschemes. Assume that the morphisms $p_U : Z\to U, \ g : U\to V$ and $q : W\to V$ are affine, and $g$ is faithfully flat. If the superscheme $Z$ is of finite type and $g$ factors through $q$, then $W$ is of finite type as well.    
\end{lm}
\begin{proof}
Choose a covering of $V$ by open affine super-subschemes $V_i\simeq \mathrm{SSpec}(A_i), 1\leq i\leq l$. The open super-subschemes $W_i=q^{-1}(V_i)$ are affine and form a covering of $W$.  

Since $p_U^{-1}g^{-1}(V_i)=p_W^{-1}(W_i)$ for each $1\leq i\leq l$, 
we have cartesian squares in the subcategory of affine superschemes  
	\[\begin{array}{ccc} 
	(gp_U)^{-1}(V_i) & \to & W_i \\
	\downarrow & & \downarrow \\
	g^{-1}(V_i) & \to & V_i
	\end{array},   \]	
each of which  satisfies the conditions of our lemma. Therefore, without loss of generality one can assume that
all superschemes $V, U, W, Z$ are affine, say $V\simeq\mathrm{SSpec}(A), W\simeq\mathrm{SSpec}(B), U\simeq \mathrm{SSpec}(C)$ and
$Z\simeq\mathrm{SSpec}(C\otimes_A B)$. 
Since $C\otimes_A B$ is a finitely generated $\Bbbk$-superalgebra, there is a finitely generated $\Bbbk$-super-subalgebra $L$ of $B$ such that $C\otimes_A L=C\otimes_A B$. Since $C$ is a faithfully flat $A$-supermodule, we obtain $L=B$.  Lemma is proven.	
\end{proof}

\section{Group superschemes}

A group object in the category $\mathcal{SF}$ is called a \emph{group superscheme}. Similarly, 
a group object in the category $\mathcal{SV}$ is called a \emph{geometric group superscheme}.
These objects form subcategories in $\mathcal{SF}$ and $\mathcal{SV}$, respectively, with morphisms preserving their group structures. They are denoted by $\mathcal{SFG}$ and $\mathcal{SVG}$, correspondingly. Theorem \ref{comparison} implies that $\mathcal{SFG}\simeq\mathcal{SVG}$.

A geometric group superscheme is called \emph{locally algebraic} if it is of locally finite type as a superscheme. Symmetrically, a group superscheme $\mathbb{G}$ is called \emph{locally algebraic} if its geometric counterpart $G$ is locally algebraic.  

Locally algebraic group superschemes form a full subcategory of $\mathcal{SFG}$, denoted by $\mathcal{SFG}_{la}$. The category $\mathcal{SFG}_{la}$ is equivalent to the full subcategory $\mathcal{SVG}_{la}$ of $\mathcal{SVG}$, consisting of all locally algebraic geometric group superschemes.  

A geometric group superscheme is called \emph{algebraic} if it is of finite type as a superscheme. As above, one can define algebraic group superschemes as two equivalent categories
$\mathcal{SVG}_a$ and $\mathcal{SFG}_a$.

Let $\mathbb{E}$ denote the \emph{trivial group superscheme}. That is, for every $A\in\mathsf{SAlg}_{\Bbbk}$, the group $\mathbb{E}(A)=\{e_A\}$ is trivial. Its geometric counterpart is isomorphic to $e\simeq \mathrm{SSpec}(\Bbbk)$ with the trivial group structure. The unique point of $(\mathrm{SSpec}(\Bbbk))^e$ is also denoted by $e$.
 
A sequence 
\[e\to H\to G\stackrel{f}{\to} R\to e\]
in the category $\mathcal{SVG}$ is called \emph{exact} if the corresponding sequence 
\[\mathbb{E}\to\mathbb{H}\to\mathbb{G}\stackrel{\bf f}{\to} \mathbb{R}\to\mathbb{E}\]
in the category $\mathcal{SFG}$ is exact, that is,
for any superalgebra $A$ we have $\mathbb{H}(A)=\ker{\bf f}(A)$,
and the \emph{sheafification} of the \emph{naive quotient functor} 
\[A\mapsto (\mathbb{G}/\mathbb{H})_{(n)}(A)=\mathbb{G}(A)/\mathbb{H}(A)\subseteq\mathbb{R}(A),\] with respect to the Grothendieck topology of 
\emph{fppf coverings}, coincides with $\mathbb{R}$. The latter is equivalent to the following condition. For any superalgebra $A$ and any $g\in\mathbb{R}(A)$, there is a \emph{finitely presented} $A$-superalgebra $A'$, which is a faithfully flat $A$-supermodule such that 
$\mathbb{R}(\iota_A)(g)\in {\bf f}(A')(\mathbb{G}(A'))$, where $\iota_A : A\to A'$ is the corresponding monomorphism of superalgebras (cf. \cite{maszub1}, page 144, and Proposition 5.15 therein). 

A geometric group superscheme $G$ (or the corresponding group superscheme $\mathbb{G}$) is called an \emph{abelian supervariety} if $G$ is a complete supervariety. Since both properties are even reducible, we obtain that $G$ is an abelian supervariety  if and only if $G_{ev}$ is an abelian variety (cf. \cite[Definition 10.13]{milne}). 

Let $\mathbb{R}$ be a group superscheme that acts on a superscheme $\mathbb{X}$ on the right, and on a superscheme $\mathbb{Y}$ on the left. Then $\mathbb{R}$ acts on
$\mathbb{X}\times\mathbb{Y}$ as $(x, y)\cdot r=(x\cdot r, r^{-1}\cdot y)$ for $x\in\mathbb{X}(B), y\in\mathbb{Y}(B), r\in\mathbb{R}(B)$ and $B\in\mathsf{SAlg}_{\Bbbk}$. The sheafification of the naive quotient
$B\mapsto(\mathbb{X}(B)\times \mathbb{Y}(B))/\mathbb{R}(B)$ is called the \emph{homogeneous fiber quotient}, and it is denoted by $\mathbb{X}\times^{\mathbb{R}}\mathbb{Y}$. This construction is functorial in both $\mathbb{X}$ and $\mathbb{Y}$.

Assume that all the previous superschemes are affine, say $\mathbb{X}\simeq\mathrm{SSp}(A), \mathbb{Y}\simeq\mathrm{SSp}(B)$, and $\mathbb{R}\simeq\mathrm{SSp}(D)$. In particular, 
$A$ and $B$ are right and left $D$-coideal superalgebras (cf. \cite{zub2}).

Let $V$ and $W$ are the right and left $D$-supercomodules respectively. Let $\tau_V$ and $\tau_V$ denote the corresponding supercomodule maps $V\to V\otimes D$ and  $W\to D\otimes W$. Then the 
\emph{cotensor product}  $V\square_{D} W$ is defined as $\ker(\mathrm{id}_V\otimes\tau_W-\tau_V\otimes\mathrm{id}_W)$.
\begin{lm}\label{if it is affine}
If $\mathbb{X}\times^{\mathbb{R}}\mathbb{Y}$ is an affine superscheme, it is canonically isomorphic to $\mathrm{SSp}(A\square_{D}B)$. Moreover, if there are $\mathbb{R}$-equivariant morphisms $\mathbb{X}'\to\mathbb{X}$ and $\mathbb{Y}'\to\mathbb{Y}$, where $\mathbb{X}'\simeq\mathrm{SSp}(A')$, $\mathbb{Y}'\simeq\mathrm{SSp}(B')$, and $\mathbb{X}'\times^{\mathbb{R}}\mathbb{Y}'$ are affine, then the induced morphism 
$\mathbb{X}'\times^{\mathbb{R}}\mathbb{Y}'\to \mathbb{X}\times^{\mathbb{R}}\mathbb{Y}$ is dual to 
the natural superalgebra morphism $A'\square_{D} B'\to A\square_{D} B$.
\end{lm}
\begin{proof}
Use \cite[Proposition 4.1]{zub2}. 
\end{proof}

\section{The affinization of group superschemes}

Let $X$ be a geometric superscheme, and $U$ be its open super-subscheme. As it has been already observed, a (geometric) superscheme 
morphism $f : Z\to X$ factors through $U$ if and only if $f^e(Z^e)\subseteq U$. 

In particular, if $X\times_Z Y$ is a fibred product in $\mathcal{SV}$, then for every open super-subschemes $U$ and $V$ of $X$ and $Y$, respectively, the open super-subscheme $p_X^{-1}(U)\cap p_Y^{-1}(V)$ of $X\times_Z Y$ is naturally isomorphic to $U\times_Z V$. Here $p_X : X\times_Z Y\to X$ and $p_Y : X\times_Z Y\to Y$ are the canonical projections. Note that if $\mathbb{X}(\Bbbk)\neq\emptyset$ (respectively, $\mathbb{Y}(\Bbbk)\neq\emptyset$), then $p_X^*$ (respectively, $p^*_Y$) is injective.    

Note that the sheaf morphisms $p_X^*$ and $p_Y^*$ induce a superalgebra morphism 
$\mathcal{O}(X)\otimes_{\mathcal{O}(Z)}\mathcal{O}(Y)\to\mathcal{O}(X\times_Z Y)$. 

Let $G$ be a group superscheme, and let $m, \iota$ and $\epsilon$ denote the multiplication morphism
$G\times G\to G$, the inversion morphism $G\to G$, and the closed immersion $e\to G$, respectively. 
We also denote by $p_1$ and $p_2$ the corresponding projections from $G\times G$ to $G$.

Recall that if $G$ is affine, say $G\simeq\mathrm{SSpec}(A)$, then $A$ is a Hopf superalgebra with the coproduct $\Delta_A=m^*$, counit 
$\epsilon_A=\epsilon^*$ and antipode $S_A=\iota^*$.
\begin{lm}\label{morphism to affine supergroup}
Let $H\simeq\mathrm{SSpec}(B)$ be an affine geometric group superscheme. Then every morphism $G\to H$ in $\mathcal{SVG}$ is uniquely defined by a superalgebra morphism $\phi : B\to\mathcal{O}(G)$, that satisfies the following conditions:
\begin{enumerate}
\item The superalgebra morphism 
\[B\otimes B\stackrel{\phi\otimes\phi}{\to}\mathcal{O}(G)\otimes\mathcal{O}(G)\stackrel{p_1^*\otimes p_2^*}
{\to}\mathcal{O}(G\times G)\] makes the diagram
\[\begin{array}{ccc}
B\otimes B & \to & \mathcal{O}(G\times G) \\
\uparrow & & \uparrow \\
B & \to & \mathcal{O}(G)
\end{array},\]
where the vertical arrows are $\Delta_B$ and $m^*$, respectively, to be commutative. 
\item $\epsilon^* \phi=\epsilon_B$. 
\end{enumerate} 
\end{lm}
\begin{proof}
Use \cite[Lemma 4.1]{maszub1}.
\end{proof}
It is clear that there is the largest super-subalgebra $C$ of $\mathcal{O}(G)$, such that the superalgebra morphism $\Delta_C : C\to\mathcal{O}(G)\stackrel{m^*}{\to} \mathcal{O}(G\times G)$ maps $C$ to $(p_1^*\otimes p_2^*)(C\otimes C)$.
\begin{pr}\label{C is Hopf superalgebra}
$C$ is a Hopf superalgebra with the coproduct $\Delta_C=m^*|_C$, the antipode $\iota^*|_C$ and the counit $\epsilon^*|_C$. 
\end{pr}
\begin{proof}
Let $q_{12}$ and $q_3$ denote the canonical projections $(G\times G)\times G\to G\times G$ and $(G\times G)\times G\to G$, respectively. Symmetrically, let $q_{1}$ and $q_{23}$ denote the canonical projections $G\times (G\times G)\to G$ and $G\times (G\times G)\to G\times G$, respectively. The natural isomorphism $\alpha : (G\times G)\times G\to
G\times (G\times G)$ is uniquely defined by the identities:
\[q_1\alpha=p_1 p_{12}, \ p_1 q_{23}\alpha=p_2 q_{12} \text{ and } \ q_3=p_2q_{23}\alpha.\]
Moreover, the following identities are satisfied:
\[m q_{12}=p_1(m\times\mathrm{id}_G), \ q_3=p_2(m\times\mathrm{id}_G), \ mq_{23}=p_2(\mathrm{id}_G\times m), 
\text{ and } q_1=p_1(\mathrm{id}_G\times m).\] 
Then there is a commutative diagram
\[\begin{array}{ccccccc}
C^{\otimes 3} & \to & \mathcal{O}(G)^{\otimes 3} & \stackrel{p_1^*\otimes p_2^*\otimes\mathrm{id}_{\mathcal{O}(G)}}{\to} & \mathcal{O}(G\times G)\otimes\mathcal{O}(G) & \stackrel{q_{12}^*\otimes q_3^*}{\to} & \mathcal{O}((G\times G)\times G) \\
\uparrow & & & & \uparrow & & \uparrow \\
C^{\otimes 2} & &  & \longrightarrow & \mathcal{O}(G)^{\otimes 2} & \stackrel{p_1^*\otimes p_2^*}{\to} &
\mathcal{O}(G\times G) \\  
\uparrow & & & & & & \uparrow \\
C & & & & \longrightarrow & & \mathcal{O}(G)
\end{array},\]
where the upper vertical arrows are $\Delta_C\otimes\mathrm{id}_C, m^*\otimes\mathrm{id}_{\mathcal{O}(G)}$, and $(m\times\mathrm{id}_G)^*$, respectively. The lower vertical arrows are $\Delta_C$ and $m^*$, respectively.
Symmetrically, there is a commutative diagram
\[\begin{array}{ccccccc}
\mathcal{O}(G\times (G\times G)) & \stackrel{q_1^*\otimes q_{23}^*}{\leftarrow} & \mathcal{O}(G)\otimes\mathcal{O}(G\times G) & \stackrel{\mathrm{id}_{\mathcal{O}(G)}\otimes p_1^*\otimes p_2^*}{\leftarrow} & \mathcal{O}(G)^{\otimes 3} & \leftarrow & C^{\otimes 3} \\
\uparrow & & \uparrow &  &  & & \uparrow \\
\mathcal{O}(G\times G) & \stackrel{p_1^*\otimes p_2^*}{\leftarrow} &  \mathcal{O}(G)^{\otimes 2} & \longleftarrow &  & &
C^{\otimes 2} \\  
\uparrow & & & & & & \uparrow \\
\mathcal{O}(G) & & \longleftarrow &  &  & & C
\end{array},\]
where the upper vertical arrows are $(\mathrm{id}_G\times m)^*, \mathrm{id}_{\mathcal{O}(G)}\otimes m^*$, and $\mathrm{id}_C\otimes\Delta_C$, respectively. Using the isomorphism $\alpha$ and the axiom of associativity, we obtain that $\Delta_C$ is a coproduct. Therefore, $C$ is a superbialgebra. We leave it for the reader to verify 
the identity 
\[\Delta_C\iota^*=t_{C, C}(\iota^*\otimes\iota^*)\Delta_C,\]
where $t_{C, C}$ is the braiding $c_1\otimes c_2\mapsto (-1)^{|c_1||c_2|}c_2\otimes c_1$ for $c_1, c_2\in C$.
Therefore, $\Delta_C$ maps $\iota^*(C)$ to $\iota^*(C)^{\otimes 2}$, hence
$\iota^*(C)\subseteq C$. The proposition is proven. 
\end{proof}
\begin{cor}\label{affinization}
The geometric group superscheme $\mathrm{SSpec}(C)$ is the largest affine quotient of $G$. This means that every homomorphism from $G$ to an affine geometric group superscheme
uniquely factors through $G\to\mathrm{SSpec}(C)$. Similarly, $\mathrm{SSp}(C)$ is the largest affine quotient of $\mathbb{G}$. We denote the superschemes $\mathrm{SSpec}(C)$ and $\mathrm{SSp}(C)$
by  $G^{aff}$ and $\mathbb{G}^{aff}$, respectively.  
\end{cor}
In general, $C$ is a proper super-subalgebra of $\mathcal{O}(G)$. But if $G$ is algebraic, then
$C=\mathcal{O}(G)$. The proof is given in two lemmas below.
\begin{lm}(see \cite[Lemma 1.8, I, \S 2]{DG})\label{no name} 
If $X$ and $Y$ are geometric superschemes such that $Y$ is of finite type and $X$ is affine, then $p^*_X\otimes p_Y^*$ induces an isomorphism $\mathcal{O}(X)\otimes\mathcal{O}(Y)\simeq\mathcal{O}(X\times Y)$, which is functorial in both $X$ and $Y$. 
\end{lm}
\begin{proof}
Choose a covering of $Y$ by open affine super-subschemes $V_i\simeq \mathrm{SSpec}(A_i), 1\leq i\leq t$. By Lemma \ref{coveingsofproducts}, $X\times V_i$ form a covering of $X\times Y$ by open affine super-subschemes. We have a commutative diagram
\[\begin{array}{ccccc}
\mathcal{O}(X\times Y) & \to & \prod_{1\leq i\leq t}\mathcal{O}(X\times V_i) & \rightrightarrows & \prod_{1\leq i\neq j\leq t}\mathcal{O}(X\times (V_i\cap V_j)) \\ 
\uparrow & & \uparrow & & \uparrow \\
\mathcal{O}(X)\otimes \mathcal{O}(Y) & \to & \prod_{1\leq i\leq t}\mathcal{O}(X)\otimes \mathcal{O}(V_i) & \rightrightarrows & \prod_{1\leq i\neq j\leq t}\mathcal{O}(X)\otimes \mathcal{O}(V_i\cap V_j)
\end{array}\]	
with the exact top and bottom rows. Besides, the left vertical arrows are $p_X^*\otimes p_Y^*, p^*_X\otimes p^*_{V_i}$ and 
$p^*_X\otimes p^*_{V_i\cap V_j}$ respectively. Since all of them, but $p^*_X\otimes p_Y^*$, are the isomorphisms, $p^*_X\otimes p_Y^*$ is an isomorphism as well.    
\end{proof}
\begin{lm}\label{no name final}
If $X$ and $Y$ are superschemes of finite type, then $p_X^*\otimes p^*_Y$ is an isomorphism, which is functorial in both $X$ and $Y$.  	
\end{lm}	
\begin{proof}
Just consider the similar diagram for any finite open affine covering of $X$ and apply Lemma \ref{no name}.	
\end{proof}

\section{The tangent functor at specific points}

The content of this section extends \cite{DG, zub2}.

Let $\mathbb{X}$ be a $\Bbbk$-functor. For every point $x\in \mathbb{X}(\Bbbk)$, one can define a $\Bbbk$-functor
$\mathrm{T}_x(\mathbb{X})$, called a \emph{tangent $\Bbbk$-functor} at point $x$, as follows.
For each superalgebra $R$, set $x_R=\mathbb{X}(\iota^R_{\Bbbk})(x)$, where $\iota_{\Bbbk}^R : \Bbbk\to R$ is the canonical embedding. Observe that $x=x_{\Bbbk}$. 

Let $R[\epsilon_0, \epsilon_1]$ denote the \emph{superalgebra of dual super-numbers}, that is
\[R[\epsilon_0, \epsilon_1]=R[x|y]/<x^2,y^2, xy, yx>, \text{ where } |x|=0, |y|=1,\]
and $\epsilon_0$ and $\epsilon_1$ denote the residue classes of $x$ and $y$, respectively. We have two $R$-superalgebra morphisms $p_R : R[\epsilon_0, \epsilon_1]\to R$ and $i_R : R\to   R[\epsilon_0, \epsilon_1]$ such that $p_R(r)=r$,  $p_R(\epsilon_i)=0$ for $i=0, 1$, and $i_R(r)=r$ for $r\in R$. Since, $p_r i_R=\mathrm{id}_R$,
the map $\mathbb{X}(p_R)$ is surjective. 

Set $\mathrm{T}_x(\mathbb{X})(R)=\mathbb{X}(p_R)^{-1}(x_R)$. It is clear that $\mathrm{T}_x(\mathbb{X})$ is a $\Bbbk$-functor. 

Assume that $\mathbb{X}$ is a superscheme. Let $\mathbb{U}\simeq\mathrm{SSp}(A)$ be an open affine super-subscheme of $\mathbb{X}$. As it has been already observed, $x_R$ belongs to $\mathbb{U}(R)$ if and only if $x$ belongs to $\mathbb{U}(\Bbbk)$.
Besides, $x$ can be identified with a superalgebra morphism $\epsilon_A : A\to \Bbbk$ and $x_R$ can be identifed with $\iota^R_{\Bbbk}\epsilon_A$. Let $\mathfrak{m}_A$ denote $\ker\epsilon_A$. 

Recall that there is a superscheme 
morphism ${\bf i}_x : \mathrm{SSp}(\mathcal{O}_x)\to\mathbb{X}$ (corresponding to the closed immersion $i_x : \mathrm{SSpec}(\mathcal{O}_x)\to X$) that factors through every open neighborhood of $x$.
\begin{lm}\label{some pre-image}
For every open super-subscheme $\mathbb{U}$ of $\mathbb{X}$ and every $R\in\mathsf{SAlg}_{\Bbbk}$, there is
$\mathbb{X}(p_R)^{-1}(\mathbb{U}(R))\subseteq\mathbb{U}(R[\epsilon_0, \epsilon_1])$. In particular,
$\mathrm{T}_x(\mathbb{X})(R)=\mathrm{T}_x(\mathbb{U})(R)$, provided $x\in\mathbb{U}(\Bbbk)$.
\end{lm}
\begin{proof}
Let $U$ be an open super-subscheme of $X$ that corresponds to $\mathbb{U}$. Then the statement of the lemma is equivalent to the following. If $f : \mathrm{SSpec}(R[\epsilon_0, \epsilon_1])\to X$ satisfies
$(f\mathrm{SSpec}(p_R))^e(\mathrm{SSpec}(R)^e)\subseteq U^e$, then $f^e((\mathrm{SSpec}(R[\epsilon_0, \epsilon_1]))^e)\subseteq U^e$. 

Note that any prime super-ideal of $R[\epsilon_0, \epsilon_1]$ has a form $\mathfrak{p}+\epsilon_0 R +\epsilon_1 R$, where $\mathfrak{p}\in (\mathrm{SSpec}(R))^e$. Since $(\mathrm{SSpec}(p_R))^e(\mathfrak{p})=\mathfrak{p}+\epsilon_0 R +\epsilon_1 R$, the statement follows. 
\end{proof}
Let $V$ be a superspace. Denote by $\mathbb{V}_a$ the $\Bbbk$-functor given by $\mathbb{V}_a(R)=V\otimes R$ for  $R\in\mathsf{SAlg}_{\Bbbk}$. From now on, assume that $\mathbb{X}$ is locally of finite type, i.e., its geometric counterpart $X$ is.
Lemma \ref{some pre-image} implies that the superspace $\mathfrak{m}_x/\mathfrak{m}_x^2$ is finite-dimensional.  

\begin{lm}\label{standard description of tangent functor}
The functor $\mathrm{T}_x(\mathbb{X})$ is isomorphic to $\mathbb{V}_a$, where $V=(\mathfrak{m}_x/\mathfrak{m}_x^2)^*$.
\end{lm}
\begin{proof}
By Lemma \ref{some pre-image}, we have a commutative diagram
\[\begin{array}{ccccc}
\mathrm{T}_x(\mathbb{X})(R) & \to & \mathbb{X}(R[\epsilon_0, \epsilon_1]) & \to & \mathbb{X}(R) \\
\parallel & & \uparrow & & \uparrow \\
\mathrm{T}_x(\mathbb{U})(R) & \to & \mathbb{U}(R[\epsilon_0, \epsilon_1]) & \to & \mathbb{U}(R) \\
\parallel & & \uparrow & & \uparrow \\
\mathrm{T}_x(\mathrm{SSp}(\mathcal{O}_x))(R) & \to & \mathrm{SSp}(\mathcal{O}_x)(R[\epsilon_0, \epsilon_1]) & \to & \mathrm{SSp}(\mathcal{O}_x)(R),
\end{array}\]
where $\mathbb{U}$ is an open affine super-subscheme of $\mathbb{X}$ with $x\in\mathbb{U}(\Bbbk)$. The lower middle and rightmost vertical arrows are ${\bf i}_x(R[\epsilon_0, \epsilon_1])$ and ${\bf i}_x(R)$, respectively.

In particular, if $\mathbb{U}\simeq\mathrm{SSp}(A)$, then $(\mathfrak{m}_A/\mathfrak{m}_A^2)^*\otimes R\simeq \mathrm{T}_x(\mathbb{U})(R)$ via
\[(\phi\otimes r)(a)=\epsilon_A(a)+(-1)^{|r||a|}\epsilon_{k}\phi(\bar{a})r,\] 
where $r\in R, \phi\in (\mathfrak{m}_A/\mathfrak{m}_A^2)^*, a\in A, \bar{a}=a-\epsilon_A(a)$ and $k\equiv |r|+|\phi| \pmod 2$. Similarly, we have $(\mathfrak{m}_x/\mathfrak{m}_x^2)^*\otimes R\simeq \mathrm{T}_x(\mathrm{SSp}(\mathcal{O}_x))(R)$, and both isomorphisms are functorial in $R$. Since
$\mathfrak{m}_A/\mathfrak{m}_A^2\simeq \mathfrak{m}_x/\mathfrak{m}_x^2$, the lemma follows.
\end{proof}
To an homogeneous element $r\in R$, we associate an endomorphism $\hat{r}$ of the superalgebra $R[\epsilon_0, \epsilon_1]$ defined by
\[\hat{r}(a+\epsilon_0 b+\epsilon_1 c)=a+(-1)^{|a||r|}(\epsilon_{|r|}br+\epsilon_{1+|r|}cr).\]
In fact, we have
\[\hat{r}((a+\epsilon_0 b+\epsilon_1 c)(a'+\epsilon_0 b'+\epsilon_1 c'))=\]
\[\hat{r}(aa'+(ba'+ab')\epsilon_0+((-1)^{|a|}ac'+ca')\epsilon_1)=\]
\[aa' +(-1)^{|r|(|a|+|a'|)}[\epsilon_{|r|}(ba'+ab')r+\epsilon_{1+|r|}((-1)^{|a|}ac'+ca')r] =\]
\[(a+(-1)^{|r||a|}(\epsilon_{|r|}br+\epsilon_{1+|r|}cr))(a'+(-1)^{|r||a'|}(\epsilon_{|r|}b'r+\epsilon_{1+|r|}c'r))=\]
\[\hat{r}(a+b\epsilon_0+c\epsilon_1)\hat{r}(a'+b'\epsilon_0+c'\epsilon_1).\]
Moreover, if $r'$ is another homogeneous element, then $\widehat{r r'}=\widehat{r'}\widehat{r}$. 
This proves the following lemma.

\begin{lm}\label{R-supermodule structure}
An endomorphism $\hat{r}$ satisfies $p_R\hat{r}=p_R$ and $\hat{r}i_R=i_R$. Therefore, $\mathbb{X}(\hat{r})$
maps $\mathrm{T}_x(\mathbb{X})(R)$ to itself. Moreover, if we identify $\mathrm{T}_x(\mathbb{X})(R)$ with
$(\mathfrak{m}_x/\mathfrak{m}_x^2)^*\otimes R$, then $\mathbb{X}(\hat{r})$ coincides with the map
$\phi\otimes r'\mapsto \phi\otimes r'r$. In particular, $\mathrm{T}_x(\mathbb{X})(R)$ has a natural structure of the right $R$-supermodule.
\end{lm}
Let ${\bf f} : \mathbb{X}\to\mathbb{Y}$ be a morphism of $\Bbbk$-functors. Choose points $x\in\mathbb{X}(K)$ and $y\in \mathbb{Y}(\Bbbk)$ such that ${\bf f}(\Bbbk)(x)=y$. For every superalgebra $R$, the map ${\bf f}(R)$
takes $x_R$ to $y_R$. Therefore, there is a commutative diagram
\[\begin{array}{ccccc}
\mathrm{T}_x(\mathbb{X})(R) & \to & \mathbb{X}(R[\epsilon_0, \epsilon_1]) & \to & \mathbb{X}(R) \\
\downarrow & & \downarrow & & \downarrow \\
\mathrm{T}_y(\mathbb{Y})(R) & \to & \mathbb{Y}(R[\epsilon_0, \epsilon_1]) & \to & \mathbb{Y}(R)
\end{array}.\]
The induced map $\mathrm{T}_x(\mathbb{X})(R)\to \mathrm{T}_y(\mathbb{Y})(R)$ is functorial in $R$. That is,  $\bf f$ induces a $\Bbbk$-functor morphism $\mathrm{d}_x{\bf f} : \mathrm{T}_x(\mathbb{X})\to \mathrm{T}_y(\mathbb{Y})$, called the \emph{differential} of $\bf f$ at the point $x$.  

Let $f : X\to Y$ be a morphism of locally algebraic geometric superschemes. If $f^e(x)=y$, then $f^*_x$ induces
a morphism of superspaces $d_x f : (\mathfrak{m}_x/\mathfrak{m}_x^2)^*\to (\mathfrak{n}_y/\mathfrak{n}_y^2)^*$, where
$\mathfrak{n}_y$ is the maximal superideal of $\mathcal{O}_{Y, y}$.
\begin{lm}\label{another interpretation of differential}
If we identify $\mathrm{T}_x(\mathbb{X})$ and $\mathrm{T}_y(\mathbb{Y})$ with $\mathbb{V}_a$ and $\mathbb{W}_a$, respectively, where $V=(\mathfrak{m}_x/\mathfrak{m}_x^2)^*$ and $W=(\mathfrak{n}_y/\mathfrak{n}_y^2)^*$, then $\mathrm{d}_x{\bf f}$ is naturally induced by $d_x f$.
\end{lm}
\begin{proof}
Replace $\mathbb{X}$ and $\mathbb{Y}$ by open affine neighborhoods $\mathbb{U}$ and $\mathbb{V}$ of $x$ and $y$, respectively, such that $\mathbb{U}\subseteq {\bf f}^{-1}(\mathbb{V})$, and use the diagram from Lemma \ref{standard description of tangent functor}.
\end{proof}

\section{Lie superalgebra of a group superscheme}

Let $\mathbb{G}$ be a locally algebraic group superscheme, $e$ be the identity element of $\mathbb{G}(\Bbbk)$, and 
$\mathrm{Lie}(\mathbb{G})$ be the tangent functor $\mathrm{T}_e(\mathbb{G})$.
Then $\mathrm{Lie}(\mathbb{G})$ is called the \emph{Lie superalgebra functor} of $\mathbb{G}$.
 
Arguing as in Lemma \ref{standard description of tangent functor}, we obtain the exact sequence
\[\begin{array}{ccccccccc}
1 & \to & \mathrm{Lie}(\mathbb{G})(R) & \to & \mathbb{G}(R[\epsilon_0, \epsilon_1]) & \to & \mathbb{G}(R) & \to & 1\\
 & & \parallel & & \uparrow & & \uparrow & & \\
0 & \to & (\mathfrak{m}_e/\mathfrak{m}^2_e)^*\otimes R & \to & \mathrm{SSp}(\mathcal{O}_e)(R[\epsilon_0, \epsilon_1]) & \to &
\mathrm{SSp}(\mathcal{O}_e)(R) & & .
\end{array}\]
Denote the superspace $(\mathfrak{m}_e/\mathfrak{m}^2_e)^*$ by $\mathfrak{g}$  and call it the \emph{Lie superalgebra} of $\mathbb{G}$.

Since $\mathbb{G}(p_R)$ and $\mathbb{G}(i_R)$ are group homomorphisms, there is an exact split sequence of groups:
\[1\to \mathrm{Lie}(\mathbb{G})(R)\to\mathbb{G}(R[\epsilon_0, \epsilon_1])\stackrel{\mathbb{G}(p_R)}{\to}\mathbb{G}(R)\to 1.\]
In other words, $\mathbb{G}(R[\epsilon_0, \epsilon_1])\simeq\mathbb{G}(R)\ltimes \mathrm{Lie}(\mathbb{G})(R)$, where $\mathbb{G}(R)$ is identified with the subgroup $\mathbb{G}(i_R)(\mathbb{G}(R))\leq\mathbb{G}(R[\epsilon_0, \epsilon_1])$. The group $\mathbb{G}(R)$ acts on $\mathrm{Lie}(\mathbb{G})(R)$ as
\[\mathrm{Ad}(g)(z)=\mathbb{G}(i_R)(g)z \mathbb{G}(i_R)(g)^{-1} \text{ for } g\in\mathbb{G}(R) \text{ and } z\in\mathrm{Lie}(\mathbb{G})(R).\]
This action is called \emph{adjoint}.
\begin{lm}\label{linearity of adjoint action}
The adjoint action of $\mathbb{G}(R)$ on $\mathrm{Lie}(\mathbb{G})(R)$ is $R$-linear and functorial in $R$. Moreover, $\mathbb{G}(R)$ acts on $\mathrm{Lie}(\mathbb{G})(R)$ by parity preserving operators.
\end{lm}
\begin{proof}
The first statement can be easily derived from $\hat{r}i_R=i_R$. The second is obvious. To prove the last statement, we define an endomorphism $\iota$ of superalgebra $R[\epsilon_0, \epsilon_1]$ by
\[a+\epsilon_0 b+\epsilon_1 c\mapsto a+\epsilon_0 b-\epsilon_1 \text{ for } a, b, c\in R.\]
It satisfies $p_R\iota=p_R, \iota i_R=i_R$ and commutes with each $\hat{r}$. Thus $\mathbb{G}(\iota)$ induces an automorphism of $R$-supermodule $\mathrm{Lie}(\mathbb{G})(R)$ that commutes with the adjoint action of $\mathbb{G}(R)$. On the other hand, if we identify $\mathrm{Lie}(\mathbb{G})(R)$ with
$(\mathfrak{m}_e/\mathfrak{m}^2_e)^*\otimes R$, then one immediately sees that $\mathbb{G}(\iota)$ sends a homogeneous element $x\in \mathrm{Lie}(\mathbb{G})(R)$ to $(-1)^{|x|}x$.  
\end{proof}

\section{A fragment of differential calculus}

In this section, we superize a fragment of differential calculus from \cite[II, \S 4]{DG}, for locally algebraic group superschemes (see also \cite{ccf, zub2}). 

Let $V$ be a finite-dimensional superspace. Recall that the affine (algebraic) group superscheme $\mathrm{GL}(V)$ is defined as
\[\mathrm{GL}(V)(R)=\mathrm{End}_R(V\otimes R)^{\times}_0 \text{ for }  R\in\mathsf{SAlg}_{\Bbbk}.\]
In other words, each $\mathrm{GL}(V)(R)$ consists of all graded (parity-preserving) $R$-linear automorphisms of $R$-supermodule $V\otimes R$. The group superscheme $\mathrm{GL}(V)$ is called the \emph{general linear supergroup}. 
 
Let $\mathbb{G}$ be a group superscheme. Then a (finite-dimensional) superspace $V$ is called a \emph{(left) $\mathbb{G}$-supermodule}, provided there is a group superscheme morphism ${\bf f} :\mathbb{G}\to\mathrm{GL}(V)$. 

From now on, all $\mathbb{G}$-supermodules are assumed to be finite-dimensional unless stated otherwise. Note that $\mathbb{G}$-supermodules form an abelian category (with graded morphisms). 

For example, $V$ is a $\mathrm{GL}(V)$-supermodule. If $\mathbb{G}$ is a locally algebraic group superscheme, then Lemma \ref{linearity of adjoint action} implies that the Lie superalgebra $\mathfrak{g}$ of $\mathbb{G}$ is a $\mathbb{G}$-supermodule via $\mathrm{Ad} : \mathbb{G}\to\mathrm{GL}(\mathfrak{g})$. 

Corollary \ref{affinization} immediately implies the following lemma.
\begin{lm}\label{equivalence of categories of supermodules}
Let $\mathbb{G}$ be a group superscheme. Then the category of $\mathbb{G}$-supermodules is naturally isomorphic to the category of $\mathbb{G}^{aff}$-supermodules. It is also isomorphic to the category of (finite-dimensional) right $\mathcal{O}(G^{aff})$-supercomodules (cf. \cite{zub2}).
\end{lm}
\begin{lm}\label{exact sequence of Lie superalgebras}
If
\[\mathbb{E} \to \mathbb{H}\stackrel{\bf i}{\to}\mathbb{G}\stackrel{\bf p}{\to}\mathbb{R}\to\mathbb{E}\]
is an exact sequence of locally algebraic group superschemes, then we have an exact sequence of $\Bbbk$-functors
\[\mathbb{E}\to \mathrm{Lie}(\mathbb{H})\stackrel{\mathrm{d}_e\bf i}{\to} \mathrm{Lie}(\mathbb{G})\stackrel{\mathrm{d}_e\bf p}{\to} \mathrm{Lie}(\mathbb{R}),\]
that is, the sequence of $A$-supermodules
\[0\to \mathrm{Lie}(\mathbb{H})(A)\stackrel{\mathrm{d}_e{\bf i}(A)}{\to} \mathrm{Lie}(\mathbb{G})(A)\stackrel{\mathrm{d}_e{\bf p}(A)}{\to} \mathrm{Lie}(\mathbb{R})(A)\]
is exact for every $A\in\mathsf{SAlg}_{\Bbbk}$.
\end{lm}
\begin{proof}
The standard diagram chasing implies the first statement. The second statement follows by Lemma \ref{another interpretation of differential}.
\end{proof}
From now on, all group superschemes are locally algebraic. Following \cite{ccf, DG, zub2}, we denote the image of $x\otimes r\in\mathfrak{g}\otimes R$ in $\mathbb{G}(R[\epsilon_0, \epsilon_1])$ by $e^{\epsilon_{|x|+|r|} x\otimes r}$. If ${\bf f} : \mathbb{G}\to\mathbb{H}$ is a morphism of group superschemes, then 
\[{\bf f}(R[\epsilon_0, \epsilon_1])(e^{\epsilon_{|x|+|r|} x\otimes r})=e^{\epsilon_{|x|+|r|}\mathrm{d}_e {\bf f}(R)(x\otimes r)}=e^{\epsilon_{|x|+|r|}\mathrm{d}_e{\bf f}(\Bbbk)(x)\otimes r}.\]
In particular, if $V$ is a $\mathbb{G}$-supermodule with respect to a homomorphism ${\bf f} :
\mathbb{G}\to\mathrm{GL}(V)$, then we have
\[{\bf f}(R[\epsilon_0, \epsilon_1])(e^{\epsilon_{|x|+|r|} x\otimes r})=\mathrm{id}_V +\epsilon_{|x|+|r|}\mathrm{d}_e {\bf f}(\Bbbk)(x)\otimes r.\]
If we denote 
\[({\bf f}(R[\epsilon_0, \epsilon_1])(e^{\epsilon_{|x|+|r|} x\otimes r}))(v\otimes 1) \ \mbox{and} \ (\mathrm{d}_e{\bf f} (R)(x\otimes r))(v\otimes 1)\] by 
\[e^{\epsilon_{|x|+|r|} x\otimes r}\cdot (v\otimes 1) \ \mbox{and} \ (x\otimes r)\cdot (v\otimes 1),\] respectively, where $v\in V$, then the above formula can be recorded as
\[e^{\epsilon_{|x|+|r|} x\otimes r}\cdot (v\otimes 1)=v\otimes 1+(-1)^{|r||v|}\epsilon_{|x|} (x\cdot v)\otimes r.\]

Consider the category of pairs $(\mathbb{G}, V)$, where $\mathbb G$ is a group superscheme, and $V$ is a $\mathbb{G}$-supermodule. The morphisms in this category are couples $({\bf f}, h)$, where ${\bf f} : \mathbb{G}\to\mathbb{H}$ is a group superscheme morphism and $h : V\to W$ is a linear map of superspaces, such that for every $R\in\mathsf{SAlg}_K$ the diagram
\[\begin{array}{ccc}
\mathbb{G}(R)\times\mathbb{V}_a(R) & \to & \mathbb{V}_a(R) \\
\downarrow & & \downarrow \\
\mathbb{H}\times\mathbb{W}_a(R) & \to & \mathbb{W}_a(R)
\end{array}\] 
is commutative. Here the horizontal maps correspond to the actions of $\mathbb{G}(R)$ and $\mathbb{H}(R)$ on $\mathbb{V}_a(R)$ and $\mathbb{W}_a(R)$, respectively, the right vertical map is ${\bf h}_a(R)=h\otimes\mathrm{id}_R$, and the left vertical map is ${\bf f}(R)\times {\bf h}_a(R)$.
\begin{lm}\label{differential of morphism of pairs}
If $({\bf f}, h)$ is a morphism of pairs $(\mathbb{G}, V)\to (\mathbb{H}, W)$, then for every
$x\in\mathfrak{g}, r\in R$, and $v\in V$ there is
\[{\bf h}_a(R)((x\otimes r)\cdot (v\otimes 1))=(-1)^{|r||v|}(\mathrm{d}_e{\bf f}(\Bbbk)(x)\cdot h(v))\otimes r.\]
\end{lm}
\begin{proof}
We have 
\[{\bf h}_a(R)(e^{\epsilon_{|x|+|r|}x\otimes r}\cdot v)=h(v)\otimes 1+\epsilon_{|x|+|r|}{\bf h}_a(R)((x\otimes r)\cdot (v\otimes 1))=\]\[e^{\epsilon_{|x|+|r|}\mathrm{d}_e{\bf f}(\Bbbk)(x)\otimes r}\cdot (h(v)\otimes 1)=h(v)\otimes 1+(-1)^{|r||v|}\epsilon_{|x|+|r|}(\mathrm{d}_e{\bf f}(\Bbbk)(x)\cdot h(v))\otimes r,\]
proving the formula.
\end{proof}
Let $\mathrm{ad}$ denote $\mathrm{d}_e\mathrm{Ad} : \mathrm{Lie}(\mathbb{G})\to \mathrm{Lie}(\mathrm{GL}(\mathfrak{g}))=\mathfrak{gl}(\mathfrak{g})_a$.
For every $x, y\in\mathfrak{g}=\mathrm{Lie}(\mathbb{G})(\Bbbk)$ and $r, r'\in R$ we set $[x\otimes r, y\otimes r']=(\mathrm{ad}(R)(x\otimes r))(y\otimes r')$.
\begin{example}\label{Lie superalgebra gl(V)}
The Lie superalgebra of $\mathrm{GL}(V)$ is canonically isomorphic to $\mathfrak{gl}(V)=\mathrm{End}_{\Bbbk}(V)$, regarded as a superspace with $\mathfrak{gl}(V)_i=\{\phi\mid \phi(V_j)\subseteq V_{i+j\pmod 2}\}$ for  $i=0, 1$. Moreover, 
if $X, Y\in\mathfrak{gl}(V)$ and $r, r'\in R$, then 
\[ [X\otimes r, Y\otimes r']=(-1)^{|r||Y|}[X, Y]\otimes rr'=(-1)^{|r||Y|}(XY-(-1)^{|X||Y|}YX)\otimes rr'.\]
Indeed, we have 
\[(e^{\epsilon_{|X|+|r|}X\otimes r})\cdot (Y\otimes r')=(\mathrm{id}_V\otimes 1 +\epsilon_{|X|+|r|}X\otimes r)(Y\otimes r')(\mathrm{id}_V\otimes 1 -\epsilon_{|X|+|r|}X\otimes r)=\]
\[Y\otimes r'+\epsilon_{|X|+|r|}((-1)^{|r||Y|}XY\otimes rr'-(-1)^{(|X|+|r|)(|Y|+|r'|)+|r'||X|}YX\otimes r'r)=\]
\[Y\otimes r' +(-1)^{|r||Y|}\epsilon_{|X|+|r|}(XY-(-1)^{|X||Y|}YX)\otimes rr',\]
proving the formula.
\end{example}
\begin{lm}\label{differential is a Lie superalgebra morphism}
Let ${\bf f} : \mathbb{G}\to\mathbb{H}$ be a morphism of locally algebraic group superschemes. Let $\mathfrak{h}$ denote the Lie superalgebra of $\mathbb{H}$. Then
for every $x, y\in\mathfrak{g}$ and $r, r'\in R$ there is 
\[\mathrm{d}_e{\bf f}(R)([x\otimes r, y\otimes r'])=(-1)^{|r||y|}[\mathrm{d}_e{\bf f}(\Bbbk)(x), \mathrm{d}_e{\bf f}(\Bbbk)(y)]\otimes rr'.\]
In particular, $\mathfrak{g}\otimes R$ and $\mathfrak{h}\otimes R$ are Lie superalgebras for the operation $[ \ , \ ]$ and
$\mathrm{d}_e{\bf f}$ is a morphism of Lie superalgebra functors.
\end{lm}
\begin{proof}
We have
\[\mathrm{d}_e{\bf f}(R)\mathrm{Ad}(R)(g)=\mathrm{Ad}(R)({\bf f}(R)(g))\mathrm{d}_e{\bf f}(R)\]
for every superalgebra $R$ and $g\in\mathbb{G}(R)$. That is, $({\bf f}, \mathrm{d}_e{\bf f}(\Bbbk))$ is a morphism of pairs $(\mathbb{G}, \mathfrak{g})\to (\mathbb{H}, \mathfrak{h})$, where $\mathfrak{g}$ and $\mathfrak{h}$ are regarded as $\mathbb{G}$-supermodule and $\mathbb{H}$-supermodule with respect to the adjoint actions. Lemma \ref{differential of morphism of pairs} implies the first statement.

To prove the second statement, one needs to show that the operation $[ \ , \ ]$ on $\mathfrak{g}$ satisfies the identities $(B2)$, $(B3)$, and $(B4)$ from \cite{masshib} (recall that $char(\Bbbk)\neq 2$). Applying the first statement to $\mathrm{Ad} : \mathbb{G}\to\mathrm{GL}(\mathfrak{g})$, one obtains
\[[[x, y], z]=(XY-(-1)^{|x||y|}YX)(z),\]
where $X=\mathrm{ad}(\Bbbk)(x), \ Y=\mathrm{ad}(\Bbbk)(y)$. Therefore, $(B4)$ (or the \emph{super Jacobi identity}) 
\[[[x, y], z]=[x, [y, z]]-(-1)^{|x||y|} [y, [x, z]]\]
follows.

Next, let $R$ denote $\Bbbk[\epsilon'_0, \epsilon'_1]$. Then the \emph{group commutator}
\[[e^{\epsilon_{|x|}x}, e^{\epsilon'_{|y|}y}]=e^{\epsilon_{|x|}x}e^{\epsilon'_{|y|}y}e^{-\epsilon_{|x|}x}e^{-\epsilon'_{|y|}y},\]
calculated in $\mathbb{G}(R[\epsilon_0, \epsilon_1])$, is equal to
\[e^{\mathrm{Ad}(R[\epsilon_0, \epsilon_1])(e^{\epsilon_{|x|}x})((-1)^{|y|}y\otimes \epsilon'_{|y|})}e^{-\epsilon'_{|y|}y}=e^{\epsilon'_{|y|}y+(-1)^{|x||y|}\epsilon_{|x|}\epsilon'_{|y|}(\mathrm{ad}(\Bbbk))(x))(y)}e^{-\epsilon'_{|y|}y}=\]\[e^{(-1)^{|x||y|}\epsilon_{|x|}\epsilon'_{|y|}[x, y]}.\]
The group identity $[g, h]=[h, g]^{-1}$ implies $\epsilon_{|x|}\epsilon'_{|y|}[x, y]=-\epsilon'_{|y|}\epsilon_{|x|}[y, x]$, hence $(B3)$ (or the \emph{super skew-symmetry}) $[x, y]=-(-1)^{|x||y|}[y, x]$ follows. 

The identity $(B2)$ states $[[x, x], x]=0$ for every $x\in\mathfrak{g}_1$. If $char(\Bbbk)\neq 3$, then $(B2)$ follows from $(B4)$. If $char(\Bbbk)=3$, then $(B4)$ does not imply $(B2)$. However, in Lemma \ref{(B2) for any g}, we prove the identity $(B2)$ in general. This finishes the proof.
\end{proof}

\begin{lm}\label{adjoint action commutes with super-bracket}
The adjoint action of $\mathbb{G}$ on $\mathrm{Lie}(\mathbb{G})$ commutes with the super-bracket functor.
\end{lm}
\begin{proof}
For every $g\in \mathbb{G}(R)$ we have 
\[\mathrm{Ad}(R)(g)e^{(-1)^{|x||y|}\epsilon_{|x|}\epsilon'_{|y|}[x, y]}=e^{(-1)^{|x||y|}\epsilon_{|x|}\epsilon'_{|y|}\mathrm{Ad}(R)(g)[x, y]}=\]\[\mathrm{Ad}(R)(g)[e^{\epsilon_{|x|}x}, e^{\epsilon'_{|y|}y}]=[\mathrm{Ad}(R)(g)e^{\epsilon_{|x|}x}, \mathrm{Ad}(R)(g)e^{\epsilon'_{|y|}y}]=\]
\[[e^{\epsilon_{|x|}\mathrm{Ad}(R)(g)x} , e^{\epsilon'_{|y|}\mathrm{Ad}(R)(g)y}]=e^{(-1)^{|x||y|}\epsilon_{|x|}\epsilon'_{|y|}[\mathrm{Ad}(R)(g)x, \mathrm{Ad}(R)(g)y]},\]
proving our statement.
\end{proof}

\section{The formal neighborhood of the identity}

Let $G$ be a geometric group superscheme. Let $\epsilon$ denote the identity morphism $e\to G$. We have a commutative diagram
\[\begin{array}{ccccc}
& & G & & \\
& & \uparrow & & \\
& & G\times G & &  \\
 & \swarrow p_1 & \uparrow & p_2 \searrow  &  \\
G & \stackrel{\epsilon}{\leftarrow} & e & \stackrel{\epsilon}{\rightarrow} & G 
\end{array},\]
where the closed immersion $e\to G\times G$ is induced by the universality of the direct product. For every affine open neighborhood $U$ of $e$, the immersion $e\to G\times G$  factors through $U\times U\subseteq G\times G$. If we denote the image of the point $e$ in $(G\times G)^e$ by $e\times e$, then $\mathcal{O}_{e\times e}$ is naturally isomorphic to $(\mathcal{O}_e\otimes\mathcal{O}_e)_{\mathfrak{n}_{e\times e}}$, where $\mathfrak{n}_{e\times e}={\mathcal{O}_e\otimes\mathfrak{m}_e+\mathfrak{m}_e\otimes\mathcal{O}_e}$. Moreover, $m^*_{e\times e}$ is a local superalgebra morphism from $\mathcal{O}_e$ to $(\mathcal{O}_e\otimes\mathcal{O}_e)_{\mathfrak{n}_{e\times e}}$.

For any non-negative integer $k$, let  $N_k(G)\simeq\mathrm{SSpec}(\mathcal{O}_e/\mathfrak{m}_e^{k+1})$ denote the $k$-th neighborhood of $\epsilon$.
\begin{lm}\label{a certain super-subgroup}
For any non-negative integers $k$ and $t$, there are commutative diagrams
\[\begin{array}{ccc}
G\times G & \stackrel{m}{\to} & G \\
\uparrow & & \uparrow \\
N_k(G)\times N_t(G) & \to & N_{k+t}(G)
\end{array}\]
and
\[\begin{array}{ccc}
G & \stackrel{\iota}{\to} & G \\
\uparrow & & \uparrow \\
N_k(G) & \to & N_k(G)
\end{array}.\]
\end{lm}
\begin{proof}
The local superalgebra $\mathcal{O}_e/\mathfrak{m}_e^{k+1}\otimes \mathcal{O}_e/\mathfrak{m}_e^{t+1}$ is canonically isomorphic to 
$(\mathcal{O}_e\otimes\mathcal{O}_e)_{\mathfrak{n}_{e\times e}}/({\mathcal{O}_e\otimes\mathfrak{m}^{k+1}_e+\mathfrak{m}^{t+1}_e\otimes\mathcal{O}_e})_{\mathfrak{n}_{e\times e}}$. Moreover, $m^*_{e\times e}$ induces
a local superalgebra morphism 
\[\mathcal{O}_e/\mathfrak{m}_e^{k+t+1}\to (\mathcal{O}_e\otimes\mathcal{O}_e)_{\mathfrak{n}_{e\times e}}/(\mathfrak{m}^{k+1}_e\otimes\mathcal{O}_e+\mathcal{O}_e\otimes\mathfrak{m}^{t+1})_{\mathfrak{n}_{e\times e}}\]
that makes the diagram
\[\begin{array}{ccc}
\mathcal{O}_{e\times e} & \stackrel{m^*_{e\times e}}{\leftarrow} & \mathcal{O}_e \\
\downarrow & & \downarrow \\
\mathcal{O}_e/\mathfrak{m}_e^{k+1}\otimes \mathcal{O}_e/\mathfrak{m}_e^{t+1} & \leftarrow & \mathcal{O}_e/\mathfrak{m}_e^{k+t+1}
\end{array}\]
 commutative. The proof of the second statement is analogous.
\end{proof}

Translating to the category $\mathcal{SFG}$, one sees that the group superscheme $\mathbb{G}$
contains an ascending chain of closed super-subschemes $\mathbb{N}_0(\mathbb{G})\subseteq\mathbb{N}_1(\mathbb{G})\subseteq \ldots$, such that each $\mathbb{N}_k(\mathbb{G})$ is isomorphic to $\mathrm{SSp}(\mathcal{O}_e/\mathfrak{m}_e^{k+1})$. Moreover, the terms  satisfy $\mathbb{N}_k(\mathbb{G})^{-1}\subseteq\mathbb{N}_k(\mathbb{G})$ and $\mathbb{N}_k(\mathbb{G})\mathbb{N}_t(\mathbb{G})\subseteq\mathbb{N}_{k+t}(\mathbb{G})$ for any non-negative integers $k$ and $t$. Therefore, $\mathbb{N}(\mathbb{G})=\cup_{k\geq 0}\mathbb{N}_k(\mathbb{G})$ is a group subfunctor of $\mathbb{G}$ such that for every $R\in\mathsf{SAlg}_{\Bbbk}$, the group $\mathbb{N}(\mathbb{G})(R)$ consists of all superalgebra morphisms $\mathcal{O}_e\to R$ vanishing on some power of $\mathfrak{m}_e$. 
We call $\mathbb{N}(\mathbb{G})$ the \emph{formal neighborhood} of the identity in
$\mathbb{G}$. 

The group functor $\mathbb{N}(\mathbb{G})$ is "quasi-affine" in the following sense. The complete
local superalgebra $\widehat{\mathcal{O}_e}=\varprojlim_n \mathcal{O}_e/\mathfrak{m}_e^{n+1}$ is a \emph{complete Hopf superalgebra} for the comultiplication $\Delta : \widehat{\mathcal{O}_e}\to \widehat{\mathcal{O}_e}\widehat{\otimes} \widehat{\mathcal{O}_e}$ and the antipode $S : \widehat{\mathcal{O}_e}\to \widehat{\mathcal{O}_e}$, induced by $m^*_{e\times e}$ and $\iota^*_e$, respectively (see \cite{homastak}, Definition 3.3). Moreover, there is \[\Delta(\widehat{\mathfrak{m}_e})\subseteq
\widehat{\mathcal{O}_e}\widehat{\otimes} \widehat{\mathfrak{m}_e}+\widehat{\mathfrak{m}_e}\widehat{\otimes} \widehat{\mathcal{O}_e} \ \mbox{and} \ S(\widehat{\mathfrak{m}_e})\subseteq \widehat{\mathfrak{m}_e},\]
that is, $\widehat{\mathfrak{m}_e}$ is a (closed) Hopf superideal of $\widehat{\mathcal{O}_e}$. 

For every $R\in\mathsf{SAlg}_{\Bbbk}$, the group $\mathbb{N}(\mathbb{G})(R)$ consists of all continuous superalgebra morphisms $\widehat{\mathcal{O}_e}\to R$, where $R$ is regarded as a discrete superalgebra. The group operations of $\mathbb{N}(\mathbb{G})(R)$ are defined by $(gh)(a)=(g\otimes h)(\Delta(a))$ and $g^{-1}(a)=g(S(a))$ for $g, h\in\mathbb{N}(\mathbb{G})(R)$ and $a\in \widehat{\mathcal{O}_e}$. We use the "Sweedler notation" and write $\Delta(a)= a_{(0)}\otimes a_{(1)}$, by omitting the summation symbol. The sum on the left is convergent in the $\mathfrak{n}_{e\times e}$-adic topology. Then $(gh)(a)=g(a_{(0)}) h(a_{(1)})$, where the sum on the right contains only finitely many non-zero summands. 

Assume again that $\mathbb{G}$ is locally algebraic. For a non-negative integer $n$, let $\mathrm{hyp}_n(\mathbb{G})$ denote the superspace $(\widehat{\mathcal{O}_e}/\widehat{\mathfrak{m}_e}^{n+1})^*$. Set $\mathrm{hyp}(\mathbb{G})=\cup_{n\geq 0}\mathrm{hyp}_n(\mathbb{G})\subseteq (\widehat{\mathcal{O}_e})^*$.
The superspace $\mathrm{hyp}(\mathbb{G})$ has a natural Hopf superalgebra structure with respect to
the product
\[(\phi\psi)(a)= (-1)^{|\psi||a_{(0)}|}\phi(a_{(0)})\psi(a_{(1)}),\] 
the coproduct $\Delta^*(\phi)=\phi_{(0)}\otimes\phi_{(1)}$, uniquely defined by the identity
\[\phi(ab)= (-1)^{|\phi_{(1)}||a|}\phi_{(0)}(a)\phi_{(1)}(b) \text{ for } \phi, \psi\in\mathrm{hyp}(\mathbb{G})
\text{ and } a, b\in \widehat{\mathcal{O}_e},\]
and the antipode $S^*$, such that $(S^*(\phi))(a)=\phi(S(a))$. 
Furthermore, for every $R\in\mathsf{SAlg}_{\Bbbk}$, $\widehat{\mathcal{O}_e}\otimes R=\varprojlim_n
(\mathcal{O}_e\otimes R)/(\mathfrak{m}_e\otimes R)^{n+1}$ is a complete Hopf $R$-superalgebra with respect to the comultiplication $\Delta\otimes\mathrm{id}_R$ and the antipode $S\otimes\mathrm{id}_R$, and $\mathrm{hyp}(\mathbb{G})\otimes R$ has the unique structure of Hopf $R$-superalgebra  for every $R\in\mathsf{SAlg}_{\Bbbk}$ such that the pairing $(\mathrm{hyp}(\mathbb{G}))\otimes  R)\times (\widehat{\mathcal{O}_e}\otimes R)\to R$ given by 
$<\phi\otimes r, a\otimes r'>=(-1)^{|r||a|}\phi(a)rr'$,
is a Hopf pairing in the sense of \cite{masshib}.

The following lemma is folklore.
\begin{lm}\label{standard properties of hyp}
For every non-negative integers $k$ and $t$, there is:
\begin{enumerate}
\item $\mathrm{hyp}_k(\mathbb{G})\mathrm{hyp}_t(\mathbb{G})\subseteq\mathrm{hyp}_{k+t}(\mathbb{G})$,
\item $\Delta^*(\mathrm{hyp}_k(\mathbb{G}))\subseteq\oplus_{0\leq s\leq k}\mathrm{hyp}_s(\mathbb{G})\otimes \mathrm{hyp}_{k-s}(\mathbb{G})$,
\item $S^*(\mathrm{hyp}_k(\mathbb{G}))\subseteq \mathrm{hyp}_k(\mathbb{G})$.
\end{enumerate}
\end{lm}
For a Hopf superalgebra $H$, let $\mathsf{Gpl}(H)$ denote the subgroup of $H^{\times}$ consisting of all (even) group-like elements of $H$.
Let $\mathfrak{n}_{\mathbb{G}}$ denote the Lie superalgebra of $\mathbb{N}(\mathbb{G})$.
\begin{lm}\label{super-hyperalgebra}
The group functor $\mathbb{N}(\mathbb{G})$ is canonically identified with the group functor $R\mapsto \mathsf{Gpl}(\mathrm{hyp}(\mathbb{G})\otimes R)$.
Additionally, we have $\mathfrak{n}_{\mathbb{G}}\simeq\mathrm{hyp}_1(\mathbb{G})^+=\{\phi\in\mathrm{hyp}_1(\mathbb{G})\mid \phi(1)=0\}$ and the Lie super-bracket on $\mathfrak{n}_{\mathbb{G} }$ is defined by $[\phi, \psi]=\phi\psi-(-1)^{|\phi||\psi|}\psi\phi$. In particular, this operation satisfies the identity $(B2)$.
\end{lm}
\begin{proof}
We identify the superspace $\mathrm{hyp}(\mathbb{G})\otimes R$ with $\varinjlim_n \mathrm{Hom}_K(\widehat{\mathcal{O}_e}/\mathfrak{m}_e^{n+1}, R)$ via the above pairing.
Then the first statement follows. 

Furthermore, the previous dual super-numbers technique shows that the Lie superalgebra $\mathfrak{n}_{\mathbb{G}  }$ can be identified with the super-subspace $\mathsf{P}$ of $\mathrm{hyp}(\mathbb{G})^+$ consisting of all primitive elements. Also, $\dim\mathfrak{n}_{\mathbb{G}}=\dim\mathfrak{g}$ (see Lemma \ref{(B2) for any g} below) and $\mathrm{hyp}_1(\mathbb{G})^+\subseteq\mathsf{P}$.

Since $\mathbb{N}(\mathbb{G})$ is "represented" by the complete Hopf superalgebra $\widehat{\mathcal{O}_e}$,  to prove the last statement, one can mimic the calculation from \cite[Lemma 3.4]{zub2}. Finally, $[[\phi, \phi], \phi]=[2\phi^2, \phi]=0$ for every $\phi\in\mathfrak{s}_1$. 
\end{proof}
\begin{lm}\label{(B2) for any g}
Let $\mathbb{G}$ be a locally algebraic group superscheme. Then the operation $[ \ , \ ]$ on $\mathfrak{g}$ satisfies the identity $(B2)$.
\end{lm}
\begin{proof}
Since each immersion $\mathbb{N}_k(\mathbb{G})\to\mathbb{G}$ factors through any open affine neighborhood $\mathbb{U}$ of $e=e_{\Bbbk}\in\mathbb{G}(\Bbbk)$, we have the following commutative diagram
\[\begin{array}{ccccccccc}
1 & \to & \mathrm{Lie}(\mathbb{G})(R) & \to & \mathbb{G}(R[\epsilon_0, \epsilon_1]) & \to & \mathbb{G}(R) & \to & 1\\
 & & \parallel & & \uparrow & & \uparrow & & \\
 & & (\mathfrak{m}_e/\mathfrak{m}^2_e)^*\otimes R & \to & \mathbb{N}_k(\mathbb{G})(R[\epsilon_0, \epsilon_1]) & \to &
\mathbb{N}_k(\mathbb{G})(R) & & \\
& & \parallel & &  \uparrow & & \uparrow & & \\ 
& & (\mathfrak{m}_e/\mathfrak{m}^2_e)^*\otimes R & \to & \mathbb{N}_t(\mathbb{G})(R[\epsilon_0, \epsilon_1]) & \to &
\mathbb{N}_t(\mathbb{G})(R) & &
\end{array},\]
where $t\leq k$. Using Lemma \ref{a certain super-subgroup}, we obtain a commutative diagram of groups
\[\begin{array}{ccccccccc}
1 & \to & \mathrm{Lie}(\mathbb{G})(R) & \to & \mathbb{G}(R[\epsilon_0, \epsilon_1]) & \to & \mathbb{G}(R) & \to & 1\\
 & & \parallel & & \uparrow & & \uparrow & & \\
1 & \to & (\mathfrak{m}_e/\mathfrak{m}^2_e)^*\otimes R & \to & \mathbb{N}(\mathbb{G})(R[\epsilon_0, \epsilon_1]) & \to &
\mathbb{N}(\mathbb{G})(R) & \to & 1 
\end{array}.\] 
Thus $\mathfrak{g}=\mathfrak{n}_{\mathbb{G}}$ and using the same trick with group commutators as in Lemma \ref{differential is a Lie superalgebra morphism}, one can derive that the Lie super-bracket on $\mathfrak{g}$ coincides with the Lie super-bracket on $\mathfrak{n}_{\mathbb{G}}$. Lemma \ref{super-hyperalgebra} concludes the proof.
\end{proof}
Let the $\Bbbk$-functor morphism ${\bf c} : \mathbb{G}\times \mathbb{G}\to \mathbb{G}$ defines the conjugation action of $\mathbb{G}$ on itself. In other words, for any $R\in\mathsf{SAlg}_{\Bbbk}$ and $g, h\in\mathbb{G}(R)$ there is
${\bf c}(R)(g, h)=ghg^{-1}$.  
\begin{lm}\label{formal neighborhood is normal}
Each $\mathbb{N}_k(\mathbb{G})$ is invariant with respect to the conjugation action of $\mathbb{G}$ on itself.
Therefore, $\mathbb{N}(\mathbb{G})$ is a normal group subfunctor of $\mathbb{G}$. 
\end{lm}
\begin{proof}
We use the equivalent language of geometric superschemes. The conjugation action is defined by a superscheme morphism $ c : G\times G\to G$. Let $U\simeq\mathrm{SSpec}(A)$ be an affine neighborhood of $e$.  For every affine open super-subscheme $V\simeq\mathrm{SSpec}(B)$, there is a commutative diagram
\[\begin{array}{ccc}
V\times U & \stackrel{c|_{V\times U}}{\to} & G \\
\uparrow & & \uparrow \\
V\times e & \stackrel{c|_{V\times e}}{\to} & e
\end{array}.\]
Then $c^{-1}|_{V\times U}(U)=(V\times U)\cap c^{-1}(U)$ can be covered by open affine supersubschemes $\mathrm{SSpec}((B\otimes A)_{g})$. There is a dual commutative diagram 
\[\begin{array}{ccc}
(B\otimes A)_g & \stackrel{c^*}{\leftarrow} & A \\
\downarrow & & \downarrow \\
B_{\overline{g}} & \leftarrow & \Bbbk
\end{array},\]
where the vertical arrows are $\mathrm{id}_B\otimes\epsilon_A$ and $\epsilon_A$, respectively, and 
$\overline{g}=(\mathrm{id}_B\otimes\epsilon_A)(g)$. Since localization is a faithful functor, we obtain $c^*(\mathfrak{m}_A)\subseteq (B\otimes\mathfrak{m}_A)_g$. Since $c^*(\mathfrak{m}^{k+1}_A)\subseteq (B\otimes\mathfrak{m}_A)^{k+1}_g=(B\otimes\mathfrak{m}^{k+1}_A)_g$, $c^*$ induces a superalgebra morphism $(B\otimes A/\mathfrak{m}_A^{k+1})_g \leftarrow  A/\mathfrak{m}^{k+1}_A$. Considering all $g$ and $V$, one sees that $c$ sends $G\times N_k(G)$ to $N_k(G)$, which proves the lemma.
\end{proof}
\begin{lm}\label{if one factor is n-th neighborhood}
Let $X$ be a geometric superscheme. For every finite-dimensional superalgebra $A$, we have 
$\mathcal{O}(X\times\mathrm{SSpec}(A))\simeq\mathcal{O}(X)\otimes A$.
\end{lm}
\begin{proof}
Let $Y$ denote $\mathrm{SSpec}(A)$. Choose a covering of $X$ by affine open super-subschemes $U_i$. Then $V_i=U_i\times Y$ is an open covering of $X\times Y$ by affine super-subschemes. If the elements $a_1, \ldots , a_k$ form a basis of $A$, then each element $f\in \mathcal{O}(V_i)\simeq\mathcal{O}(U_i)\otimes A$ can be uniquely expressed as $\sum_{1\leq j\leq k} f_j\otimes a_j$. Moreover, if for some $f'\in\mathcal{O}(V_s)$ we have $f|_{V_i\cap V_s}=f'|_{V_i\cap V_s}$, then $f_j|_{W}=f'_j|_W$ for any open affine super-subscheme $W$ of $U_i\cap U_s$ and $1\leq j\leq k$, that is, $f_j|_{U_i\cap U_s}=f'_j|_{U_i\cap U_s}$ for every $1\leq j\leq k$. Thus, our lemma follows.  
\end{proof}
\begin{lm}\label{affinization of conjugation action}
The conjugation action of $\mathbb{G}$ on $\mathbb{N}(\mathbb{G})$ factors through the action of $\mathbb{G}^{aff}$ on $\mathbb{N}(\mathbb{G})$.
\end{lm}
\begin{proof}
For simplicity, we use $N_k$ instead of $N_k(G)$ and $\mathbb{N}$ instead of $\mathbb{N}(\mathbb{G})$. Let $c_k$ denote $c|_{G\times N_k}$.  
By Lemma \ref{if one factor is n-th neighborhood}, $c_k^*$ maps $\mathcal{O}(N_k)$ to $\mathcal{O}(G)\otimes\mathcal{O}(N_k)$, that is, in Sweedler's notation, $c_k^*(a)=a_{(0)}\otimes a_{(1)}$ for $a\in\mathcal{O}(N_k)$. The superspace generated by the coefficients $a_{(0)}$, where $a$ ranges over $\mathcal{O}(N_k)$, is called the \emph{coefficient superspace} of $\mathcal{O}(N_k)$, and it is denoted by $\mathrm{cf}(\mathcal{O}(N_k))$. 

The fact that $c_k$ defines an action implies the commutativity of the diagram 
\[\begin{array}{cccccc}
\mathcal{O}(G)\otimes\mathcal{O}(N_k) & \stackrel{m^*\otimes \mathrm{id}_{\mathcal{O}(N_k)}}{\to} & \mathcal{O}(G\times G)\otimes\mathcal{O}(N_k) & \stackrel{\beta^*}{\leftarrow} & \mathcal{O}(G\times (G\times N_k)) \\  
\uparrow & & & & \uparrow \\
\mathcal{O}(N_k) & & \stackrel{c_k^*}{\longrightarrow} & & \mathcal{O}(G)\otimes\mathcal{O}(N_k)
\end{array},\]
where the vertical arrows are $c_k^*$ and $(\mathrm{id}_G\times c_k)^*$, respectively, and $\beta$ is the canonical isomorphism $(G\times G)\times N_k\to G\times (G\times N_k)$. Arguing as in 
Proposition \ref{C is Hopf superalgebra}, one sees that $(\mathrm{id}_G\times c_k)^*$ factors through 
$\mathrm{id}_{\mathcal{O}(G)}\otimes c_k^*$. Moreover, the composition of the morphism 
\[\mathcal{O}(G)^{\otimes 2}\otimes\mathcal{O}(N_k)\simeq\mathcal{O}(G)\otimes \mathcal{O}(G\times N_k)\to\mathcal{O}(G\times (G\times N_k))\]
with $\beta^*$ can be identified with $p_1^*\otimes p_2^*\otimes\mathrm{id}_{\mathcal{O}(N_k)}$.
Therefore, $\mathrm{cf}(\mathcal{O}(N_k))\subseteq\mathcal{O}(G^{aff})$, and
$\mathcal{O}(N_k)=\mathcal{O}_e/\mathfrak{m}_e^{k+1}$ is a left $\mathcal{O}(G^{aff})$-coideal superalgebra (cf. \cite{zub2}).
Furthermore, for any $t\geq k$, there is a commutative diagram
\[\begin{array}{ccc}
\mathcal{O}_e/\mathfrak{m}_e^{t+1} & \stackrel{c^*_t}{\to} & \mathcal{O}(G^{aff})\otimes \mathcal{O}_e/\mathfrak{m}_e^{t+1} \\
\downarrow & & \downarrow \\
\mathcal{O}_e/\mathfrak{m}_e^{k+1} & \stackrel{c^*_k}{\to} & \mathcal{O}(G^{aff})\otimes \mathcal{O}_e/\mathfrak{m}_e^{k+1},
\end{array}
\]
hence $c^*=\varprojlim c^*_k$ defines a structure of left $\mathcal{O}(G^{aff})$-coideal superalgebra on
$\widehat{\mathcal{O}_e}$. In other words, the conjugation action of $\mathbb{G}$ on $\mathbb{N}$ is defined by
\[(g\cdot h)(a)=(\overline{g}\cdot h)(a)=\overline{g}(a_{(0)})h(a_{(1)}),\]
where  $g\in\mathbb{G}(R), h\in\mathbb{N}(R), a\in\widehat{\mathcal{O}_e}, c^*(a)=a_{(0)}\otimes a_{(1)}, R\in\mathsf{SAlg}_{\Bbbk}$, and $\overline{g}$ is the image of $g$ under the group homomorphism $\mathbb{G}(R)\to\mathbb{G}^{aff}(R)$. 
\end{proof}
Since $\widehat{\mathcal{O}_e}$ is a left $\mathcal{O}(G^{aff})$-coideal superalgebra, $\mathbb{G}(R)$ acts on $\widehat{\mathcal{O}_e}\otimes R$ on the right by $R$-superalgebra automorphisms
\[(a\otimes r)\cdot g=\overline{g}(a_{(0)}) a_{(1)}\otimes r \text{ for } g\in\mathbb{G}(R) \text{ and } r\in R,\]
and this action is functorial in $R$. We also have
\[<g\cdot h, x>=<h, x\cdot g> \text{ for } x\in\widehat{\mathcal{O}_e}\otimes R, \] where $h\in\mathbb{N}(R)$ is regarded as a group-like element of $\mathrm{hyp}(\mathbb{G})\otimes R$. 
\begin{lm}\label{conjugation action on certain superalgebras}
The group superscheme $\mathbb{G}$ acts on  $\widehat{\mathcal{O}_e}$ by Hopf superalgebra automorphisms.
Moreover, if we define its action on $\mathrm{hyp}(\mathbb{G})$ by the above formula
\[<g\cdot \phi, a>=<\phi, a\cdot g> \text{ for } g\in\mathbb{G}(R), \phi\in \mathrm{hyp}(\mathbb{G})\otimes R, \text{ and }a\in \widehat{\mathcal{O}_e}\otimes R,\]
then this action also preserves the Hopf superalgebra structure of $\mathrm{hyp}(\mathbb{G})\otimes R$ and is functorial in $R$.
\end{lm}
\begin{proof}
Since $< \ , \ >$ is a Hopf duality, the second statement is obvious.

We have a commutative diagram
\[\begin{array}{ccc}
\mathbb{G}^{aff}\times\mathbb{N} & \stackrel{{\bf c}}{\to} & \mathbb{N} \\
 & &  \uparrow\\
\uparrow & & \mathbb{N}\times\mathbb{N} \\
 & & \uparrow \\
 \mathbb{G}^{aff}\times\mathbb{N}\times\mathbb{N} & \to & (\mathbb{G}^{aff}\times\mathbb{N})\times (\mathbb{G}^{aff}\times\mathbb{N}), 
\end{array}\]
where the lower horizontal arrow is defined as $(g, s, s')\mapsto (g, s, g, s')$, the left vertical arrow is $\mathrm{id}_{\mathbb{G}^{aff}}\times {\bf m}$, and the right vertical arrows are ${\bf c}\times {\bf c}$ and $\bf m$, read from the bottom to the top, respectively. The commutativity of the corresponding diagram of superalgebras 
\[\begin{array}{ccc}
\mathcal{O}(G^{aff})\otimes\widehat{\mathcal{O}_e} & \stackrel{c^*}{\leftarrow} & \widehat{\mathcal{O}_e} \\
 & &  \downarrow\\
\downarrow & & \widehat{\mathcal{O}_e}\otimes\widehat{\mathcal{O}_e}\\
 & & \downarrow \\
 \mathcal{O}(G^{aff})\otimes\widehat{\mathcal{O}_e}\otimes\widehat{\mathcal{O}_e} & \leftarrow & (\mathcal{O}(G^{aff})\otimes\widehat{\mathcal{O}_e})\otimes (\mathcal{O}(G^{aff})\otimes\widehat{\mathcal{O}_e}), 
\end{array}\]
is equivalent to the identity
\[(\star) \  a_{(0)}\otimes (a'_{(1)})'_{(0)}\otimes (a'_{(1)})'_{(1)}= (-1)^{|(a'_{(1)})_{(0)}||(a'_{(0)})_{(1)}|}(a'_{(0)})_{(0)}(a'_{(1)})_{(0)}\otimes (a'_{(0)})_{(1)}\otimes (a'_{(1)})_{(1)},\]
for any $a\in \widehat{\mathcal{O}_e}$, where 
\[\Delta(a)=a'_{(0)}\otimes a'_{(1)}, \ c^*(a)=a_{(0)}\otimes a_{(1)}, \ \Delta(a_{(1)})=(a_{(1)})'_{(0)}\otimes (a_{(1)})'_{(1)}, \] 
\[c^*(a'_{(0)})=(a'_{(0)})_{(0)}\otimes (a'_{(0)})_{(1)}, \ c^*(a'_{(1)})= (a'_{(1)})_{(0)}\otimes (a'_{(1)})_{(1)} .\]    
Applying $g\otimes\mathrm{id}_{\widehat{\mathcal{O}_e}}^{\otimes 2}$ to both parts of $(\star)$, we obtain $\Delta(a\cdot g)= a'_{(0)}\cdot g\otimes a_{(1)}'\cdot g$. 
\end{proof}
\begin{cor}\label{another definition of Ad}
The above action of $\mathbb{G}$ on $\mathfrak{n}_{\mathbb{G}}\simeq\mathrm{hyp}(\mathbb{G})_1^+$ coincides with $\mathrm{Ad}$.
\end{cor}
\begin{lm}\label{crucial property of formal neighborhood}
If $\mathbb{U}$ is an open super-subscheme of $\mathbb{G}$, then $\mathbb{N}\mathbb{U}\subseteq\mathbb{U}$,
and  $\mathbb{U}\mathbb{N}\subseteq\mathbb{U}$.
\end{lm} 
\begin{proof}
All one needs to show is that for any $k\geq 0$, the morphism $U\times N_k\stackrel{m}{\to} G$ factors through $U$. That is, $m^e$ maps $(U\times N_k)^e$ to $U^e$. Without losing generality, one can assume that $U$ is affine, say $U\simeq\mathrm{SSpec}(A)$. Then every prime super-ideal of $A\otimes\mathcal{O}_e/\mathfrak{m}_e^{k+1}$ has a form $\mathfrak{p}\otimes \mathcal{O}_e/\mathfrak{m}_e^{k+1}+A\otimes \mathfrak{m}_e/\mathfrak{m}_e^{k+1}$, where $\mathfrak{p}\in (\mathrm{SSpec}(A))^e$. Thus 
$(U\times N_k)^e=(U\times N_0)^e=(U\times e)^e$ and our statement follows by one of the group axioms. 
\end{proof}

\section{The functor $\mathsf{gr}$}

\subsection{Filtered superalgebras and their graded companions}

A \emph{(downward) filtered} superalgebra is a couple $(A, \mathsf{I}$), where $A$ is a superalgebra, and $\mathsf{I}$ is a super-ideal filtration
\[A=I_0\supseteq I_1\supseteq I_2\supseteq\ldots \]
of $A$ such that $I_k I_l\subseteq I_{k+l}$ for every $k, l\geq 0$. We associate with any filtered superalgebra $(A, \mathsf{I})$ a graded superalgebra $\mathsf{gr}_{\mathsf{I}}(A)=\oplus_{k\geq 0} I_k/I_{k+1}$.

If there is a super-ideal $I$ such that $I_k=I^k$ for each $k\geq 0$, then this filtration is called $I$-\emph{adic}, and the corresponding graded superalgebra is denoted by $\mathsf{gr}_I(A)$.

If $(B, \mathsf{J})$ is another filtered superalgebra, then the tensor product $A\otimes B$ is a filtered superalgebra with respect to the filtration
\[T_k=I_0\otimes J_k+I_1\otimes J_{k-1}+\ldots +I_k\otimes J_0 \text{ for } k\geq 0.\] 
We call the filtration $\mathsf{T}$ a \emph{tensor product} of filtrations $\mathsf{I}$ and $\mathsf{J}$.
 
A morphism of filtered superalgebras from $(A, \mathsf{I})$ to $(B, \mathsf{J})$, is a superalgebra morphism
$\phi : A\to B$ such that $\phi(I_k)\subseteq J_k$ for every $k\geq 0$. Any such morphism induces a superalgebra morphism $\mathsf{gr}_{\mathsf{I}}(A)\to \mathsf{gr}_{\mathsf{J}}(B)$.

If $(A, \mathsf{I})$ is a filtered superalgebra, then for every $n\geq 0$ the superalgebra $A/I_{n+1}$ has a finite filtration 
\[A/I_{n+1}=I_0/I_{n+1}\supseteq I_1/I_{n+1}\supseteq\ldots\supseteq I_n/I_{n+1}\supseteq 0,\]
which is denoted by $\mathsf{I}_{\leq n}$.  
\begin{pr}\label{gr commutes with tensor product}
There is a natural isomorphism
\[\mathsf{gr}_{\mathsf{I}}(A)\otimes \mathsf{gr}_{\mathsf{J}}(B)\simeq\mathsf{gr}_{\mathsf{T}}(A\otimes B),\]
induced by the canonical embeddings $A\to A\otimes B$ and $B\to A\otimes B$, which is functorial in both $A$ and $B$. 
\end{pr}
\begin{proof}
Both $A\to A\otimes B$ and $B\to A\otimes B$ are morphisms of filtered superalgebras.
Therefore, they induce the canonical morphisms of graded superalgebras
\[\mathsf{gr}_{\mathsf{I}}(A)\to \mathsf{gr}_{\mathsf{T}}(A\otimes B) \text{ and } \mathsf{gr}_{\mathsf{J}}(B)\to \mathsf{gr}_{\mathsf{T}}(A\otimes B)\]
given by 
\[a+I_{k+1}\mapsto a\otimes 1+T_{k+1} \text{ and } b+J_{l+1}\mapsto 1\otimes b+T_{l+1}.\] 
Thus, there is a superalgebra morphism
\[\mathsf{gr}_{\mathsf{I}}(A)\otimes \mathsf{gr}_{\mathsf{J}}(B)\to\mathsf{gr}_{\mathsf{T}}(A\otimes B)\]
which takes $(a+I_{k+1})\otimes (b+J_{l+1})$ to $a\otimes b +T_{k+l+1}$ for every $a\in I_k, b\in J_l$.

Moreover, if we define a graded superalgebra structure on $\mathsf{gr}_{\mathsf{I}}(A)\otimes \mathsf{gr}_{\mathsf{J}}(B)$ as 
\[(\mathsf{gr}_{\mathsf{I}}(A)\otimes \mathsf{gr}_{\mathsf{J}}(B))(n)=\oplus_{0\leq k\leq n} I_k/I_{k+1}\otimes J_{n-k}/J_{n-k+1} \text{ for } n\geq 0,\]
then the above map is graded superalgebra morphism.

For every $k\geq 0$, choose elements $a_{i, k}\in I_k$ and $b_{j, k}\in J_k$, where $i$ and $j$ run over index sets $S_k$  and $L_k$, respectively, such that they form bases of $I_k$ modulo $I_{k+1}$, and $J_k$ modulo $J_{k+1}$, respectively. Then the homogeneous component $(\mathsf{gr}_{\mathsf{I}}(A)\otimes \mathsf{gr}_{\mathsf{J}}(B))_n$ has a basis consisting of all elements 
\[(a_{i, k}+I_{k+1})\otimes (b_{j, n-k}+J_{n-k+1}) \text{ for } 0\leq k\leq n, i\in S_k \text{ and } j\in L_{n-k}.\]  
Moreover, the elements
\[a_{i, k}\otimes b_{j, n-k} \text{ for } 0\leq k\leq n, i\in S_k \text{ and }j\in L_{n-k}\]
generate $T_n$ modulo $T_{n+1}$. Since our morphism induces a one-to-one correspondence between the basis of $(\mathsf{gr}_{\mathsf{I}}(A)\otimes \mathsf{gr}_{\mathsf{J}}(B))_n$ and the generating set of $T_n/T_{n+1}$, all we need is to show that the elements $a_{i, k}\otimes b_{j, n-k}$ are linearly independent over $T_{n+1}$.  

Consider the filtered superalgebras $(A/I_{n+1}, \mathsf{I}_{\leq n})$ and $(B/J_{n+1}, \mathsf{J}_{\leq n})$. The elements $\overline{a_{i, k}}=a_{i, k}+I_{n+1}$ and $\overline{b_{j, l}}=b_{j, l}+I_{n+1}$, where $0\leq k, l\leq n, i\in S_k,$ and  $j\in L_l$, form bases of vector (super)spaces $A/I_{n+1}$ and $B/J_{n+1}$, respectively. The superalgebra 
\[A/I_{n+1}\otimes B/J_{n+1}\simeq (A\otimes B)/(I_{n+1}\otimes B +A\otimes J_{n+1})\]  
has a basis 
\[\overline{a_{i, k}}\otimes \overline{b_{j, l}} \text{ for } 0\leq k, l\leq n, i\in S_k \text{ and } j\in L_l.\]
Let $\mathsf{T}'$ denote the tensor product of filtrations $\mathsf{I}_{\leq n}$ and $\mathsf{J}_{\leq n}$. Then $T'_{2n+1}=0$ and for every $0\leq s\leq 2n$ the superspace $T'_s/T'_{s+1}$ is spanned by the part of the above basis consisting of the elements $\overline{a_{i, k}}\otimes \overline{b_{j, s-k}}$, where $k, s-k\leq n$. In particular, the latter elements are linearly independent over $T'_{s+1}$.   
If we apply this remark to $s=n$, and observe that $T'_n/T'_{n+1}$ is naturally isomorphic to $T_n/T_{n+1}$, the proof concludes.
\end{proof}

\subsection{Graded companions of geometric superschemes}

Recall that if $\mathcal{F}$ is a presheaf of abelian groups on a topological space $X$, then its sheafification is denoted by $\mathcal{F}^+$.

If the index set $I$ is infinite, then a direct sum of sheaves $\mathcal{F}=\oplus_{i\in I}\mathcal{F}_i$ is a presheaf but not necessary a sheaf (see \cite[Exercise II.1.10]{hart}). 

The direct product $\overline{\mathcal{F}}=\prod_{i\in I}\mathcal{F}_i$ is a sheaf on $X$. Let $\widehat{\mathcal{F}}$ denote a sub-presheaf of $\overline{\mathcal{F}}$ such that for every open subset $U\subseteq X$ a section $f\in\overline{\mathcal{F}}(U)$ belongs to $\widehat{\mathcal{F}}(U)$ if and only if there is an open covering $\{U_j\}_{j\in J}$ of $U$ with $f|_{U_j}\in\mathcal{F}(U_j)$ for each $j\in J$. The following lemma can be easily derived from \cite[Proposition-Definition II.1.2]{hart}.  
\begin{lm}\label{sheafification of  direct sum}
$\widehat{\mathcal{F}}$ is a sheaf isomorphic to $\mathcal{F}^+$.
\end{lm}
Let $X$ be a geometric superscheme. A quasi-coherent super-ideal sheaf $\mathcal{J}$ on $X$ is called locally nilpotent if for every  $x\in X^e$ the stalk $\mathcal{J}_x$ is a locally nilpotent super-ideal of
$\mathcal{O}_{X, x}$. Let $\mathsf{gr}_{\mathcal{J}}(\mathcal{O}_X)$  denote the sheaf of superalgebras
$(\oplus_{n\geq 0}\mathcal{J}^n/\mathcal{J}^{n+1})^+ .$
Recall that $\mathcal{I}_X$ is the super-ideal sheaf generated by $(\mathcal{O}_X)_1$.
\begin{pr}\label{graded superscheme}
The following statements hold:
\begin{enumerate}
\item A geometric superspace $\mathsf{gr}_{\mathcal{J}}(X)=(X^e, \mathsf{gr}_{\mathcal{J}}(\mathcal{O}_X))$ is a superscheme;
\item $X\to \mathsf{gr}_{\mathcal{I}_X}(X)$ is an endofunctor of the category $\mathcal{SV}$ that takes immersions to immersions;
\item A morphism $f : X\to Y$ of superschemes of locally finite type is an isomorphism if and only if $\mathsf{gr}(f) : \mathsf{gr}_{\mathcal{I}_X}(X)\to \mathsf{gr}_{\mathcal{I}_Y}(Y)$ is. 
\end{enumerate}
\end{pr}
\begin{proof}
Without loss of generality, one can assume that $X$ is affine, say $X=\mathrm{SSpec}(A)$. Then $\mathcal{J}=\widetilde{J}$, where $J$ is a super-ideal of $A$ such that $J_{\mathfrak{p}}$ is locally nilpotent for every $\mathfrak{p}\in (\mathrm{SSpec}(A))^e$. Therefore, $J$ is locally nilpotent.

Since $(\mathrm{SSpec}(\mathsf{gr}_J(A)))^e=(\mathrm{SSpec}(A/J))^e=(\mathrm{SSpec}(A))^e$, \cite[Proposition 2.1 (3)]{zub1} (see also \cite[Proposition II.5.2 (c)]{hart}) implies
\[\mathsf{gr}_{\mathcal{J}}(\mathcal{O}_X)\simeq (\oplus_{n\geq 0}\widetilde{J^n/J^{n+1}})^+\simeq \mathcal{O}_{\mathrm{SSpec}(\mathsf{gr}_J(A))},\]
hence $\mathsf{gr}_{\mathcal{J}}(X)\simeq\mathrm{SSpec}(\mathsf{gr}_J(A))$.

If $f : X\to Y$ is a morphism in $\mathcal{SV}$, then $f^*(\mathcal{I}_Y)\subseteq f_*^e\mathcal{I}_X$, and $f^*$ induces the required morphism of sheaves $\mathsf{gr}_{\mathcal{I}_Y}(\mathcal{O}_Y)\to f_*^e\mathsf{gr}_{\mathcal{I}_X}(\mathcal{O}_X)$. 

Since the functor $A\to\mathsf{gr}_{I_A}(A)$ commutes with localizations, for every $x\in X^e$ there is
\[\mathcal{O}_{\mathsf{gr}_{\mathcal{I}_X}(X), x}\simeq\mathsf{gr}_{I_{\mathcal{O}_x}}(\mathcal{O}_x),\]
and this isomorphism is functorial in $X$. Thus, the second statement follows.

Similarly, if $\mathsf{gr}(f)$ is an isomorphism, then $f^e$ is a homeomorphism of topological spaces and for every $x\in X^e$ the local morphism $f^*_e$ induces the isomorphism   
\[\mathsf{gr}_{I_{\mathcal{O}_{Y, y}}}(\mathcal{O}_{Y, y})\simeq \mathsf{gr}_{I_{\mathcal{O}_{X, x}}}(\mathcal{O}_{X, x}),\]
where $y=f^e(x)$. By \cite[Proposition 1.10]{maszub3}, the local superalgebras $\mathcal{O}_{X, x}$, and   
$\mathcal{O}_{Y, y}$ are Hausdorff spaces with respect to their $I_{\mathcal{O}_{X, x}}$-adic, and $I_{\mathcal{O}_{Y, y}}$-adic topologies, respectively. Hence $\mathcal{O}_{Y, y}\simeq \mathcal{O}_{X, x}$ and the proposition is proven.
\end{proof}
\begin{pr}\label{direct product and gr}
For every geometric superschemes $X$ and $Y$, there is an isomorphism
\[\mathsf{gr}_{\mathcal{I}_{X\times Y}}(X\times Y)\simeq \mathsf{gr}_{\mathcal{I}_X}(X)\times \mathsf{gr}_{\mathcal{I}_Y}(Y).\]
Moreover, for every morphisms $X\to Z$ and $Y\to T$ in $\mathcal{SV}$, the diagram
\[\begin{array}{ccc}
\mathsf{gr}_{\mathcal{I}_{X\times Y}}(X\times Y) & \simeq & \mathsf{gr}_{\mathcal{I}_X}(X)\times \mathsf{gr}_{\mathcal{I}_Y}(Y)\\
\downarrow & & \downarrow \\
\mathsf{gr}_{\mathcal{I}_{Z\times T}}(Z\times T) & \simeq & \mathsf{gr}_{\mathcal{I}_Z}(Z)\times \mathsf{gr}_{\mathcal{I}_T}(T)
\end{array}\]
is commutative.
\end{pr}
\begin{proof}
Let $U\simeq\mathrm{SSpec}(A)$ and $V\simeq\mathrm{SSpec}(B)$ be open super-subschemes of $X$ and $Y$, respectively. Let $p_X$ and $p_Y$ denote the canonical projections $X\times Y\to X$ and $X\times Y\to Y$ respectively. As we have already observed, $p_X^{-1}(U)\cap p_Y^{-1}(V)\simeq U\times V\simeq\mathrm{SSpec}(A\otimes B)$ is an open super-subscheme of $X\times Y$.
Proposition \ref{graded superscheme} and Proposition \ref{gr commutes with tensor product} imply that there is the unique open immersion 
\[\phi_{U V} : \mathsf{gr}_{\mathcal{I}_{U\times V}}(U\times V)\to \mathsf{gr}_{\mathcal{I}_{X}}(X)\times \mathsf{gr}_{\mathcal{I}_{Y}}(Y)\]
such that
\[p_{\mathsf{gr}_{\mathcal{I}_{X}}(X)}\phi_{U V}=\mathsf{gr}(p_{U}) \text{ and } \ p_{\mathsf{gr}_{\mathcal{I}_{Y}}(Y)}\phi_{U V}=\mathsf{gr}(p_{V}).\]
Besides, $\phi_{U V}$ induces an isomorphism onto the open affine super-subscheme $\mathsf{gr}_{\mathcal{I}_{U}}(U)\times \mathsf{gr}_{\mathcal{I}_{V}}(V)$ of $\mathsf{gr}_{\mathcal{I}_{X}}(X)\times \mathsf{gr}_{\mathcal{I}_{Y}}(Y)$. Furthermore, for every open affine super-subschemes $U'\subseteq U$ and 
$V'\subseteq V$, the composition of the natural open immersion
\[\mathsf{gr}_{\mathcal{I}_{U'\times V'}}(U'\times V')\to \mathsf{gr}_{\mathcal{I}_{U\times V}}(U\times V)\]
and $\phi_{U, V}$ coincides with $\phi_{U' V'}$. 

Using these remarks, one can construct a collection of open immersions
\[\phi_{U V} : \mathsf{gr}_{\mathcal{I}_{U\times V}}(U\times V)\to \mathsf{gr}_{\mathcal{I}_{X}}(X)\times \mathsf{gr}_{\mathcal{I}_{Y}}(Y)\]
compatible with each other, where $U$ and $V$ run over open affine coverings of $X$ and $Y$, respectively. Thus the first statement follows. 

The  same reduction to affine open coverings can be used to prove the second statement. We leave the details for the reader.
\end{proof}
From now on, the graded companion $\mathsf{gr}_{\mathcal{I}_X}(X)$ of a geometric superscheme $X$ is denoted just by $\mathsf{gr}(X)$. Similarly, $\mathsf{gr}_{I_A}(A)$ is denoted by $\mathsf{gr}(A)$.

\begin{pr}\label{gr and group objects}
$G\mapsto \mathsf{gr}(G)$ is an endofunctor of the category $\mathcal{SVG}$. Moreover, if $G$ is (locally) algebraic, then  $\mathsf{gr}(G)$ is.
\end{pr}
\begin{proof}
Let $m, \iota$, and $\epsilon$ denote the multiplication, the inverse, and the identity morphisms of $G$, respectively. 

Composition of $\mathsf{gr}(m)$ with the above isomorphism  $\mathsf{gr}(G\times G)\simeq\mathsf{gr}(G)\times\mathsf{gr}(G)$ defines the multiplication morphism $m' : \mathsf{gr}(G)\times\mathsf{gr}(G)\to \mathsf{gr}(G)$. Moreover, if $f : G\to H$ is a morphism of (geometric) group superschemes, then 
$\mathsf{gr}(f)$ is a morphism of (geometric) \emph{groupoid} superschemes.

Similarly, $\mathsf{gr}(\iota)$ and $\mathsf{gr}(\epsilon)$ are the inverse and the unit morphisms of $\mathsf{gr}(G)$, respectively. 
 
Using Proposition \ref{direct product and gr}, one obtains the following commutative diagram
\[\begin{array}{ccccc}
\mathsf{gr}(G\times G\times G) & \simeq & \mathsf{gr}(G)\times\mathsf{gr}(G\times G) & \simeq & \mathsf{gr}(G)\times \mathsf{gr}(G)\times \mathsf{gr}(G) \\
\downarrow & & \downarrow & & \downarrow  \\
\mathsf{gr}(G\times G) & \simeq & \mathsf{gr}(G)\times \mathsf{gr}(G) & = & \mathsf{gr}(G)\times \mathsf{gr}(G) \\
\downarrow & & \downarrow & & \downarrow  \\
\mathsf{gr}(G) & = & \mathsf{gr}(G) & = & \mathsf{gr}(G) , 
\end{array}\]
where the top vertical arrows are $\mathsf{gr}(\mathrm{id}_G\times m)$, $\mathrm{id}_{\mathsf{gr}(G)}\times\mathsf{gr}(m)$ and $\mathrm{id}_{\mathsf{gr}(G)}\times m'$, respectively, and the bottom vertical arrows are $\mathsf{gr}(m)$ and twice $m'$, respectively. Constructing the symmetric diagram for 
\[G\times G\times G\stackrel{m\times\mathrm{id}_G}{\to} G\times G\stackrel{m}{\to} G,\]
one can derive that $m'$ satisfies the axiom of associativity. We leave for the reader the routine checking of all remaining group axioms. 
\end{proof}
\begin{lm}\label{embedding in and projection onto even super-subscheme}
Let $X$ be a geometric superscheme. There are natural morphisms $i_X : X_{ev}\to\mathsf{gr}(X)$ and
$q_X : \mathsf{gr}(X)\to X_{ev}$, both functorial in $X$,  such that $q_X i_X=\mathrm{id}_{X_{ev}}$.  
Moreover, $i_X$ is a closed immersion that maps $X_{ev}$ isomorphically onto $\mathsf{gr}(X)_{ev}$.
\end{lm}
\begin{proof}
Since the underlying topological spaces of all the appearing superschemes are the same, all we need is to define morphisms of their superalgebra sheaves. We have two morphisms of superalgebra presheaves
\[\mathcal{O}_{X_{ev}}=\mathcal{O}_{X}/\mathcal{I}_X\to \oplus_{n\geq 0}\mathcal{I}_X^n/\mathcal{I}_X^{n+1}\]
and
\[\oplus_{n\geq 0}\mathcal{I}_X^n/\mathcal{I}_X^{n+1}\to \mathcal{O}_{X}/\mathcal{I}_X=\mathcal{O}_{X_{ev}},\]
where the first morphism is the isomorphism onto the $0$th component of $\oplus_{n\geq 0}\mathcal{I}_X^n/\mathcal{I}_X^{n+1}$ and the second morphism is the projection onto its $0$th component. 
These morphisms uniquely extend to morphisms of superalgebra sheaves
\[ q^*_X : \mathcal{O}_{X_{ev}}\to (\oplus_{n\geq 0}\mathcal{I}_X^n/\mathcal{I}_X^{n+1})^+=\mathcal{O}_{\mathsf{gr}(X)}\]
and
\[i_X^* : \mathcal{O}_{\mathsf{gr}(X)}=(\oplus_{n\geq 0}\mathcal{I}_X^n/\mathcal{I}_X^{n+1})^+\to\mathcal{O}_{X_{ev}}.\]
Since the composition of morphisms of presheaves is an identity map, we have $i_X^* q_X^*=\mathrm{id}_{\mathcal{O}_{X_{ev}}}$. The functoriality follows. 

Finally, since the last statement is local, one can assume that $X$ is affine, say $X=\mathrm{SSpec}(A)$. Then
$i_X$ is induced by the projection $\mathsf{gr}{A}\mapsto\overline{A}$ (and $q_X$ is induced by the embedding $\overline{A}\mapsto\mathsf{gr}(A)$). The lemma is proven.
\end{proof}
\begin{lm}\label{direct product and i and p}
For every geometric superschemes $X$ and $Y$, there are the following commutative diagrams
\[\begin{array}{ccc}
(X\times Y)_{ev} & \simeq & X_{ev}\times Y_{ev} \\
\downarrow & & \downarrow \\
\mathsf{gr}(X\times Y) & \simeq & \mathsf{gr}(X)\times \mathsf{gr}(Y) 
\end{array} \
\mbox{and} \
\begin{array}{ccc}
\mathsf{gr}(X\times Y) & \simeq & \mathsf{gr}(X)\times \mathsf{gr}(Y) \\
\downarrow & & \downarrow \\
(X\times Y)_{ev} & \simeq & X_{ev}\times Y_{ev}
\end{array},\]
where the vertical arrows in the first diagram are $i_{X\times Y}$ and $i_X\times i_Y$, while 
the vertical arrows in the second diagram are $q_{X\times Y}$ and $q_X\times q_Y$. 
\end{lm}
\begin{proof}
Arguing as in Proposition \ref{direct product and gr}, one can reduce the general case to $X=\mathrm{SSpec}(A)$ and $ Y=\mathrm{SSpec}(B)$. Then the commutativity of the above diagrams is equivalent to
the commutativity of the diagrams
\[\begin{array}{ccc}
\overline{A\otimes B} & \simeq & \overline{A}\otimes\overline{B} \\
\uparrow & & \uparrow \\
\mathsf{gr}(A\otimes B) & \simeq & \mathsf{gr}(A)\otimes\mathsf{gr}(B) 
\end{array} \
\mbox{and} \
\begin{array}{ccc}
\mathsf{gr}(A\otimes B) & \simeq & \mathsf{gr}(A)\otimes\mathsf{gr}(B) \\
\uparrow & & \uparrow \\
\overline{A\otimes B} & \simeq & \overline{A}\otimes\overline{B} 
\end{array},\]
where the vertical arrows are the corresponding projections and embeddings.
\end{proof}
\begin{cor}\label{i and q for group superschemes}
If $G$ is a group superscheme, then $i_G$ and $q_G$ are morphisms of group superschemes.
\end{cor}
\begin{proof}
For example, let us consider $i_G$. We have two commutative diagrams
\[\begin{array}{ccc}
(G\times G)_{ev} & \simeq & G_{ev}\times G_{ev} \\
\downarrow & & \downarrow \\
\mathsf{gr}(G\times G) & \simeq & \mathsf{gr}(G)\times \mathsf{gr}(G) 
\end{array} \ \mbox{and} \ 
\begin{array}{ccc}
(G\times G)_{ev} & \stackrel{m_{ev}}{\to} & G_{ev} \\
\downarrow & & \downarrow \\
\mathsf{gr}(G\times G) & \stackrel{\mathsf{gr}(m)}{\to }& \mathsf{gr}(G)
\end{array},\]
where $m : G\times G\to G$ is the multiplication morphism. Since the multiplication morphisms of $G_{ev}$ and $\mathsf{gr}(G)$ are factored through the isomorphisms (horizontal arrows) of the first diagram, our statement follows. 
The case of $q_G$ is similar. 
\end{proof}

\section{An interpretation of the functor of points and some useful consequences}

Due to the category equivalence $\mathcal{SVG}\simeq\mathcal{SFG}$, the results of the previous section can be reformulated as follows. There is an endofunctor of the category $\mathcal{SFG}$ that assigns to each
group superscheme $\mathbb{G}$ a group superscheme $\mathsf{gr}(\mathbb{G})$ such that there are natural morphisms (in $\mathcal{SFG}$)
\[{\bf i}_{\mathbb{G}} : \mathbb{G}_{ev}\to\mathsf{gr}(\mathbb{G}), \ {\bf q}_{\mathbb{G}} : \mathsf{gr}(\mathbb{G})\to\mathbb{G}_{ev},\]
where ${\bf i}_{\mathbb{G}}$ is an isomorphism onto $\mathsf{gr}(\mathbb{G})_{ev}$ and ${\bf q}_{\mathbb{G}}{\bf i}_{\mathbb{G}}=\mathrm{\bf id}_{\mathbb{G}_{ev}}$.

Let $\mathbb{G}_{odd}$ denote $\ker {\bf q}_{\mathbb{G}}$. Lemma \ref{one point subfunctor} implies that $\mathbb{G}_{odd}$ is a closed normal group super-subscheme of $\mathsf{gr}(\mathbb{G})$. Moreover, we have \[\mathsf{gr}(\mathbb{G})=\mathsf{gr}(\mathbb{G})_{ev}\ltimes\mathbb{G}_{odd}\simeq\mathbb{G}_{ev}\ltimes\mathbb{G}_{odd}.\]
If ${\bf f} : \mathbb{G}\to\mathbb{H}$ is a morphism of group superschemes, then Lemma \ref{embedding in and projection onto even super-subscheme} implies $\mathsf{gr}({\bf f})(\mathbb{G}_{odd})\leq\mathbb{H}_{odd}$.  Thus, 
$\mathbb{G}\to\mathbb{G}_{odd}$ is an endofunctor of the category $\mathcal{SFG}$. 

\begin{pr}\label{structure of odd complement}
Let $\mathbb{G}$ be a locally algebraic group superscheme. The group superscheme $\mathbb{G}_{odd}$ is affine, and represented by a local graded Hopf superalgebra
$\Lambda(\mathfrak{g}^*_1)$, where the elements of $\mathfrak{g}^*_1$ are assumed to be primitive.

Moreover, if ${\bf f} : \mathbb{G}\to\mathbb{H}$ is a morphism of locally algebraic group superschemes,
then  $\mathbb{G}_{odd}\to\mathbb{H}_{odd}$ is induced by the (dual) linear map $(\mathrm{d}_e{\bf f})^* : \mathfrak{h}_1^*\to\mathfrak{g}^*_1$.
\end{pr}
\begin{proof}
We have the exact sequence 
\[e\to G_{odd}\to\mathsf{gr}(G)\stackrel{q_G}{\to} G_{ev}\to e,\]
that is the geometric counterpart of the exact sequence
\[\mathbb{E}\to\mathbb{G}_{odd}\to\mathsf{gr}(\mathbb{G})\stackrel{{\bf q}_{\mathbb{G}}}{\to}\mathbb{G}_{ev}\to\mathbb{E}.\]
Since for every field extension $\Bbbk\subseteq F$ and every group superscheme $\mathbb{H}$ there is $\mathbb{H}(F)=\mathbb{H}_{ev}(F)$, the underlying topological space of $G_{odd}$ consists of the unit element $e$ only (see \cite[Lemma 5.5]{maszub1}). Therefore, $G_{odd}$ is affine, and represented by a local Hopf superalgebra $A$ with the nilpotent maximal superideal $A^+=\ker\epsilon_A$. 

Since ${\bf q}_{\mathbb{G}}{\bf i}_{\mathbb{G}}={\bf id}_{\mathbb{G}_{ev}}$, Lemma \ref{exact sequence of Lie superalgebras} implies that the Lie superalgebra of $\mathbb{G}_{odd}$ is purely-odd. Hence,
$A$ is generated by odd elements. In particular, $I_A=A^+$.

Furthermore, $\mathsf{gr}(G)$ is naturally isomorphic to $G_{ev}\times G_{odd}$ so that the projection onto $G_{ev}$ is identified with $q_{G}$. 
Thus 
\[\mathcal{O}_{\mathsf{gr}(G), e}\simeq \mathcal{O}_{G_{ev}, e}\otimes\mathcal{O}_{G_{odd}, e}=\mathcal{O}_{G_{ev}, e}\otimes A\]
and the maximal superideal $\mathfrak{n}$ of $\mathcal{O}_{\mathsf{gr}(G), e}$ is identified with 
$\overline{\mathfrak{m}_{e}}\otimes A +\mathcal{O}_{G_{ev}, e}\otimes A^+$, where $\overline{\mathfrak{m}_e}=\mathfrak{m}_e/I_{\mathcal{O}_e}$. Since the superideal $A^+$ is nilpotent, the $\mathfrak{n}$-adic topology on $\mathcal{O}_{\mathsf{gr}(G), e}$ coincides with its $\overline{\mathfrak{m}_e}$-adic topology. Besides, the grading of $\mathcal{O}_{\mathsf{gr}(G), e}$ is given by 
\[\mathcal{O}_{\mathsf{gr}(G), e}(k)\simeq \mathcal{O}_{G_{ev}, e}\otimes (A^+)^k/(A^+)^{k+1} \text{ for } k\geq 0.\]
In particular, 
there is an isomorphism of complete graded superalgebras $\widehat{\mathcal{O}_{\mathsf{gr}(G), e}}\simeq \widehat{\mathcal{O}_{G_{ev}, e}}\otimes A$, such that
\[\widehat{\mathcal{O}_{\mathsf{gr}(G), e}}(k)\simeq \widehat{\mathcal{O}_{G_{ev}, e}}\otimes (A^+)^k/(A^+)^{k+1} \text{ for } k\geq 0.\]
Moreover, the epimorphism of complete Hopf superalgebras $\widehat{\mathcal{O}_{\mathsf{gr}(G), e}}\to A$, induced by the closed immersion $G_{odd}\to\mathsf{gr}(G)$,
is identified with
\[\widehat{\mathcal{O}_{\mathsf{gr}(G), e}}\to \widehat{\mathcal{O}_{\mathsf{gr}(G), e}}/\widehat{\mathcal{O}_{\mathsf{gr}(G), e}}\widehat{\mathcal{O}_{G_{ev}, e}}^+.\]
Thus $A\simeq \mathsf{gr}(A)$ is a graded Hopf superalgebra, generated by 
component $A(1)=A^+/(A^+)^2$, whose elements are primitive. Let $J$ denote the kernel of the natural
epimorphism $\Lambda(A(1))\to A$ of graded Hopf superalgebras. Choose a basis $a_1, \ldots, a_t$ of $A(1)$. The induction on $t$ shows  that for every $1\leq i\leq t$, the induced epimorphism 
$\Lambda(A(1))/\Lambda(A(1))a_i\to A/Aa_i$
is an isomorphism. Hence $J\subseteq\cap_{1\leq i\leq t}\Lambda(A(1))a_i=\Bbbk a_1\ldots a_t$.
Since $\Bbbk a_1\ldots a_t$ is not Hopf superideal, $J=0$. 

Finally, $A(1)$ is the odd component of $\mathfrak{n}/\mathfrak{n}^2\simeq \mathfrak{m}_e/\mathfrak{m}_e^2$.  That is,  $A(1)\simeq\mathfrak{g}_1^*$, and Lemma \ref{another interpretation of differential} concludes the proof.
\end{proof}
\begin{lm}\label{grading of S_n}
The induced morphism $\mathsf{gr}(N_k)\to\mathsf{gr}(G)$ coincides with the $k$th neighborhood of the closed embedding $e\to\mathsf{gr}(G)$. 
\end{lm}
\begin{proof}
It is enough to show that the kernel of the induced epimorphism 
\[\mathsf{gr}(\mathcal{O}_e)\to\mathsf{gr}(\mathcal{O}_e/\mathfrak{m}_e^{k+1})\]
coincides with $\mathfrak{n}^{k+1}$, where $\mathfrak{n}$ is the maximal superideal of $\mathsf{gr}(\mathcal{O}_e)\simeq \mathcal{O}_{\mathsf{gr}(G), e}$. 

Let $V$ denote the superspace $I_{\mathcal{O}_e}/I_{\mathcal{O}_e}^2$. As it was already shown in the previous lemma, the graded superalgebra $\mathsf{gr}(\mathcal{O}_e)$ is naturally isomorphic to $\overline{\mathcal{O}_e}\otimes\Lambda(V)$. Besides, $\mathfrak{n}$ is identified with the graded superideal
$\overline{\mathfrak{m}_e}\otimes 1 + \overline{\mathcal{O}_e}\otimes\Lambda(V)^+$, where $\overline{\mathfrak{m}_e}=\mathfrak{m}_e/I_{\mathcal{O}_e}$. For every $k\geq 0$, there is
\[\mathfrak{n}^{k+1}=\sum_{k+1-t\leq l\leq k+1}\overline{\mathfrak{m}_e}^l\otimes (\oplus_{k+1-l\leq s\leq t}\Lambda^s(V))=\oplus_{0\leq i\leq t}\overline{\mathfrak{m}_e}^{k+1-i}\otimes\Lambda^i(V),\]
where $t=\dim V$. In particular, the graded superalgebra $\mathsf{gr}(\mathcal{O}_e)/\mathfrak{n}^{k+1}$ is isomorphic to
$\oplus_{0\leq i\leq t}\overline{\mathcal{O}_e}/\overline{\mathfrak{m}_e}^{k+1-i}\otimes\Lambda^i(V).$

On the other hand, the isomorphism $\mathsf{gr}(\mathcal{O}_{e})\simeq \overline{\mathcal{O}_e}\otimes\Lambda(V)$ implies that each element $f\in \mathcal{O}_e$ has the "canonical" form
\[f=\sum_{0\leq s\leq t, 1\leq i_1<\ldots < i_s\leq t} f_{i_1, \ldots , i_s}v_{i_1}\ldots v_{i_s}, \]
where the elements $v_1, \ldots v_t\in (\mathcal{O}_e)_1$ form a basis of $V$, and the coefficients 
$f_{i_1, \ldots , i_s}\in (\mathcal{O}_e)_0$ are uniquely defined modulo $I_{\mathcal{O}_e}$, or equivalently, modulo $(\mathcal{O}_e)_1^2$. 

In particular, $f$ belongs to $\mathfrak{m}_e$ if and only if its "free" coefficient (corresponding to $s=0$) does. Moreover, 
$f$ belongs to $\mathfrak{m}_e^{k+1}$ if and only if each $f_{i_1, \ldots , i_s}$ belongs to $\mathfrak{m}_e^{k+1-s}$ modulo $I_{\mathcal{O}_e}$. Therefore, every element $\overline{f}$ of superalgebra
$\mathcal{O}_e/\mathfrak{m}_e^{k+1}$ has a form
\[\sum_{0\leq s\leq t, 1\leq i_1<\ldots < i_s\leq t} \overline{f}_{i_1, \ldots , i_s}\overline{v_{i_1}}\ldots \overline{v_{i_s}},  \text{ where } \overline{f}_{i_1, \ldots , i_s}\in (\mathcal{O}_e/\mathfrak{m}_e^{k+1})_0,\]
and  each $f_{i_1, \ldots , i_s}$ is uniquely defined modulo $\mathfrak{m}_e^{k+1-s}+I_{\mathcal{O}_e}$. The lemma  follows.\end{proof}
\begin{cor}\label{formal neighborhood of identity in gr(G)}
The formal neighborhood of the identity in $\mathsf{gr}(\mathbb{G})$ coincides with
$\cup_{k\geq 0}\mathsf{gr}(\mathbb{N}_k)$. 
\end{cor}
If we denote the formal neighborhood of the identity in $\mathsf{gr}(\mathbb{G})$ by $\mathsf{gr}(\mathbb{N})$, then the above corollary states that $\mathsf{gr}(\mathbb{N})$ is "represented" by complete
Hopf superalgebra $\mathsf{gr}(\widehat{\mathcal{O}_e})\simeq\widehat{\mathsf{gr}(\mathcal{O}_e)}$, whose comultiplication and antipode are
$\mathsf{gr}(\Delta)$ and $\mathsf{gr}(S)$, respectively. Moreover, arguing as in Lemma \ref{grading of S_n}, one sees that each element $f\in\widehat{\mathcal{O}_e}$ has a "canonical" form
\[\sum_{0\leq s\leq t, 1\leq i_1<\ldots < i_s\leq t} f_{i_1, \ldots , i_s}v_{i_1}\ldots v_{i_s}, \text{ where } f_{i_1, \ldots , i_s}\in (\widehat{\mathcal{O}_e})_0,\]
and the coefficients $f_{i_1, \ldots , i_s}$ are uniquely defined modulo $(\widehat{\mathcal{O}_e})_1^2$.

It is clear that $\mathbb{N}_{ev}=\mathbb{N}\cap\mathbb{G}_{ev}$ is the formal neighborhood of the identity in $\mathbb{G}_{ev}$. Moreover, $\mathrm{hyp}(\mathbb{G}_{ev})$ can be naturally identified with 
the purely-even subalgebra of $\mathrm{hyp}(\mathbb{G})$, consisting of all $\phi\in \mathrm{hyp}(\mathbb{G})$ such that $\phi(I_{\widehat{\mathcal{O}_e}})=0$. More generally, let $\mathrm{hyp}^{(k)}(\mathbb{G})$ denote the super-subspace $\{\phi\in\mathrm{hyp}(\mathbb{G})\mid \phi(I^{k+1}_{\widehat{\mathcal{O}_e}})=0\}$. In particular, $\mathrm{hyp}(\mathbb{G}_{ev})=\mathrm{hyp}^{(0)}(\mathbb{G})$ and $\mathrm{hyp}(\mathbb{G})=\mathrm{hyp}^{(t)}(\mathbb{G})$.  The following lemma is similar to Lemma \ref{standard properties of hyp}.
\begin{lm}\label{standard properties of another filtration of hyp}
For every nonnegative integers $0\leq k, l\leq t$ there is:
\begin{enumerate}
\item $\mathrm{hyp}^{(k)}(\mathbb{G})\mathrm{hyp}^{(l)}(\mathbb{G})\subseteq \mathrm{hyp}^{(k+l)}(\mathbb{G})$,
\item $\Delta^*(\mathrm{hyp}^{(k)}(\mathbb{G}))\subseteq\sum_{0\leq s\leq k} \mathrm{hyp}^{(s)}(\mathbb{G})\otimes \mathrm{hyp}^{(k-s)}(\mathbb{G})$,
\item $S^*(\mathrm{hyp}^{(k)}(\mathbb{G}))\subseteq \mathrm{hyp}^{(k)}(\mathbb{G})$.
\end{enumerate}
\end{lm}
Lemma \ref{standard properties of another filtration of hyp} immediately implies that \[\mathsf{gr}(\mathrm{hyp}(\mathbb{G}))=\oplus_{0\leq k\leq t}\mathrm{hyp}^{(k)}(\mathbb{G})/\mathrm{hyp}^{(k-1)}(\mathbb{G})\] has a natural structure of (finitely) graded Hopf superalgebra. 

\begin{pr}\label{isomorphism of graded hyp and hyp of graded}
There is a canonical isomorphism $\mathsf{gr}(\mathrm{hyp}(\mathbb{G}))\simeq\mathrm{hyp}(\mathsf{gr}(\mathbb{G}))$ of graded Hopf superalgebras. 
\end{pr}
\begin{proof}
First of all, the superspace $\mathrm{hyp}(\mathsf{gr}(\mathbb{G}))$ is isomorphic to $\oplus_{0\leq k\leq t}(I_{\widehat{\mathcal{O}_e}}^k/I_{\widehat{\mathcal{O}_e}}^{k+1})^{\underline{*}}$,
where $(I_{\widehat{\mathcal{O}_e}}^k/I_{\widehat{\mathcal{O}_e}}^{k+1})^{\underline{*}}$ consists of all continuous linear maps $I_{\widehat{\mathcal{O}_e}}^k/I_{\widehat{\mathcal{O}_e}}^{k+1}\to \Bbbk$, and $\Bbbk$ is regarded as a discrete vector space. In fact, arguing as in  Lemma \ref{grading of S_n}, one sees that the $\widehat{\mathfrak{n}}$-adic topology of $\mathsf{gr}(\widehat{\mathcal{O}_e})$ (respectively, the factor-topology of $I_{\widehat{\mathcal{O}_e}}^k/I_{\widehat{\mathcal{O}_e}}^{k+1}$)  is equivalent to its $\widehat{\overline{\mathfrak{m}_e}}$-adic topology.
Thus, $\mathsf{gr}(\widehat{\mathcal{O}_e})$ is isomorphic to
$\oplus_{0\leq k\leq t}I_{\widehat{\mathcal{O}_e}}^k/I_{\widehat{\mathcal{O}_e}}^{k+1}$ as a (linear) topological superspace.

Next, we have a natural superspace morphism 
\[\mathrm{hyp}^{(k)}(\mathbb{G})/\mathrm{hyp}^{(k-1)}(\mathbb{G})\to (I_{\widehat{\mathcal{O}_e}}^{k-1}/I_{\widehat{\mathcal{O}_e}}^k)^{\underline{*}},\] induced by the Hopf pairing $< \ , \ >$, which is obviously injective.  On the other hand, every $\phi\in (I_{\widehat{\mathcal{O}_e}}^{k-1}/I_{\widehat{\mathcal{O}_e}}^k)^{\underline{*}}$ is some $\psi\in \mathsf{gr}(\widehat{\mathcal{O}_e})^{\underline{*}}$ restricted on  $I_{\widehat{\mathcal{O}_e}}^{k-1}/I_{\widehat{\mathcal{O}_e}}^k$. Since $\psi$ is continuous, there is $l\geq 0$ such that $\psi(\widehat{\mathfrak{n}}^{l+1})=0$.

Choose a basis of 
$\mathsf{gr}(\widehat{\mathcal{O}_e})$ consisting of elements
$\overline{f} v_{i_1}\ldots v_{i_s}$, where $1\leq i_1<\ldots < i_s\leq t, 0\leq s\leq t$, and $f\in(\widehat{\mathcal{O}_e})_0$. We also assume that for every $k\geq 0$, the elements $\overline{f}$, that belong to $\widehat{\overline{\mathfrak{m}_e}}^k\setminus \widehat{\overline{\mathfrak{m}_e}}^{k+1}$, form a basis
of $\widehat{\overline{\mathfrak{m}_e}}^k /\widehat{\overline{\mathfrak{m}_e}}^{k+1}$. Then the elements
$f v_{i_1}\ldots v_{i_s}$ form a basis of $\widehat{\mathcal{O}_e}$. Moreover, the superideal $\widehat{\mathfrak{m}_e}^{k+1}$ has a basis consisting of all $f v_{i_1}\ldots v_{i_s}$ such that 
$f\in \widehat{\mathfrak{m}_e}^{k+1-s}+ I_{\widehat{\mathcal{O}_e}}$ (or equivalently, $\overline{f}\in \widehat{\overline{\mathfrak{m}_e}}^{k+1-s}$). Similarly, the superideal $I_{\widehat{\mathcal{O}_e}}^k$ has a basis consisting of all $f v_{i_1}\ldots v_{i_s}$ with $s\geq k$.

Set $\psi'(f v_{i_1}\ldots v_{i_s})=\psi(\overline{f} v_{i_1}\ldots v_{i_s})$ for each basic element
$f v_{i_1}\ldots v_{i_s}$. Then $\psi'$ belongs to $\mathrm{hyp}^{(k)}(\mathbb{G})$ and its image in
$(I_{\widehat{\mathcal{O}_e}}^{k-1}/I_{\widehat{\mathcal{O}_e}}^k)^{\underline{*}}$ coincides with $\phi$. In other words, we have a natural graded pairing
\[\mathsf{gr}(\mathrm{hyp}(\mathbb{G}))\times \mathsf{gr}_{I_{\widehat{\mathcal{O}_e}}}(\widehat{\mathcal{O}_e})\to \Bbbk\]
that induces an isomorphism $\mathsf{gr}(\mathrm{hyp}(\mathbb{G}))\simeq\mathrm{hyp}(\mathsf{gr}(\mathbb{G}))$ of superspaces. It is a Hopf pairing by the definition of Hopf superalgebra structure on
$\mathsf{gr}(\widehat{\mathcal{O}_e})$. The proposition is proven.
\end{proof}
Let $\gamma_1, \ldots, \gamma_t$ be a basis of $\mathfrak{g}_1\simeq V^*$, dual to the basis $v_1, \ldots, v_t$ of $V$. 
\begin{pr}\label{a canonical form in hyp(G)}
Every element $\phi\in\mathrm{hyp}_k(\mathbb{G})$ can be uniquely expressed as
\[\phi=\sum_{0\leq s\leq t, 1\leq i_1<\ldots < i_s\leq t}\phi_{i_1, \ldots, i_s}\gamma_{i_1}\ldots\gamma_{i_s}, \text{ where }\phi_{i_1, \ldots, i_s}\in \mathrm{hyp}_{k-s}(\mathbb{G}_{ev}).\]
\end{pr}
\begin{proof}
It is sufficient to prove this statement for a (locally algebraic) group superscheme $\mathbb{G}\simeq\mathsf{gr}(\mathbb{G})$. In this case, we have $\widehat{\mathcal{O}_e}\simeq A\otimes\Lambda(V)$, where
$A\simeq\overline{\widehat{\mathcal{O}_e}}$ is a closed Hopf subalgebra of $\widehat{\mathcal{O}_e}$. Since $\widehat{\mathcal{O}_e}$ is also a graded Hopf superalgebra with $\widehat{\mathcal{O}_e}(1)=A\otimes V$, then each $v\in V$ is skew
primitive, i.e. $\Delta'(v)=v\otimes a+ b\otimes v$, where $a$ and $b$ are group-like elements from $A$, and $\Delta'=\mathsf{gr}(\Delta)$.

Further, the Hopf pairing $\mathrm{hyp}(\mathbb{G})\times\widehat{\mathcal{O}_e}\to \Bbbk$ is graded. Therefore, it is enough to show that 
$<\phi\gamma_{i_1}\ldots\gamma_{i_s}, f v_{j_1}\ldots v_{j_l}>\neq 0$ implies $s=l, i_1=j_1, \ldots, i_s=j_s$. Moreover, in this case
\[<\phi\gamma_{i_1}\ldots\gamma_{i_s}, f v_{i_1}\ldots v_{i_s}>=<\phi, fb_{i_1}\ldots b_{i_s}>,\]
where $\Delta'(v_i)=a_i\otimes v_i+v_i\otimes b_i, 1\leq i\leq t$.
We use an induction on $s$. 

If $s=0$, then the statement is obvious. Let $s\geq 1$. We have
\[\Delta'(f v_{j_1}\ldots v_{j_l})=( f_{(0)}\otimes f_{(1)})\prod_{1\leq p\leq l}(v_{j_p}\otimes a_{j_p}+ b_{j_p}\otimes v_{j_p})=\]\[\sum_{K\subseteq\{j_1, \ldots, j_l\}} \pm f_{(0)}b_{\overline{K}}v_K\otimes f_{(1)} a_{K}v_{\overline{K}},\]
where $\Delta'(f)=f_{(0)}\otimes f_{(1)}$ and $v_K$ (respectively, $a_K$ or $b_K$) denotes a product of all $v_{j_p}$ (respectively, $a_{j_p}$ or $b_{j_p}$) for $j_p\in K$ (in an arbitrary order), and $\overline{K}=\{j_1, \ldots, j_l\}\setminus K$ is the complement of $K$. Then
\[<(\phi\gamma_{i_1}\ldots\gamma_{i_{s-1}})\gamma_{i_s}, f v_{j_1}\ldots v_{j_l}>=\]\[
\sum_{K\subseteq\{j_1, \ldots, j_l\}}\pm <\phi\gamma_{i_1}\ldots\gamma_{i_{s-1}}, f_{(0)}b_{\overline{K}}v_K><\gamma_{i_s}, f_{(1)} a_{K}v_{\overline{K}}>\]
contains nonzero summands for $K=\{i_1, \ldots, i_{s-1}\}$ only. Moreover, $\overline{K}$ should be a singleton, say $\overline{K}=\{j_p\}$. Thus $s=l$ and $<\gamma_{i_s}, f_{(1)} a_{K}v_{j_p}>=\epsilon_A(f_{(1)})<\gamma_{i_s}, v_{j_p}>\neq 0$ implies $i_s=j_p$. Besides, $<(\phi\gamma_{i_1}\ldots\gamma_{i_{s-1}})\gamma_{i_s}, f v_{j_1}\ldots v_{j_l}>$ is equal to
\[<\phi\gamma_{i_1}\ldots\gamma_{i_{s-1}}, f_{(0)} b_{i_s} v_{i_1}\ldots v_{i_{s-1}}>\epsilon_A(f_{(1)})=<\phi, fb_{i_1}\ldots b_{i_s}>.\]
The proposition is proven.
\end{proof}

\section{The fundamental equivalence}

\subsection{The category of Harish-Chandra pairs}

In this subsection, we follow \cite{masshib}. 

A pair $(\mathsf{G}, \mathsf{V})$, where $\mathsf{G}$ is a locally algebraic group scheme, and $\mathsf{V}$ is a  $\mathsf{G}$-module, is a \emph{Harish-Chandra pair} if the following conditions hold:
\begin{enumerate}
\item[(a)] There is a functor $\mathsf{V}_a\times\mathsf{V}_a\to\mathsf{g}_a$, denoted by $[ \ , \ ]$, where $\mathsf{g}$ is the Lie algebra of $\mathsf{G}$, such that for every $R\in\mathsf{Alg}_{\Bbbk}$, the map
$\mathsf{V}_a(R)\times\mathsf{V}_a(R)\to\mathsf{g}_a(R)$ is $R$-bilinear and symmetric.
\item[(b)] The functor $\mathsf{V}_a\times\mathsf{V}_a\to\mathsf{g}_a$ is $\mathsf{G}$-equivariant for the diagonal action of $\mathsf{G}$ on $\mathsf{V}_a\times\mathsf{V}_a$ and the adjoint action of $\mathsf{G}$ on $\mathsf{g}_a$.
\item[(c)] The induced action of $\mathsf{g}$ on $\mathsf{V}$ satisfies $[v, v]\cdot v=0$ for $v\in\mathsf{V}$.
\end{enumerate}  

The morphism of Harish-Chandra pairs $(\mathsf{G}, \mathsf{V})\to (\mathsf{H}, \mathsf{W})$ is a morphism of pairs 
$(\mathsf{f}, u)$ as in Section 7, such that the diagram
\[\begin{array}{ccc}
\mathsf{V}_a\times \mathsf{V}_a & \to & \mathsf{g}_a \\
\downarrow & & \downarrow \\
\mathsf{W}_a\times \mathsf{W}_a & \to & \mathsf{h}_a 
\end{array}\]
is commutative, where the first vertical map is $\mathsf{u}_a\times\mathsf{u}_a$ and the second vertical map is $\mathrm{d}_e\mathsf{f}$. The category of Harish-Chandra pairs is denoted by $\mathsf{HCP}$.

To every algebraic group superscheme $\mathbb{G}$, we associate a Harish-Chandra pair $(\mathsf{G}, \mathsf{V})$ with $\mathsf{G}=\mathbb{G}_{ev}$ (regarded as a group scheme), and $\mathsf{V}=\mathfrak{g}_1$.  Besides, the action of $\mathsf{G}$ on $\mathsf{V}$ is induced by the adjoint action of $\mathbb{G}$, and $[ \ , \ ]$ is the restriction of Lie super-bracket of $\mathfrak{g}$ on its odd component.
\begin{lm}\label{first functor}
The correspondence $\Phi : \mathbb{G}\mapsto (\mathsf{G}, \mathsf{V})$ is a functor from $\mathcal{SFG}_{la}$ to $\mathsf{HCP}$.
\end{lm}
We call $\Phi$ the \emph{Harish-Chandra functor}.

Recall that each group $\mathbb{N}(R)$ is identified with $\mathsf{Gpl}(\mathrm{hyp}(\mathbb{G})\otimes R)$.
For every $v\in\mathsf{V}=\mathfrak{g}_1, x\in\mathsf{g}=\mathfrak{g}_0$ and $a\in R_1, b\in R_0$ such that $b^2=0$,
define the group-like elements 
\[f(b , x)=\epsilon_e^*\otimes 1+x\otimes b, e(a, v)=\epsilon_e^* \otimes 1+v\otimes a .\]
These elements are images of the group-like elements $e^{\epsilon_0(x\otimes 1)}$ and $e^{-\epsilon_1(v\otimes 1)}$ from $\mathbb{N}(R[\epsilon_0, \epsilon_1])$ under a homomorphism of groups, induced by $\epsilon_0\mapsto b, \epsilon_1\mapsto a$. Using the identity involving group commutators 
from Lemma \ref{differential is a Lie superalgebra morphism}, we obtain the first three relations in the list:
\begin{enumerate}
\item $[e(a, v), e(a', v')]=f(-aa', [v, v'])$,
\item $[f(b, x), e(a, v)]=e(ba, [x, v])$,
\item $[f(b, x), f(b', x')]=f(bb', [x, x'])$,
\item $e(a, v)e(a', v)=f(-aa', \frac{1}{2}[v, v])e(a+a', v)$.
\end{enumerate}
A direct computation can derive the fourth relation.
The relations $(1), (2), (3)$ are the relations $(\mathrm{i}), (\mathrm{iii}), (\mathrm{iv})$ from \cite[Lemma 4.2]{masshib}, and the relation $(4)$ is $(\mathrm{ii})$ therein.

Also, for every $g\in\mathsf{G}(R)$ we have 
\begin{enumerate}
\item[(5)] $e(a, v)^g=e(a, \mathrm{Ad}(g)v)) \text{ and } f(b, x)^g=f(b, \mathrm{Ad}(g)x))$.
\end{enumerate}
Define the group subfunctors $\bf\Sigma$ and $\mathsf{F}$ of $\mathbb{N}$ and $\mathsf{N}=\mathbb{N}_{ev}$ respectively, such that ${\bf\Sigma}(R)$ is generated by all the elements $f(b, x)$ and $e(a, v)$, but $\mathsf{F}(R)$ is generated only by $f(b, x)$, where 
$a, b\in R$ and $R\in\mathsf{SAlg}_{\Bbbk}$.

\begin{lm}\label{decomposition of Sigma}
The following statements hold:
\begin{enumerate}
\item $\mathsf{F}(R)={\bf \Sigma}(R)\cap\mathsf{N}(R)$.
\item If $v_1, \ldots , v_t$ is a basis of $\mathsf{V}$, then each element of ${\bf \Sigma}(R)$ is uniquely expressed in the form
\[f e(a_1, v_1)\ldots e(a_t, v_t),\]
where $f\in\mathsf{F}(R)$ and $a_i\in R_1$ for $1\leq i\leq t$.
\item Both group subfunctors $\mathsf{F}$ and $\bf\Sigma$ are invariant for the conjugation action of $\mathsf{G}$.
\end{enumerate}
\end{lm}
\begin{proof}
To prove the first two statements, use Proposition \ref{a canonical form in hyp(G)}, and copy the proof of \cite[Proposition 4.3]{masshib}. The third statement follows from the relation $(5)$ above.
\end{proof}
\begin{rem}
The symmetric version of Lemma \ref{decomposition of Sigma} (2), where $f$ appears on the right-hand side, is also valid.  
\end{rem}
Now we can define a group subfunctor $\mathbb{G}'$ of $\mathbb{G}$, such that
$\mathbb{G}'(R)=\mathsf{G}(R){\bf \Sigma}(R)$ for $R\in\mathsf{SAlg}_{\Bbbk}$. 
\begin{cor}\label{decomposition of G'}
Each element of $\mathbb{G}'(R)$ is uniquely expressed in the form
\[ge(a_1, v_1)\ldots e(a_t, v_t),\]
where $g\in\mathsf{G}(R)$, and $a_i\in R_1$ for $1\leq i\leq t$. 
\end{cor}
Let $\bf E$ denote a subfunctor
of $\bf\Sigma$ such that ${\bf E}(R)$ consists of all products
\[e(a_1, v_1)\ldots e(a_t, v_t), \text{ where } a_i\in R_1 \text{ for } 1\leq i\leq t.\]
It is clear that ${\bf E}\simeq\mathrm{SSp}(\Lambda(\mathsf{V^*}))$ is a purely-odd affine superscheme.
Then $\mathbb{G}'$ is isomorphic to $\mathsf{G}\times {\bf E}$ as a superscheme. In particular, $\mathbb{G}'$ is a (locally algebraic) group superscheme.
\begin{theorem}\label{coincidence of G and G'}
We have $\mathbb{G}=\mathbb{G}'$.
\end{theorem}
\begin{proof}
By Proposition \ref{graded superscheme}, one has to prove that the induced morphism
$\mathsf{gr}(\mathbb{G}')\to \mathsf{gr}(\mathbb{G})$ of group superschemes is an isomorphism. 
  
Since $\mathbb{G}_{ev}=\mathbb{G}'_{ev}$, Lemma \ref{embedding in and projection onto even super-subscheme} implies that $\mathsf{gr}(\mathbb{G}')_{ev}$ is mapped isomorphically onto $\mathsf{gr}(\mathbb{G}')_{ev}$. By Proposition \ref{structure of odd complement}, it is sufficient to prove that the embedding
$\mathbb{G}'\to\mathbb{G}$ induces an isomorphism of odd components of their Lie superalgebras. But as it has been already observed, in both $\mathbb{G}$ and $\mathbb{G}'$, the elements
$e^{\epsilon_1 (v_i\otimes 1)}$ coincide with $e(-\epsilon_1, v_i)$. The theorem is proven.  
\end{proof}
\begin{cor}\label{decomposition of S}
Let $\mathsf{N}$ denote $\mathbb{N}_{ev}$. Then $\mathbb{N}=\mathsf{N}{\bf E}={\bf E}\mathsf{N}$.
\end{cor}
\begin{proof}
If $s\in\mathbb{N}(R)$ is expressed as $gx$, where $g\in\mathsf{G}(R)$ and $x\in{\bf E}(R)$, then
$g=sx^{-1}\in\mathbb{G}_{ev}(R)\cap\mathbb{N}=\mathbb{N}_{ev}(R)$, since ${\bf E}\subseteq\mathbb{N}$.
It is clear that ${\bf E}\mathsf{N}\subseteq\mathbb{N}$. Moreover, for any $s\in\mathsf{N}(R)$ and $x\in {\bf E}(R),$ we have $x^{s}=sxs^{-1}=x' f$ for $f\in\mathsf{F}(R)$ and $x'\in {\bf E}(R)$, by the symmetric version of Lemma \ref{decomposition of Sigma} (2). Thus,  $sx=x^{s}s=x' (f s)$.
\end{proof}

\subsection{A quasi-inverse of the Harish-Chandra functor}

Let $(\mathsf{G}, \mathsf{V})$ be a Harish-Chandra pair. Note that $\mathsf{g}\oplus\mathsf{V}$ has a natural Lie superalgebra structure, induced by the bilinear map $\mathsf{V}\times\mathsf{V}\to\mathsf{g}$, and by the differential of the action of $\mathsf{G}$ on $\mathsf{V}$. Let $\mathfrak{g}$ denote this Lie superalgebra. Fix  a basis $x_1, \ldots, x_l$ of $\mathfrak{g}_0=\mathsf{g}$. 

Following \cite{masshib}, one can define a group functor $\bf\Sigma'$ that associates with each superalgebra $R$ a group generated by symbols
$e'(a, v), f'(b, x)$ for $a\in R_1, b\in R_0$ such that $b^2=0, v\in\mathsf{V}$ and $x\in\mathsf{g}$, subject to the relations
$(1)-(4)$. Note that if either one of the parameters $a, b, v, x$ is zero, then the corresponding generator is the unit. Let $\mathsf{F}'$ be a group subfunctor of $\bf\Sigma'$, such that $\mathsf{F}'(R)$ is generated only by the symbols $f'(b, x)$. 

Since the proof of \cite[Lemma 4.4]{masshib} uses only the above relations, one can conclude again that any element of ${\bf\Sigma}'(R)$ has a form $f'e(a_1, v_1)\ldots e'(a_t, v_t)$, where $f'\in\mathsf{F}'(R)$. On the other hand, there is a natural morphism of the group functor ${\bf\Sigma}'$ to the group functor
$R\mapsto\mathsf{Gpl}(\mathsf{U}(\mathfrak{g})\otimes R)$ that sends $e'(a, v)$ to $1\otimes 1+v\otimes a$ and $f'(b, x)$ to $1\otimes 1+x\otimes b$, respectively. Then \cite[Proposition 4.3]{masshib} implies that the above form is unique.

Further, the action of $\mathsf{G}$ on $\mathsf{V}$ induces the action on the generators by
\[g\cdot e'(a, v)=e'(r_1a, v_1)\ldots e'(r_t a, v_t) \text{ and } g\cdot f'(b, x)=f'(s_1b, x_1)\ldots f'(s_l b, x_l),\]
where $g\cdot v=\sum_{1\leq i\leq t} v_i\otimes r_i$ and $\mathrm{Ad}(g)(x)=\sum_{1\leq j\leq l} x_j\otimes s_j$. 
\begin{lm}\label{morphism alpha}
The above action preserves the relations $(1)-(4)$. Hence, it defines a homomorphism
$\alpha(R) : \mathsf{G}(R)\to\mathrm{Aut}({\bf\Sigma}'(R))$. 
\end{lm}
\begin{proof}
Consider the relation $(1)$. Suppose that $g\cdot v'=\sum_{1\leq i\leq t} v_i\otimes r'_i$. Then for every
$x, y, z\in\{e'(r_i a, v_i), e'(r'_k a', v_k)\mid 1\leq i, k\leq t\}$ we have $[[x, y], z]=1$. Thus, the well-known commutator identity $[xy, z]=[x, z]^y [y, z]$ implies $[xy, z]=[x, z][y, z]$. Therefore, 
\[g\cdot [e'(a, v), e'(a', v')]=[g\cdot e'(a, v), g\cdot e'(a', v')]=\prod_{1\le i, k\leq t}[e'(r_i a, v_i), e'(r'_k a', v_k)]=\]
\[\prod_{1\leq i, k\leq t}f'(-r_iar'_k a', [v_i, v_k])=f'(-aa', g\cdot[v, v'])=g\cdot f'(-aa', [v, v']).\] 
The identities (3) and (4) can be derived similarly. The identity $(2)$ follows from $g\cdot [x, v]=[\mathrm{Ad}(g)(x), g\cdot v]$, which is obtained by using the identity at the beginning of the proof of Lemma \ref{differential is a Lie superalgebra morphism}. The lemma is proven. 
\end{proof}
The map $f'(b, x)\mapsto f(b, x)$ induces a morphism of group functors $i : \mathsf{F}'\to\mathsf{F}$. In fact, it is the composition of the above morphism from $\mathsf{F}'$ to the group functor $R\mapsto\mathsf{Gpl}(\mathsf{U}(\mathsf{g})\otimes R_0)$ and the natural morphism from the latter to $\mathsf{F}$
, induced by the Hopf algebra morphism $\mathsf{U}(\mathsf{g})\to \mathrm{hyp}(\mathbb{G}_{ev})$. 

It is clear that $({\bf\Sigma}'(R), \mathsf{F}'(R), \mathsf{G}(R), i(R), \alpha(R))$ is a quintuple in the sense of
\cite{masshib}, with the only difference that $\alpha(R)$ is a group homomorphism. We define a group
${\bf\Gamma}(R)$ as a factor-group of $\mathsf{G}(R)\ltimes {\bf\Sigma}'(R)$ by a normal subgroup
${\bf\Xi}(R)=\{(i(R)(f), f^{-1})\mid f\in\mathsf{F}'(R)\}$. 
\begin{rem}
${\bf\Gamma}(R)$ is isomorphic to the factor-group of the amalgamated free product $\mathsf{G}(R)*_{\mathsf{F}'(R)} {\bf\Sigma}'(R)$ modulo the relations $s^g=(\alpha(R)(g))(s)$ for $g\in\mathsf{G}(R)$ and  $s\in {\bf\Sigma}'(R)$.
\end{rem}
By \cite[Lemma 2.1]{masshib} every element of ${\bf\Gamma}(R)$ can be uniquely expressed as 
\[g e'(a_1, v_1)\ldots e'(a_t, v_t),\] where $g\in\mathsf{G}(R)$ and $a_i\in R_1$ for $1\leq i\leq t$. In other words, the superscheme $\bf\Gamma$ is isomorphic to $\mathsf{G}\times {\bf E}$, hence locally algebraic.

The following lemma is now apparent.
\begin{lm}\label{Gamma}
$R\mapsto {\bf\Gamma}(R)$ is a group functor, hence a locally algebraic group superscheme. Moreover, $\Psi : (\mathsf{G}, \mathsf{V})\to {\bf\Gamma}$ is a functor from $\mathsf{HCP}$ to
$\mathcal{SFG}_{la}$. 
\end{lm}
\begin{theorem}\label{fundamental equivalence}
The functors $\Phi$ and $\Psi$ are quasi-inverse to each other.
\end{theorem}
\begin{proof}
Since the group structure of ${\bf\Gamma}(R)$ is defined by the relations similar to $(1)-(5)$, Theorem \ref{coincidence of G and G'} implies 
$\Psi\circ\Phi\simeq\mathrm{id}_{\mathcal{SFG}_{la}}$. Further, the odd component of the Lie superalgebra of $\bf\Gamma$ can be identified with the (commutative) subgroup of ${\bf\Gamma}(K[\epsilon_0, \epsilon_1])$ consisting of all elements $e'(c_1\epsilon_1, v_1)\ldots e'(c_t\epsilon_1, v_t)$, where $c_i \in \Bbbk$ for $1\leq i\leq t$, which can be identified with $\mathsf{V}$. Moreover, the adjoint action of $\mathsf{G}$ is  identified with the original action of $\mathsf{G}$. Thus $\Phi\circ\Psi\simeq\mathrm{id}_{\mathsf{HCP}}$. 
The theorem is proven.
\end{proof}
In what follows, let $\mathsf{E}$ denote $\mathbb{E}_{ev}$. 

Let $\mathbb{G}$ be a locally algebraic group superscheme. If $\Phi(\mathbb{G})\simeq (\mathsf{G}, \mathsf{V})$ or $\Psi((\mathsf{G}, \mathsf{V}))\simeq\mathbb{G}$, then we say that $\mathbb{G}$ is represented by $(\mathsf{G}, \mathsf{V})$.

A sequence $(\mathsf{R}, \mathsf{W})\to (\mathsf{G}, \mathsf{V})\to (\mathsf{H}, \mathsf{U})$ in $\mathsf{HCP}$ is called \emph{exact} whenever the following conditions hold:
\begin{enumerate}
\item The sequences 
$0\to \mathsf{W}\to \mathsf{V}\to \mathsf{U}\to 0$ and $\mathsf{E}\to\mathsf{R}\to \mathsf{G}\to \mathsf{H}\to\mathsf{E}$ are exact in the categories of superspaces and group schemes, respectively.
\item[(2a)] $\mathsf{W}$ is a $\mathsf{G}$-submodule of $\mathsf{V}$.
\item[(2b)] $\mathsf{R}$ acts trivially on $\mathsf{V}/\mathsf{W}$.
\item[(2c)] $[\mathsf{V}, \mathsf{W}]\subseteq\mathrm{Lie}(\mathsf{R})$. 
\end{enumerate}
\begin{theorem}\label{exact sequences of supergroups}
A sequence of group superschemes $\mathbb{R}\to\mathbb{G}\to\mathbb{H}$ is exact if and only if
the sequence $\Phi(\mathbb{R})\to\Phi(\mathbb{G})\to\Phi(\mathbb{H})$ is exact.
\end{theorem}
\begin{proof}
Suppose that $\mathbb{R}, \mathbb{G}$ and $\mathbb{H}$ are represented by the pairs
$(\mathsf{R}, \mathsf{W}), (\mathsf{G}, \mathsf{V})$, and $(\mathsf{H}, \mathsf{U})$, respectively.
Without loss of a generality, one can replace $\mathbb{R}\to\mathbb{G}\to\mathbb{H}$ by
the sequence $\Psi((\mathsf{R}, \mathsf{W}))\to \Psi((\mathsf{G}, \mathsf{V}))\to \Psi((\mathsf{H}, \mathsf{U}))$. Then $\mathbb{R}$ coincides with $\ker(\mathbb{G}\to\mathbb{H})$ if and only if
$\mathsf{R}=\ker(\mathsf{G}\to\mathsf{H}), \mathsf{W}=\ker(\mathsf{V}\to\mathsf{U})$ and $\mathsf{W}$ satisfies the conditions $(2a)-(2c)$. 
For example, let $A\in \mathsf{SAlg}_{\Bbbk}$. Since $e'(a, v)he'(a, v)^{-1}\in\mathbb{R}(A)$ for every $h\in\mathbb{R}(A), a\in A_1$ and $v\in\mathsf{V}$, then
\[e'(a, v)he'(a, v)^{-1}h^{-1}=e'(a, v)e'(a, -h\cdot v)=e'(a, v-h\cdot v)\in\mathbb{R}(A)\]
implies $v-h\cdot v\in\mathsf{W}$, hence $(2b)$. The converse statement is obvious. 

Suppose that $\mathbb{G}\to\mathbb{H}$ is surjective in the Grothendieck topology. Consider a couple $h\in\mathsf{H}(A_0)=\mathbb{H}(A_0)$ and $e'(a, u)\in\mathbb{H}(A)$. There is an fppf covering $A'$ of $A$, such that $e'(\iota(a), u)=e'(a', \overline{v}),$ where $v\in\mathsf{V}$, and $\overline{v}$ is the image of $v$ in $\mathsf{U}$. That is, $\mathsf{V}\to\mathsf{U}$ is a surjective map. Finally, let $\iota : A_0\to B$ be an fppf covering of $A_0$ , such that $\mathbb{H}(\iota)(h)$ belongs to
the image of $\mathbb{G}_{ev}(B)=\mathsf{G}(B_0)$. Then $B_0$ is an fppf covering of $A_0$.
Conversely, if $h\in\mathsf{H}(A_0)$ belongs to the image of $\mathsf{G}$ up to an fppf covering $B$ of
$A_0$ (in the category of algebras), then $h$ belongs to the image of $\mathbb{G}$ up to an fppf covering
$B\otimes_{A_0} A$ of $A$. The theorem is proven. 
\end{proof}
\begin{cor}\label{faithfully flatness of certain quotient morphisms}
If $\mathbb{R}$ is a normal group super-subscheme of an algebraic group scheme $\mathbb{G}$, then the sheaf quotient morphism
$\mathbb{G}\to \mathbb{G}/\mathbb{R}$ is faithfully flat. 
\end{cor}
\begin{proof}
The superschemes $\mathbb{G}$ and $\mathbb{G}/\mathbb{R}$ can be identified with $\mathsf{G}\times\mathrm{SSp}(\Lambda(\mathsf{V}^*))$ and
$\mathsf{G}/\mathsf{R}\times \mathrm{SSp}(\Lambda((\mathsf{V}/\mathsf{W})^*)$. Then the quotient morphism is induced by the quotient morphism $\mathsf{G}\to \mathsf{G}/\mathsf{R}$, which is faithfully flat by \cite[Theorem 7.35]{milne}, and by the canonical embedding
$\Lambda((\mathsf{V}/\mathsf{W})^*)\to \Lambda(\mathsf{V}^*)$, so that $\Lambda(\mathsf{V}^*)$ is a free $\Lambda((\mathsf{V}/\mathsf{W})^*)$-supermodule. 	
\end{proof} 

\section{Applications}

\subsection{Radicals}

Let $\mathbb{G}$ be an algebraic group superscheme, represented by a Harish-Chandra pair $(\mathsf{G}, \mathsf{V})$, where $\mathsf{V}=\mathfrak{g}_1$, and $\mathsf{G}=\mathbb{G}_{ev}$ is regarded as an algebraic group scheme. In what follows $\mathsf{H}^0$ denotes the connected component of an algebraic group scheme $\mathsf{H}$.

Let $\mathsf{R}$ be a normal group subscheme of $\mathsf{G}$. The largest normal group super-subscheme $\mathbb{H}$ of $\mathbb{G}$ such that $\mathsf{H}=\mathbb{H}_{ev}\leq\mathsf{R}$, is called the $\mathsf{R}$-{\it radical} of $\mathbb{G}$. If the $\mathsf{R}$-radical of $\mathbb{G}$ is trivial, then $\mathbb{G}$ is called \emph{$\mathsf{R}$-semisimple}.

Let $\mathsf{W}$ be a $\mathsf{G}$-submodule of $\mathsf{V}$ such that $[\mathsf{W}, \mathsf{V}]\subseteq \mathrm{Lie}(\mathsf{R})$ and $[\mathsf{W}, \mathsf{V},  \mathsf{V}]=[[\mathsf{W}, \mathsf{V}],  \mathsf{V}]\subseteq\mathsf{W}$. Such a submodule is called \emph{$\mathsf{R}$-subordinated}. The sum of two $\mathsf{R}$-subordinated submodules is again $\mathsf{R}$-subordinated. Thus, there is the largest $\mathsf{R}$-subordinated submodule, denoted by $\mathsf{W}_{\mathsf{R}}$. 

Set $\mathsf{H}_{\mathsf{R}}=\ker (\mathsf{R}\to\mathrm{GL}(\mathsf{V}/\mathsf{W}_{\mathsf{R}}))$. Since $\mathsf{W}_{\mathsf{R}}$ is a $\mathsf{G}$-submodule of $\mathsf{V}$, we have $\mathsf{H}_{\mathsf{R}}\unlhd\mathsf{G}$.
\begin{lm}\label{R-radical}
The Harish-Chandra sub-pair $(\mathsf{H}_{\mathsf{R}}, \mathsf{W}_{\mathsf{R}})$ represents the $\mathsf{R}$-radical of $\mathbb{G}$.
\end{lm}
\begin{proof}
Let $(\mathsf{H}, \mathsf{W})$ be a Harish-Chandra sub-pair of $(\mathsf{G}, \mathsf{V})$. Then  $(\mathsf{H}, \mathsf{W})$ represents a normal group super-subscheme of $\mathbb{G}$ if and only if it satisfies the conditions $(2a)-(2c)$ from Theorem \ref{exact sequences of supergroups}.
Assume that $\mathsf{H}\leq\mathsf{R}$. Since $\mathsf{H}$ acts trivially on $\mathsf{V}/\mathsf{W}$, we have 
\[ [\mathsf{W}, \mathsf{V}, \mathsf{V}]\subseteq [\mathrm{Lie}(\mathsf{H}), \mathsf{V}]\subseteq\mathsf{W},\]
that is, $\mathsf{W}$ is $\mathsf{R}$-subordinated. Thus $\mathsf{W}\subseteq\mathsf{W}_{\mathsf{R}}$ and $\mathsf{H}\leq\mathsf{H}_{\mathsf{R}}$. To complete the proof, it remains to show that $[\mathsf{W}_{\mathsf{R}}, \mathsf{V}]\subseteq\mathrm{Lie}(\mathsf{H_R})$. Recall that 
\[\mathrm{Lie}(\mathsf{H}_{\mathsf{R}})=\{x\in\mathrm{Lie}(\mathsf{R})\mid [x, \mathsf{V}]\subseteq\mathsf{W}_{\mathsf{R}}\},\]
hence $[\mathsf{W}_{\mathsf{R}}, \mathsf{V}]\subseteq\mathrm{Lie}(\mathsf{H_R})$ if and only if $[\mathsf{W}_{\mathsf{R}}, \mathsf{V}, \mathsf{V}]\subseteq\mathsf{W}_{\mathsf{R}}$. The lemma is proven.
\end{proof}
\begin{rem}\label{another definition}
The submodule $\mathsf{W}_{\mathsf{R}}$ can be also defined as
\[ \mathsf{W}_{\mathsf{R}}=\{w\in\mathsf{V}\mid \mbox{for any odd positive integer} \ n, \mbox{ there is } [\ldots [w, \underbrace{\mathsf{V}], \ldots , \mathsf{V}}_{n-\mbox{times}}]\subseteq \mathrm{Lie}(\mathsf{R})\}.\]	
\end{rem}
\begin{cor}\label{affine radical}
Every algebraic group superscheme $\mathbb{G}$ contains the largest connected normal affine group super-subscheme, denoted by $\mathbb{G}_{aff}$. The group super-subschemes $\mathbb{G}_{aff}$ can be characterized by the property that $\mathbb{G}/\mathbb{G}_{aff}$ does not contain non-trivial normal connected affine group super-subschemes.
\end{cor}
\begin{proof}
Let $\mathsf{G}_{aff}$ denote the largest connected normal affine group subscheme of $\mathsf{G}$ (cf. \cite[Proposition 8.1]{milne}). Then the sub-pair $(\mathsf{H}^0_{\mathsf{G}_{aff}}, \mathsf{W}_{\mathsf{G}_{aff}})$ represents $\mathbb{G}_{aff}$, i.e., $\mathbb{G}_{aff}$ is nothing else but the connected component of the $\mathsf{G}_{aff}$-radical of $\mathbb{G}$. 
The second statement is now apparent.
\end{proof}
Let $\mathbb{G}$ be an algebraic group superscheme, represented by a Harish-Chandra pair $(\mathsf{G}, \mathsf{V})$. For every group subscheme $\mathsf{S}$ of $\mathsf{G}$, let $\mathsf{V}_{\mathsf{S}}$ denote the smallest $\mathsf{G}$-submodule of $\mathsf{V}$ such that $\mathsf{S}\leq\ker(\mathsf{G}\to\mathrm{GL}(\mathsf{V}/\mathsf{V}_{\mathsf{S}}))$. 
\begin{lm}\label{R-semisimple}
Let $\mathsf{R}$ be a normal group subscheme of $\mathsf{G}=\mathbb{G}_{ev}$. Then $\mathbb{G}$ is $\mathsf{R}$-semisimple if and only if the following conditions hold:
\begin{enumerate}
\item For any non-trivial normal group subscheme $\mathsf{S}$ of $\mathsf{G}$ such that $\mathsf{S}\leq\mathsf{R}$, there is $[\mathsf{V}, \mathsf{V}_{\mathsf{S}}]\not\subseteq\mathrm{Lie}(\mathsf{S})$;
\item  There is no non-zero $\mathsf{G}$-submodule $\mathsf{W}$ of $\mathsf{V}$ such that $[\mathsf{V}, \mathsf{W}]=0$.
\end{enumerate}
\end{lm}
\begin{proof}
If on the contrary, $[\mathsf{V}, \mathsf{V}_{\mathsf{S}}]\subseteq\mathrm{Lie}(\mathsf{S})$ or there is non-zero $\mathsf{G}$-submodule $\mathsf{W}$ of $\mathsf{V}$ such that $[\mathsf{V}, \mathsf{W}]=0$, then $(\mathsf{S}, \mathsf{V}_{\mathsf{S}})$ or $(\mathsf{E}, W)$ represent a non-trivial normal group super-subscheme of $\mathbb{G}$. Conversely, suppose that $(\mathsf{S}, \mathsf{W})$ represents a normal group super-subscheme of $\mathbb{G}$. Since $[\mathsf{V}, \mathsf{V}_{\mathsf{S}}]\subseteq [\mathsf{V}, \mathsf{W}]\subseteq\mathrm{Lie}(\mathsf{S})$, $(1)$ implies $\mathsf{S}=\mathsf{E}$, and $(2)$ concludes the proof.
\end{proof}
A connected algebraic group superscheme $\mathbb{G}$ is called \emph{pseudoabelian} if $\mathbb{G}_{aff}=\mathbb{E}$. 

Let $\mathbb{G}$ be a pseudoabelian group superscheme and $\mathbb{G}\neq\mathbb{E}$. Then $\mathbb{G}_{ev}=\mathsf{G}$ is not affine. Recall that $\mathsf{G}^{aff}$ denotes the largest affine quotient of $\mathsf{G}$. 

Set $\mathsf{A}(\mathsf{G})=\ker (\mathsf{G}\to \mathsf{G}^{aff})$. Then $\mathsf{A}(\mathsf{G})$ is a nontrivial anti-affine algebraic group, hence smooth and connected normal group subscheme of positive dimension (cf. \cite[Corollary 8.14 and Proposition 8.37]{milne}). Since the natural morphism 
$\mathsf{G}\to \mathrm{GL}(\mathsf{V})$ factors through $\mathsf{G}\to\mathsf{G}^{aff}$, $\mathsf{A}(\mathsf{G})$ acts trivially on $\mathsf{V}$. Besides, by \cite[Corollary 8.13]{milne}, $\mathsf{A}(\mathsf{G})$ is central in $\mathsf{G}$. 
\begin{lm}\label{Even part of peseudoabelian group superscheme}
The group scheme $\mathsf{A}(\mathsf{G})$ is an abelian group variety such that $\mathsf{G}=\mathsf{G}_{aff}\mathsf{A}(\mathsf{G})$.
\end{lm}
\begin{proof}
The sub-pair $(\mathsf{A}(\mathsf{G})_{aff}, 0)$ represents a connected normal affine group super-subscheme of $\mathbb{G}$. Thus, $\mathsf{A}(\mathsf{G})_{aff}=\mathsf{E}$, and by \cite[Theorem 8.28]{milne}, $\mathsf{A}(\mathsf{G})$ is an abelian group variety. The group scheme $\mathsf{G}/\mathsf{G}_{aff}\mathsf{A}(\mathsf{G})$ is a quotient of both $\mathsf{G}/\mathsf{G}_{aff}$ and $\mathsf{G}/\mathsf{A}(\mathsf{G})$. Recall that $\mathsf{G}/\mathsf{G}_{aff}$ is an abelian variety by \cite[Theorem 8.28]{milne}. Therefore,  $\mathsf{G}/\mathsf{G}_{aff}\mathsf{A}(\mathsf{G})$ is a complete and affine connected scheme simultaneously, hence trivial (cf. \cite[A.75(g)]{milne}). 
\end{proof}
The following example shows that the class of pseudoabelian group superschemes is extensive.
\begin{example}\label{examples of pseudoabelian supergroups}
For every connected algebraic group scheme $\mathsf{G}$ such that $\mathsf{A}(\mathsf{G})$ is an abelian variety, and $\mathsf{G}=\mathsf{G}_{aff}\mathsf{A}(\mathsf{G})$, there exists a pseudoabelian group superscheme $\mathbb{G}$ with $\mathbb{G}_{ev}=\mathsf{G}$. 
Let $\mathsf{W}$ be a faithful $\mathsf{G}_{aff}/(\mathsf{G}_{aff}\cap\mathsf{A}(\mathsf{G}))$-module and $\mathsf{V}=\mathsf{W}\oplus\mathsf{W}^*$. Then $\mathsf{V}$ has the natural structure of $\mathsf{G}$-module for the diagonal action of $\mathsf{G}_{aff}$ and the trivial action of $\mathsf{A}(\mathsf{G})$.  Furthermore, we have a bilinear symmetric $\mathsf{G}$-equivariant map 
$\mathsf{V}\times\mathsf{V}\to\mathrm{Lie}(\mathsf{G})$ defined by the rule 
\[[v, w]=[\phi, \psi]=0, [\phi, v]=\phi(v)x \text{ for } v, w\in\mathsf{W} \text{ and } \phi, \psi\in\mathsf{W}^*,\]
where $x\in\mathrm{Lie}(\mathsf{A})\setminus 0$.

Then $(\mathsf{G}, \mathsf{V})$ is a Harish-Chandra pair representing a group superscheme $\mathbb{G}$. By Lemma \ref{R-radical}, the group superscheme $\mathbb{G}$ is pseudoabelian.
\end{example}
The following theorem is a super-version of well-known Barsotti-Chevalley theorem (cf. \cite[Theorem 8.27]{milne}).
\begin{theorem}\label{super Barsotti-Chevalley}
Let $\mathbb{G}$ be a connected algebraic group superscheme. Then there are normal group super-subschemes
$\mathbb{G}_1$ and $\mathbb{G}_2$ of $\mathbb{G}$ such that $\mathbb{G}_1$ is affine and connected,
$\mathbb{G}_2/\mathbb{G}_1$ is a (purely-even) abelian group variety, and $\mathbb{G}/\mathbb{G}_2$ is affine.
\end{theorem}
\begin{proof}
Set $\mathbb{G}_1=\mathbb{G}_{aff}$. Then $\mathbb{H}=\mathbb{G}/\mathbb{G}_1$ is a pseudoabelian group superscheme represented by a Harish-Chandra pair $(\mathsf{H}, \mathsf{V})$. As above,  $(\mathsf{A}(\mathsf{H}), 0)$ represents a normal (purely-even) group subvariety $\mathbb{A}$ of $\mathbb{H}$. Moreover, Lemma \ref{Even part of peseudoabelian group superscheme} implies that $\mathbb{H}/\mathbb{A}$ is an affine group superscheme. 
We define $\mathbb{G}_2$ as $\pi^{-1}(\mathbb{A})$, where $\pi : \mathbb{G}\to\mathbb{G}/\mathbb{G}_1$ is the quotient morphism.  
\end{proof}
\begin{cor}\label{abelian group superscheme}
Let $\mathbb{G}$ be an abelian group supervariety. Then 
\begin{enumerate}
\item $\mathbb{G}_{ev}\unlhd\mathbb{G}$, and $\mathbb{G}_{ev}$ is an abelian group supervariety;
\item $\mathbb{G}_{aff}\cap\mathbb{G}_{ev}=\mathbb{E}$, hence $\mathbb{G}_{aff}$ is a (purely-odd) affine unipotent group superscheme;
\item $\mathbb{G}/(\mathbb{G}_{aff}\times\mathbb{G}_{ev})$ is a (purely-odd) affine unipotent group superscheme.
\end{enumerate}
\end{cor}
\begin{proof}
We have seen that $\mathbb{G}$ is an abelian group supervariety if and only if $\mathsf{G}=\mathbb{G}_{res}$ is an abelian group variety. Thus $\mathsf{G}$ acts trivially on $\mathsf{V}=\mathrm{Lie}(\mathbb{G})_1$ and $(\mathbb{G}_1)_{ev}$ is trivial, which implies the first and second statements. The third statement is  clear.
\end{proof}
\begin{rem}
If $\mathbb{G}$ is an abelian group supervariety, then the group super-subscheme $\mathbb{G}_{aff}$ is represented by the pair $(\mathsf{E}, \mathsf{W})$, where $\mathsf{W}=\{w\in\mathsf{V}\mid [w, \mathsf{V}]=0\}$. Respectively, $\mathbb{G}_{ev}\times\mathbb{G}_{aff}$ is represented by the pair $(\mathsf{G}, \mathsf{W})$.
\end{rem}

\subsection{Anti-affine group superschemes}

An algebraic group superscheme $\mathbb{G}$ is called \emph{anti-affine} whenever $\mathcal{O}(G)$ is a Grassman-like superalgebra. Set $\Phi(\mathbb{G})=(\mathsf{G}, \mathsf{V})$.
Theorem \ref{coincidence of G and G'} implies that $\mathbb{G}$ is anti-affine if and only if
the algebraic group scheme $\mathsf{G}$ is (cf. \cite{milne}, page 39).  Moreover, $(\mathsf{A}(\mathsf{G}), 0)$ represents a central anti-affine group super-subscheme $\mathbb{A}$ of $\mathbb{G}$ such that $\mathbb{G}/\mathbb{A}\simeq\mathbb{G}^{aff}$. But an anti-affine group super-subscheme of $\mathbb{G}$ is not necessarily central in $\mathbb{G}$, contrary to the purely-even case (cf. \cite[Corollary 8.13]{milne}). For example, if $\mathbb{G}$ is a psedoabelian group superscheme from Example \ref{examples of pseudoabelian supergroups}, then its group super-subscheme, represented by the pair $(\mathsf{A}(\mathsf{G}), \mathsf{V})$, is anti-affine but not central. 

Similarly, an anti-affine group superscheme is no longer commutative, but it is always nilpotent.
If $\mathbb{G}$ is anti-affine, then $\mathbb{G}_{ev}$ is central, and $\mathbb{G}/\mathbb{G}_{ev}$ is a purely-odd unipotent group superscheme.

\section{Sheaf quotients are superschemes}

Let $\mathbb{G}$ be an algebraic group superscheme, and $\mathbb{H}$ its group super-subscheme.
They are represented by Harish-Chandra pairs $(\mathsf{G}, \mathsf{V})$ and $(\mathsf{H}, \mathsf{W})$, respectively.

For more details about sheaves, sheaf quotient  and sheafifications (completions) in the Grothendieck topology of fppf coverings, we refer to
\cite{jan, zub2, zub3}. In this section, we prove the following theorem.
\begin{theorem}\label{sheaf quotient}
The sheaf quotient $\mathbb{G}/\mathbb{H}$ is a superscheme of finite type and the quotient morphism $\mathbb{G}\to \mathbb{G}/\mathbb{H}$ is faithfully flat.
\end{theorem} 
The proof will be given in the series of lemmas.

\subsection{A reduction}

As above, let $\mathsf{A}(\mathsf{G})$ denote $\ker (\mathsf{G}\to\mathsf{G}^{aff})$. 

Let us start with two elementary observations. 
First, suppose $\mathsf{L}$ is a group subscheme of $\mathsf{A}(\mathsf{G})$. In that case, the pair $(\mathsf{L}, 0)$ represents
a central group super-subscheme $\mathbb{L}$ of $\mathbb{G}$. Besides, the group superscheme $\mathbb{G}/\mathbb{L}$ is represented by the pair $(\mathsf{G}/\mathsf{L}, \mathsf{V})$ such that $\ker(\mathsf{G}/\mathsf{L}\to (\mathsf{G}/\mathsf{L})^{aff})=\mathsf{A}(\mathsf{G})/\mathsf{L}$, i.e. $(\mathsf{G}/\mathsf{L})^{aff}$ is naturally isomorphic to $\mathsf{G}^{aff}$.

Second, if $\mathsf{R}$ is an affine normal group subscheme of $\mathsf{G}$, then $(\mathsf{G}/\mathsf{R})_{aff}=\mathsf{G}_{aff}\mathsf{R}/\mathsf{R}$. Every normal affine connected group subscheme of $\mathsf{G}/\mathsf{R}$ has the form $\mathsf{L}/\mathsf{R}$, where $\mathsf{L}$ is a normal affine group subscheme of $\mathsf{G}$ (use \cite[Proposition 8.1]{milne}). By \cite[Proposition 1.52]{milne}, $\mathsf{L}^0$ is normal, hence $\mathsf{L}^0\leq\mathsf{G}_{aff}$. Since 
$(\mathsf{L}/\mathsf{R})/(\mathsf{L}^0\mathsf{R}/\mathsf{R})\simeq \mathsf{L}/\mathsf{L}^0\mathsf{R}$
is etale and $\mathsf{L}^0\mathsf{R}/\mathsf{R}$ is connected, we have $\mathsf{L}/\mathsf{R}=\mathsf{L}^0\mathsf{R}/\mathsf{R}\leq\mathsf{G}_{aff}\mathsf{R}/\mathsf{R}$.
\begin{lm}\label{reduction to affine 1}
If Theorem \ref{sheaf quotient} holds for affine group super-subschemes $\mathbb{H}$, then it holds for arbitrary $\mathbb{H}$.	
\end{lm}
\begin{proof}
By the first observation, $(\mathsf{H}\cap\mathsf{A}(\mathsf{G}), 0)$ represents a central group super-subscheme $\mathbb{L}$ of $\mathbb{G}$ such that $\mathbb{L}\leq\mathbb{H}$.  By Proposition 4.2 from \cite{zub3}, we have a commutative diagram
\[\begin{array}{ccc}
\mathbb{G} & \to & \mathbb{G}/\mathbb{L} \\
\downarrow & & \downarrow \\
\mathbb{G}/\mathbb{H} & \simeq & (\mathbb{G}/\mathbb{L})/(\mathbb{H}/\mathbb{L})
\end{array},\]
where $\mathsf{H}/\mathsf{L} \cap \mathsf{A}(\mathsf{G})/\mathsf{L}=\mathsf{E}$, hence $\mathsf{H}/\mathsf{L}$ is affine and, in its turn, 
$\mathbb{H}/\mathbb{L}$ is affine as well. It remains to note that $\mathbb{G} \to \mathbb{G}/\mathbb{L}$ and $\mathbb{G}/\mathbb{L}\to (\mathbb{G}/\mathbb{L})/(\mathbb{H}/\mathbb{L})$ are  faithfully flat morphisms of superschemes of finite type by Corollary \ref{faithfully flatness of certain quotient morphisms} and 
by the assumption of lemma, respectively. 
\end{proof}
From now on we assume that $\mathbb{H}$ is affine. If it is the case, then $\mathbb{G}\to \mathbb{G}/\mathbb{H}$ is always affine and faithfully flat, provided $\mathbb{G}/\mathbb{H}$ is a superscheme (use \cite[Proposition 9.8]{zub2} and mimic the proof of  \cite[I.5.7(1)]{jan}).
\begin{lm}\label{reduction to affine 2}
Let $\mathbb{H}_1$ and $\mathbb{H}_2$ are (affine) group super-subschemes of $\mathbb{G}$ such that $\mathbb{H}_1\leq \mathbb{H}_2$.
If Theorem \ref{sheaf quotient} holds for the couple $\mathbb{H}_2\leq \mathbb{G}$ and $\mathbb{H}_2/\mathbb{H}_1$ is an affine superscheme of finite type, then it holds for the couple $\mathbb{H}_1\leq \mathbb{G}$. 	
\end{lm} 
\begin{proof}
We have a cartesian square (see \cite[Remark 9.11]{zub2}) :
\[\begin{array}{ccc}
\mathbb{G}\times \mathbb{H}_2/\mathbb{H}_1 & \to & \mathbb{G}/\mathbb{H}_1 \\
\downarrow & & \downarrow \\
\mathbb{G} & \to & \mathbb{G}/\mathbb{H}_2 
\end{array},  \]
which satisfies the conditions of Lemma \ref{finite type and flatness} (use \cite[Proposition 9.8 and Remark 9.11]{zub2}). Therefore, $\mathbb{G}/\mathbb{H}_1$ is a superscheme of finite type. 	
\end{proof}
\begin{lm}\label{reduction to affine}
Theorem \ref{sheaf quotient} holds for any couple $\mathbb{H}\leq\mathbb{G}$, whenever it holds for each couple that satisfies :
\begin{enumerate} 
\item $\mathsf{G}=\mathsf{M}\times\mathsf{A}(\mathsf{G})$; 
\item $\mathsf{M}$ is an affine group subscheme of $\mathsf{G}$ with $\mathsf{H}\leq\mathsf{M}$.
\end{enumerate}
\end{lm}
\begin{proof}
Let $\mathsf{R}$ denote the group scheme $\mathsf{G}_{aff}\mathsf{H}\cap\mathsf{A}(\mathsf{G})$. Set $\mathbb{R}=\Psi(\mathsf{R}, 0)$. The second observation implies that $(\mathsf{G}/\mathsf{R})_{aff}=\mathsf{G}_{aff}\mathsf{R}/\mathsf{R}$ and $(\mathsf{G}_{aff}\mathsf{H})/\mathsf{R}\cap\mathsf{A}(\mathsf{G})/\mathsf{R}=\mathsf{E}$. 
Since $\mathbb{H}\mathbb{R}/\mathbb{H}\simeq\mathbb{R}$ is affine and of finite type, $\mathbb{G}/\mathbb{H}$ is a superscheme of finite type provided $\mathbb{G}/\mathbb{H}\mathbb{R}\simeq (\mathbb{G}/\mathbb{R})/(\mathbb{H}\mathbb{R}/\mathbb{R})$ is. 
Therefore, one can assume that $\mathsf{G}_{aff}\mathsf{H}\cap\mathsf{A}(\mathsf{G})=\mathsf{E}$. 
Arguing as in Lemma \ref{Even part of peseudoabelian group superscheme}, one can show that $\mathsf{A}(\mathsf{G})$ is an abelian variety and $\mathsf{G}^0=\mathsf{G}_{aff}\mathsf{A}(\mathsf{G})$. In particular, $\mathsf{A}(\mathsf{G})$ can be identified with $(\mathsf{G}/\mathsf{G}_{aff}\mathsf{H})^0$. By \cite[Theorem 1.1]{brion}, there is a finite, hence affine, group subscheme $\mathsf{F}$ of $\mathsf{G}/\mathsf{G}_{aff}\mathsf{H}$ such that
$\mathsf{F}\mathsf{A}(\mathsf{G})=\mathsf{G}/\mathsf{G}_{aff}\mathsf{H}$. The inverse image of $\mathsf{F}$ in $\mathsf{G}$ is an affine group subscheme $\mathsf{M}$ such that $\mathsf{G}=\mathsf{M}\mathsf{A}(\mathsf{G})$. Repeating the above arguments with $\mathsf{R}=\mathsf{M}\cap\mathsf{A}(\mathsf{G})$, one can reduce the general case to $\mathsf{M}\cap\mathsf{A}(\mathsf{G})=\mathsf{E}$. The proposition is proven.
\end{proof}

\subsection{Sheaf quotient}

Let $\mathbb{G}$ be an algebraic group superscheme, represented by a Harish-Chandra pair $(\mathsf{G}, \mathsf{V})$. Let $\mathbb{H}$ be its closed (affine) group super-subscheme, represented by a sub-pair
$(\mathsf{H}, \mathsf{W})$. 
Without loss of generality, one can assume that $\mathbb{G}$ and $\mathbb{H}$ satisfy the conditions of Lemma \ref{reduction to affine}. 

Since $\mathsf{X}=\mathsf{G}/\mathsf{H}$ is a scheme (of finite type), any open affine subscheme $\mathsf{U}'$ of $\mathsf{G}/\mathsf{H}$ has a form $\mathsf{U}/\mathsf{H}$, where $\mathsf{U}$ is an open affine $\mathsf{H}$-saturated subscheme of $\mathsf{G}$
(cf. \cite[I.5.7.1(1)]{jan}). 

Set $\mathrm{Sp}(\mathsf{A})=\mathsf{U}$ and $\mathrm{Sp}(\mathsf{D})=\mathsf{H}$. Then $\mathsf{A}'=\mathcal{O}(\mathsf{U}')\simeq
\mathsf{A}^{\mathsf{D}}=\mathsf{A}\square_{\mathsf{D}} \Bbbk$.

Recall that the superschemes $\mathbb{G}$ and $\mathbb{H}$ can be represented as  $\mathsf{G}{\bf E}$ and $\mathsf{H}{\bf E}'$, respectively.
\begin{lm}\label{a certain open super-subscheme}
The subfunctor $\mathbb{U}=\mathsf{U}{\bf E}={\bf E}\mathsf{U}$ is an open affine $\mathbb{H}$-saturated super-subscheme of $\mathbb{G}$.   
\end{lm}
\begin{proof}
It is clear that $\mathsf{U}{\bf E}$ is open in $\mathbb{G}$. Further, we have $\mathbb{U}=\mathsf{U}{\bf E}=\mathsf{U}\mathsf{N}{\bf E}=\mathsf{U}\mathbb{N}$, where $\mathbb{N}=\mathbb{N}(\mathbb{G})$ and $\mathsf{N}=\mathbb{N}_{ev}$. Since $\mathbb{N}$ is a normal group subfunctor, we also have $\mathbb{U}=\mathbb{N}\mathsf{U}={\bf E}\mathsf{N}\mathsf{U}=
{\bf E}\mathsf{U}$. Finally, for every $hx\in\mathbb{H}$, where $h\in\mathsf{H}$ and $x\in {\bf E}'$, there is  
\[\mathbb{U}(hx)=\mathsf{U}\mathbb{N}(hx)=(\mathsf{U}h) (\mathbb{N}^{h^{-1}} x)\subseteq\mathsf{U}\mathbb{N}=\mathbb{U}.\]
The lemma is proven.
\end{proof}
Set $\mathrm{SSp}(A)=\mathbb{U}$ and $\mathrm{SSp}(D)=\mathbb{H}$.  
Let $\tau : A\to A\otimes D$ be the corresponding $D$-coideal superalgebra map.

Recall that
$\mathsf{G}=\mathsf{M}\times\mathsf{A}(\mathsf{G})$, where $\mathsf{H}\leq\mathsf{M}$ and $\mathsf{M}$ is affine. Thus $\mathsf{U}$ can be chosen so that $\mathsf{U}=\mathsf{U}_{aff}\times\mathsf{U}_{ab}$, where $\mathsf{U}_{aff}$ is an open affine subscheme of $\mathsf{M}$ and $\mathsf{U}_{ab}$ is an open affine subscheme of $\mathsf{A}(\mathsf{G})$.
Moreover, the algebra $\mathsf{A}$ is isomorphic to $\mathsf{A}_{aff}\otimes \mathsf{A}_{ab}$ as a $\mathsf{D}$-comodule, where $\mathsf{A}_{aff}\simeq \Bbbk [\mathsf{U}_{aff}], \mathsf{A}_{ab}\simeq \Bbbk [\mathsf{U}_{ab}]$, and $\mathsf{A}_{ab}$ is a trivial $\mathsf{D}$-comodule.

Let $V$ be a right (super)comodule over a (super)coalgebra $C$. Let $C\to B$ be a morphism of (super)coalgebras. Then $V$ and $C$ has the natural structures of right and left (super)comodules over $B$. Moreover, the (super)comodule map $\sigma : V\to V\otimes C$ factors through $
V\square_B C\subseteq V\otimes C$ and we denote $V\to V\square_B C$ by $\sigma_B$. 
 
Following notations from \cite{mastak}, one can define a morphism of $D$-coideal superalgebras
\[\omega_{\theta} : A\stackrel{\tau_{\mathsf{D}}}{\to} (\wedge(\mathsf{V^*})\otimes \mathsf{A})\square_{\mathsf{D}} D\stackrel{\wedge(\theta')\square_{\mathsf{D}}\mathrm{id}_{D}}{\to} (\mathsf{A}\otimes\wedge(\mathsf{Z}))\square_{\mathsf{D}}D,\]
where $\mathsf{Z}=\ker(\mathsf{V^*}\to\mathsf{W^*})$. The proofs of Lemma 4.1, Proposition 4.2, and Lemma 4.4 in \cite{mastak} can be repeated verbatim. Only the definition of $\theta'$ and  the proof of Lemma 4.6 in \cite{mastak} needs a commentary. The structure of the right $\mathsf{G}$-module of
$\mathsf{V}^*$ is completely defined by the structure of $\mathsf{M}$-module because $\mathsf{A}(\mathsf{G})$ acts trivially on $\mathsf{V}^*$. Therefore, one can define the map 
\[\kappa=\kappa_{\mathsf{A}} : \mathsf{A}\otimes\mathsf{V}^*\to \mathsf{V}^*\otimes \mathsf{A}\]
by
\[a_1\otimes a_2\otimes v^*\mapsto v^*_{(0)}\otimes a_1 j(v^*_{(1)})\otimes a_2,\]
where $a_1\in \mathsf{A}_{aff}, a_2\in \mathsf{A}_{ab}$, $v^*\mapsto v_{(0)}\otimes v^*_{(1)}$ is the corresponding comodule map $\mathsf{V}^*\to \mathsf{V}^*\otimes \Bbbk [\mathsf{M}]$, and $j : \Bbbk [\mathsf{M}]\to \mathsf{A}_{aff}$ is dual to the open immersion $\mathsf{U}_{aff}\to \mathsf{M}$.

We will show that $\omega_{\theta}$ is an isomorphism. 
\begin{rem}\label{base change and iso} 
The definition of $\omega_{\theta}$ is consistent with a base change, and once needed, $\Bbbk$ can be replaced by a suitable field extension $L\supseteq \Bbbk$.
\end{rem}
It has been proven that $\mathsf{gr}(\mathbb{G})\simeq \mathsf{G}\ltimes \mathbb{G}_{odd}$ and $\mathsf{gr}(\mathbb{H})\simeq \mathsf{H}\ltimes \mathbb{H}_{odd}$, where $\mathbb{G}_{odd}\simeq\mathrm{SSp}(\Lambda(\mathsf{V}^*))$, $\mathbb{H}_{odd}\simeq\mathrm{SSp}(\Lambda(\mathsf{W}^*))$
and $\mathsf{V}^*$ and $\mathsf{W}^*$ consist of primitive elements.

Using Proposition \ref{graded superscheme} and arguing as in Proposition \ref{gr and group objects}, one sees that $\mathsf{gr}(\mathbb{U})\simeq\mathrm{SSp}(\mathsf{gr}(A))$ is an open $\mathsf{gr}(\mathbb{H})$-saturated super-subscheme of $\mathsf{gr}(\mathbb{G})$. 
\begin{pr}\label{grading of the first map} The following statements hold:
\begin{enumerate}
\item $\mathsf{gr}(\mathbb{U})\simeq\mathsf{U}\mathbb{G}_{odd}=\mathbb{G}_{odd}\mathsf{U}$.
\item As the right and left $\mathsf{D}$-coideal superalgebras, $\mathsf{gr}(A)$ and $\mathsf{gr}(D)$ are naturally isomorphic to $A$ and $D$, respectively. 
\item $\mathsf{gr}(\tau_{\mathsf{D}})=\mathsf{gr}(\tau)_{\mathsf{D}}$. 
\end{enumerate}
\end{pr}
\begin{proof}
The geometric counterpart of $\mathsf{U}\mathbb{G}_{odd}$ is the open super-subscheme $q_G^{-1}(U_{ev})\simeq U_{ev}\times G_{odd}$, where $U$ is the geometric counterpart of $\mathbb{U}$ in $G$.  Since $\mathsf{gr}(U)^e=U^e=U_{ev}^e=(q_G^{-1}(U_{ev}))^e$, the condition $(1)$ follows.

Further, the right $\mathsf{gr}(D)$-coideal superalgebra structure on $\mathsf{gr}(A)$ is defined by a superalgebra morphism
$\mathsf{gr}(A)\stackrel{\mathsf{gr}(\tau)}{\to} \mathsf{gr}(A\otimes D)\simeq\mathsf{gr}(A)\otimes \mathsf{gr}(D)$. 
More precisely, for an element $a\in I_{A}^n$, let $\tau(a)=\sum_{0\leq k\leq n} a_{(0), k}\otimes a_{(1), n-k}$ be adapted to the $I_{A\otimes D}$-adic filtration. That is,
$a_{(0), k}\in I^k_{A}$ and $a_{(1), n-k}\in I^{n-k}_{D}$ for every
$0\leq k\leq n$.
Then $\mathsf{gr}(\tau)$ is defined by
\[a+I_{A}^{n+1}\mapsto\sum_{0\leq k\leq n} (a_{(0), k}+I_{A}^{k+1})\otimes (a_{(1), n-k}+I_{D}^{n-k+1}).\] 
It follows that the right $\mathsf{D}$-coideal superalgebra structure of $\mathsf{gr}(A)$ is defined as
\[a+I_{A}^{n+1}\mapsto (a_{(0), n}+I_{A}^{n+1})\otimes (a_{(1), 0}+I_{D}),\]
where $a\mapsto a_{(0), n}\otimes (a_{(1), 0}+I_{D})$
is the right $\mathsf{D}$-coideal superalgebra map of $A$. 

On the other hand, if we identify
$A$ with $\Lambda(\mathsf{V}^*)\otimes \mathsf{A}$, then the latter map is $x\otimes a\mapsto x\otimes a_{(0)}\otimes a_{(1)}$, where $x\in\Lambda(\mathsf{V}^*), a\in \mathsf{A}$ and $a\mapsto a_{(0)}\otimes a_{(1)}$ is the $\mathsf{D}$-coideal algebra map of $\mathsf{A}$. Moreover, we have $I_{A}^k=(\oplus_{s\geq k}\Lambda^s(\mathsf{V}^*))\otimes \mathsf{A}$ for  $0\leq k\leq\dim\mathsf{V}$. Then $A$ and $\mathsf{gr}(A)$ are isomorphic as $\mathsf{D}$-coideal superalgebras. The case of $D$ is similar.

Furthermore, $A\square_{\mathsf{D}} D$ is isomorphic to $\Lambda(\mathsf{V}^*)\otimes \mathsf{B}\otimes \Lambda(\mathsf{W}^*)$, where 
\[\mathsf{B}=\{a_{(0)}\otimes a_{(1)}\mid a\in \mathsf{A}\}\subseteq \mathsf{A}\otimes \mathsf{D} .\]
In particular, $I^k_{A\square_{\mathsf{D}} D}=A\square_{\mathsf{D}} D\cap I^k_{A\otimes D}$ for $0\leq k\leq \dim\mathsf{V}+\dim\mathsf{W}$, and we obtain a natural isomorphism
$\mathsf{gr}(A\square_{\mathsf{D}} D)\simeq \mathsf{gr}(A)\square_{\mathsf{D}}\mathsf{gr}(D)$ that makes the diagram
\[\begin{array}{ccccc}
\mathsf{gr}(A) & \stackrel{\mathsf{gr}(\tau_{\mathsf{D}})}{\to} & \mathsf{gr}(A\square_{\mathsf{D}} D) & \to & \mathsf{gr}(A\otimes D)  \\
\parallel & & \downarrow & & \downarrow  \\
\mathsf{gr}(A) & \stackrel{\mathsf{gr}(\tau)_{\mathsf{D}}}{\to} & \mathsf{gr}(A)\square_{\mathsf{D}}\mathsf{gr}(D) & \to & \mathsf{gr}(A)\otimes\mathsf{gr}(D)
\end{array}\]
commutative. The proposition is proven. 
\end{proof}
\begin{lm}\label{grading the last morphism}
The morphism $\mathsf{gr}(\wedge(\theta')\square_{\mathsf{D}}\mathrm{id}_{D})$ is identified with
$\wedge(\theta')\square_{\mathsf{D}}\mathrm{id}_{\mathsf{gr}(D)}$.
\end{lm}
\begin{proof}
Recall that the morphism $\wedge(\theta')$ is defined as
\[v^*_1\wedge\ldots\wedge v^*_k\otimes a\mapsto \sum_{i_1, \ldots, i_k} a_{i_1}\ldots a_{i_k}a\otimes z_{i_1}\wedge\ldots\wedge z_{i_k},\]
where $\theta'(v^*_s)=\sum_{i_s} a_{i_s}\otimes z_{i_s}$ for $1\leq s\leq k$.
Moreover, the morphism $\wedge(\theta')\square_{\mathsf{D}}\mathrm{id}_D$ is defined as
\[v^*_1\wedge\ldots\wedge v^*_k\otimes a_{(0)}\otimes a_{(1)}\otimes\wedge w^*_1\wedge\ldots\wedge w^*_l\mapsto\]
\[\sum_{i_1, \ldots, i_k} a_{i_1}\ldots a_{i_k}a_{(0)}\otimes z_{i_1}\wedge\ldots\wedge z_{i_k}\otimes a_{(1)}\otimes\wedge w^*_1\wedge\ldots\wedge w^*_l,\]
and is  compatible with the corresponding filtrations of the superalgebras $(\wedge(\mathsf{V}^*)\otimes \mathsf{A})\square_{\mathsf{D}} D$ and $(\mathsf{A}\otimes\wedge(\mathsf{Z}))\square_{\mathsf{D}} D$. 
\end{proof}
The following lemma is now obvious.
\begin{lm}\label{gr of omega}
The morphism $\mathsf{gr}(\omega_{\theta})$ can be naturally identified with
\[\mathsf{gr}(A)\stackrel{\mathsf{gr}(\tau)_{\mathsf{D}}}{\to} (\wedge(\mathsf{V}^*)\otimes \mathsf{A})\square_{\mathsf{D}}\mathsf{gr}(D)\stackrel{\wedge(\theta')\square_{\mathsf{D}}\mathrm{id}_{\mathsf{gr}(D)}}{\to} (\mathsf{A}\otimes\wedge(\mathsf{Z}))\square_{\mathsf{D}}\mathsf{gr}(D).\]
\end{lm} 
Let $\theta$ be a retract of the inclusion $A\otimes\mathsf{Z}\to A\otimes\mathsf{V}^*$ from \cite[Lemma 4.4]{mastak}. Let $\mathsf{V}^*=\mathsf{Z}\oplus\mathsf{T}$, and elements $t_j$ form a basis of
$\mathsf{T}$ for $1\leq j\leq s=\dim\mathsf{W}^*$. The free $A$-module $\ker\theta$ has a basis consisting of the elements $t'_j$ such that $t'_j = t_j\pmod{A\otimes\mathsf{Z}}$. 

Choose a closed point $x$ in $\mathsf{U}$. By Remark \ref{base change and iso}, one can assume that
$x\in\mathsf{U}(\Bbbk)$. That is,  $x$ is an algebra morphism $\mathsf{A}\to \Bbbk$. 

The natural embedding $\mathbb{H}_{odd}\to \mathsf{U}\times\mathbb{G}_{odd}\simeq\mathsf{gr}(\mathbb{U})$, defined as $h\mapsto (x, h)$, is dual to the superalgebra morphism induced by $x$ and $\mathsf{V}^*\to\mathsf{W}^*$. The natural (right-hand side) action of group $\mathbb{H}_{odd}$ on $\mathsf{gr}(\mathbb{U})$ is defined by a coideal superalgebra map
\[z\mapsto z\otimes 1 \text{ and } t'_j\mapsto t'_j\otimes 1+1\otimes \overline{t_j} \text{ for } z\in\mathsf{Z},\]
where $\overline{t_j}$ is the image of $t_j$ in $\mathsf{W}^*$.
Finally, the $\mathsf{A}$-superalgebra morphism $\wedge(\theta): \mathsf{A}\otimes\Lambda(\mathsf{V}^*)\to \mathsf{A}\otimes\Lambda(\mathsf{Z})$ induces the embedding ${\bf s} :\mathsf{U}\times\mathrm{SSp}(\Lambda(\mathsf{Z}))\to\mathsf{gr}(\mathbb{U})$.
\begin{lm}\label{product of two embeddings}
There is a natural isomorphism
\[(\mathsf{U}\times\mathrm{SSp}(\Lambda(\mathsf{Z})))\times\mathbb{H}_{odd}\simeq \mathsf{gr}(\mathbb{U}),\]
induced by the above two embeddings and the multiplication map.
\end{lm}
\begin{proof}
Choose a basis $z_1, \ldots, z_{t-s}$ of $\mathsf{Z}$. We use notations from Proposition \ref{a canonical form in hyp(G)}. Any couple $(g, h)\in (\mathsf{U}\times\mathrm{SSp}(\Lambda(\mathsf{Z})))\times\mathbb{H}_{odd}$ is taken to a superalgebra morphism
\[(g, h)(a z_I t'_J)=\sum_{S\subseteq J} g(\wedge(\theta)(az_I t'_S))h(\overline{t}_{J\setminus S})=
g(az_I)h(\overline{t}_J),\]
where $I\subseteq \{ 1, \ldots, t-s \} , J\subseteq \{1, \ldots , s  \}$.  
The proof concludes by the observation that $A\otimes\Lambda(\mathsf{V}^*)\simeq (A\otimes\Lambda(\mathsf{Z}))\otimes\Lambda(\mathsf{T}')$, where $\mathsf{T}'$ is the $\Bbbk$-span of $t'_j$ for $1\leq j\leq s$.
\end{proof}
\begin{lm}\label{this is an isomorphism}
The superalgebra morphism $\mathsf{gr}(\omega_{\theta})$ is an isomorphism, and therefore $\omega_{\theta}$ is as well.
\end{lm}
\begin{proof}
Consider the following sequence of superscheme morphisms
\[(\star) \ \mathsf{gr}(\mathbb{U})\stackrel{\bf a}{\leftarrow} \mathsf{gr}(\mathbb{U})\times^{\mathsf{H}}\mathsf{gr}(\mathbb{H})\stackrel{\bf b}{\leftarrow} (\mathsf{U}\times\mathbb{G}_{odd})\times^{\mathsf{H}}\mathsf{gr}(\mathbb{H})\stackrel{\bf d}{\leftarrow} (\mathsf{U}\times\mathrm{SSp}(\Lambda(\mathsf{Z})))\times^{\mathsf{H}}\mathsf{gr}(\mathbb{H})\simeq\]\[\mathrm{SSp}((\mathsf{A}\otimes\wedge(\mathsf{Z}))\square_{\mathsf{D}}\mathsf{gr}(D)).\]
The morphism  $\bf a$ is induced by $(u, h)\mapsto uh$ for $u\in\mathsf{gr}(\mathbb{U})$ and $h\in\mathsf{gr}(\mathbb{H})$.  
Since $\mathsf{gr}(\mathbb{H})\simeq\mathsf{H}\times\mathbb{H}_{odd}$, the superscheme $\mathsf{gr}(\mathbb{U})\times^{\mathsf{H}}\mathsf{gr}(\mathbb{H})$ is canonically isomorphic to
$\mathsf{gr}(\mathbb{U})\times\mathbb{H}_{odd}$, hence it is affine. Lemma \ref{if it is affine} implies
that $\mathsf{gr}(\mathbb{U})\times^{\mathsf{H}}\mathsf{gr}(\mathbb{H})\simeq\mathrm{SSp}(\mathsf{gr}(A)\square_{\mathsf{D}}\mathsf{gr}(D))$ and the morphism $\bf a$ is dual to $\mathsf{gr}(\tau)_{\mathsf{D}}$.

Similarly, we have 
\[(\mathsf{U}\times\mathbb{G}_{odd})\times^{\mathsf{H}}\mathsf{gr}(\mathbb{H})\simeq (\mathsf{U}\times\mathbb{G}_{odd})\times \mathbb{H}_{odd}\simeq\mathrm{SSp}((\Lambda(\mathsf{V}^*)\otimes \mathsf{A})\square_{\mathsf{D}}\mathsf{gr}(D)),\]
and the isomorphism $\bf b$ is dual to the superalgebra isomorphism induced by $\kappa^{-1}$, 
\[(\mathsf{U}\times\mathrm{SSp}(\Lambda(\mathsf{Z})))\times^{\mathsf{H}}\mathsf{gr}(\mathbb{H})\simeq(\mathbb{G}_{odd}\times\mathsf{U})\times\mathbb{H}_{odd}\simeq 
\mathrm{SSp}((\mathsf{A}\otimes\Lambda(\mathsf{Z}))\square_{\mathsf{D}}\mathsf{gr}(D))\]
and the morphism ${\bf d}={\bf s}\times^{\mathsf{H}}\mathrm{id}_{\mathsf{gr}(\mathbb{H})}$ is just $\mathrm{SSp}(\wedge(\theta)\square\mathrm{id}_{\mathsf{gr}(D)})$. Thus the composition
${\bf a}{\bf b}{\bf d}$ coincides with $\mathrm{SSp}(\mathsf{gr}(\omega_{\theta}))$.
The above isomorphisms  allow us to identify ${\bf a}{\bf b}{\bf d}$ with the isomorphism from Lemma \ref{product of two embeddings}. Thus, the lemma follows.
\end{proof}
\begin{lm}\label{U/H is affine}
The sheaf quotient $\mathbb{U}/\mathbb{H}$ is an affine superscheme of finite type.
\end{lm}
\begin{proof}
By \cite[Theorem 3.1 (1) and Corollary 4.14(2)]{mastak}, we need to show that the superalgebra morphism
$\alpha: A\otimes A\to A\otimes D$ given by $a\otimes a'\mapsto \sum aa'_{(0)}\otimes a'_{(1)}$, is surjective. Equivalently, the functor morphism $\mathrm{SSp}(\alpha): \mathbb{U}\times\mathbb{H}\to\mathbb{U}\times\mathbb{U}$ given by $(u, h)\mapsto (u, uh),$ maps $\mathbb{U}\times\mathbb{H}$ isomorphically onto a closed super-subscheme of $\mathbb{U}\times\mathbb{U}$. Define a functor morphism ${\bf p} : \mathbb{G}\times\mathbb{G}\to \mathbb{G}$ by $(g, g')\mapsto g^{-1}g'$. Then $\mathrm{SSp}(\alpha)(\mathbb{U}\times\mathbb{H})=(\mathbb{U}\times\mathbb{U})\cap {\bf p}^{-1}(\mathbb{H})$, hence it is closed in $\mathbb{U}\times\mathbb{U}$. The lemma follows. 
\end{proof}
Let $\mathsf{V}'$ be an open affine subscheme of $\mathsf{U}'$. As above, $\mathsf{V}'=\mathsf{V}/\mathsf{H}$, where $\mathsf{V}\simeq\mathrm{Sp}(\mathsf{B})$ is an affine subscheme of $\mathsf{U}$ and $\mathsf{B}'=\mathcal{O}(\mathsf{V}')\simeq \mathsf{B}^{\mathsf{D}}$. Set $\mathbb{V}=\mathsf{V}{\bf E}$. Then $\mathbb{V}$ is an open affine $\mathbb{H}$-saturated super-subscheme of $\mathbb{U}$.
\begin{lm}\label{open affine in U/H}
$\mathbb{V}/\mathbb{H}$ is an open affine super-subscheme of $\mathbb{U}/\mathbb{H}$.
\end{lm}
\begin{proof}
The inclusion $\mathsf{V}'\subseteq\mathsf{U}'$ is defined by a morphism $\phi : \mathsf{A}\to \mathsf{B}$ of $\mathsf{D}$-coideal algebras. By \cite[Lemma 3.5]{maszub1}, there are (even) elements $a_1, \ldots, a_k\in \mathsf{A}'$ such that 
$\sum_{1\leq i\leq k}\mathsf{B}'\phi(a_i)=\mathsf{B}'$, and for each $i$, the induced morphism $\mathsf{A}'_{a_i}\to \mathsf{B}'_{\phi(a_i)}$ is an isomorphism. Thus, $\mathsf{B}=\sum_{1\leq i\leq k}\mathsf{B}\phi(a_i)$ and $\mathsf{A}_{a_i}\simeq \mathsf{B}_{\phi(a_i)}$ for $1\leq i\leq k$. There is $B=\mathcal{O}(\mathbb{V})\simeq (\mathsf{B}\otimes\Lambda(\mathsf{Z}))\square_{\mathsf{D}} D$, and the morphism of $D$-coideal superalgebras $A\to B$, dual to the inclusion $\mathbb{V}\subseteq\mathbb{U}$, is $(\phi\otimes\mathrm{id}_{\Lambda(\mathsf{Z})})\square_{\mathsf{D}}\mathrm{id}_{D}$. Moreover, by \cite[Proposition 4.2(1)]{mastak}, $\mathbb{V}/\mathbb{H}\simeq\mathrm{SSp}((\mathsf{B}\otimes\Lambda(\mathsf{Z}))^{\mathsf{D}})$ and $\mathbb{U}/\mathbb{H}\simeq\mathrm{SSp}((\mathsf{A}\otimes\Lambda(\mathsf{Z}))^{\mathsf{D}})$. Therefore, the inclusion $\mathbb{V}/\mathbb{H}\subseteq \mathbb{U}/\mathbb{H}$
is dual to $\phi\otimes\mathrm{id}_{\Lambda(\mathsf{Z})}$, and the superalgebras $\mathcal{O}(\mathbb{U}/\mathbb{H})$ and $\mathcal{O}(\mathbb{V}/\mathbb{H})$ satisfy the conditions of \cite[Lemma 3.5]{maszub1}, for the elements $a_i\otimes 1$, where $1\leq i\leq k$. The lemma is proven.
\end{proof}
\begin{cor}
The statement of Lemma \ref{open affine in U/H} remains valid for an arbitrary open subscheme $\mathsf{V}'$ of $\mathsf{U}'$.
\end{cor}
\begin{proof}
Indeed, $\mathsf{V}'$ is a finite union of open affine subschemes.
\end{proof}
Fix a finite open covering of $\mathsf{G}/\mathsf{H}$ by affine subschemes $\mathsf{U}_i'$ for $1\leq i\leq l$. Then there is a collection of affine superschemes $\mathbb{U}_i'=\mathbb{U}_i/\mathbb{H}$ and quotient morphisms $\mathbb{U}_i\to\mathbb{U}_i'$ such that for every pair  of indices $1\leq i\neq j\leq l$ we have $(\mathbb{U}_i\cap\mathbb{U}_j)/\mathbb{H}=\mathbb{U}_i'\cap\mathbb{U}_j'$. 

Consider the collection of corresponding geometric superschemes $U_i'$. 
Since the underlying topological space of each $U'_i$ coincides with $(\mathsf{U}_i')^e$ (see \cite[Proposition 4.2]{mastak}), one can construct a geometric superscheme $Z$ with
$Z^e=\mathsf{X}^e$ as follows. For any open subset $W\subseteq Z^e$, we define 
\[\mathcal{O}_Z(W)=\ker(\prod_{1\leq i\leq l}\mathcal{O}_{U'_i}(W\cap (U_i')^e)\rightrightarrows\prod_{1\leq i\neq j\leq l}\mathcal{O}_{U_i'\cap U_j'}(W\cap(U'_i)^e\cap (U'_j)^e)),\]
where $U'_i\cap U'_j$ denotes the geometric counterpart of $\mathbb{U}_i'\cap\mathbb{U}_j'$.
It is clear that $\mathcal{O}_Z$ is a superalgebra sheaf such that $\mathcal{O}_Z|_{(U'_i)^e}\simeq \mathcal{O}_{U'_i}$ 
for $1\leq i\leq l$. 

The following lemma concludes the proof of the main theorem.
\begin{lm}\label{gluing}
There is the unique morphism $\mathbb{G}\to\mathbb{Z}$, the restriction of which on each $\mathbb{U}_i$ coincides with $\mathbb{U}_i\to\mathbb{U}_i'\subseteq\mathbb{Z}$. In particular, $\mathbb{Z}\simeq\mathbb{G}/\mathbb{H}$.
\end{lm}
\begin{proof}
The open immersions $U'_i\to Z$ induce the open embeddings $\mathbb{U}'_i\to\mathbb{Z}$. Thus, $\mathbb{U}'_i$ form an open covering of $\mathbb{Z}$.  Since $\mathbb{Z}$ is a local functor, 
the collection of morphisms $\mathbb{U}_i\to\mathbb{U}_i'$ uniquely extends to a morphism $\mathbb{G}\to\mathbb{Z}$, which is constant on $\mathbb{H}$-orbits. Moreover, any morphism from $\mathbb{G}$ to a faisceau $\mathbb{Y}$, which is constant on $\mathbb{H}$-orbits , is uniquely defined by its restrictions on $\mathbb{U}_i$ for $1\leq i\leq l$.
Since such morphisms $\mathbb{U}_i\to\mathbb{Y}$ factor through the morphisms $\mathbb{U}_i'\to\mathbb{Y}$, $\mathbb{G}\to\mathbb{Y}$ factors through the unique morphism $\mathbb{Z}\to\mathbb{Y}$. Lemma is proven.
\end{proof}
\begin{center}
\bf Acknowledgements	
\end{center}
The work was supported by UAEU grant G00003324. The first author was also supported by grant  JSPS KAKENHI Grant Number 20K03552.
We thank for both supports.

\end{document}